\documentclass[pdflatex,sn-basic,Numbered]{sn-jnl}

\usepackage{graphicx}%
\usepackage{multirow}%
\usepackage{amsmath,amssymb,amsfonts}%
\usepackage{amsthm}%
\usepackage{mathrsfs}%
\usepackage[title]{appendix}%
\usepackage{xcolor}%
\usepackage{textcomp}%
\usepackage{manyfoot}%
\usepackage{booktabs}%
\usepackage{algorithm}%
\usepackage{algorithmicx}%
\usepackage{algpseudocode}%
\usepackage{listings}%

\usepackage{physics, accents, tabularx }
\usepackage{tabularx}
\graphicspath{ {./images/} }
\usepackage{pgfplots,wrapfig}
\pgfplotsset{compat=1.18} 
\numberwithin{equation}{section}

\theoremstyle{plain}%
\newtheorem{theorem}{Theorem}[section]
\newtheorem{corollary}[theorem]{Corollary}
\newtheorem{lemma}[theorem]{Lemma}

\theoremstyle{definition}%
\newtheorem{rem}[theorem]{Remark}
\newtheorem{definition}{Definition}%

\newtheorem{innercustomthm}{Assumptions}
\newenvironment{custom_assumption}[1]
  {\renewcommand\theinnercustomthm{#1}\innercustomthm}
  {\endinnercustomthm}

\newcommand{\N}{\mathbb{N}}

\newcommand{\E}{\mathbb{E}}
\newcommand{\R}{\mathbb{R}}
\newcommand{\Prob}{\mathbf{P}}

\DeclareMathAlphabet{\mathbbo}{U}{bbold}{m}{n}
\newcommand{\1}{\mathbbo{1}}

\newcommand{\Int}{\mathsf{Int}}
\newcommand{\Cl}{\mathsf{Cl}}

\newcommand{\OM}{\mathrm{M}} 
\newcommand{\LP}{\mathrm{P}} 
\newcommand{\TP}{\mathrm{T}} 
\newcommand{\x}{\mathbf{x}} 

\newcommand{\X}{\boldsymbol{X}} 
\newcommand{\bmu}{\boldsymbol{\mu}} 

\newcommand{\olsi}[1]{\,\overline{\!{#1}}}
\newcommand{\rhobar}{\olsi{\rho}} 
\newcommand{\epsilonbar}{\olsi{\varepsilon}}
\newcommand{\thetabar}{\olsi{\vartheta}}

\begin{document}

\title[Exit-problem for a class of non-Markov processes with path dependency]{Exit-problem for a class of non-Markov processes with path dependency}


\author*[1]{\fnm{Ashot} \sur{Aleksian}}\email{ashot.aleksian@tse-fr.eu}

\author[2]{\fnm{Aline} \sur{Kurtzmann}}\email{aline.kurtzmann@univ-lorraine.fr}
\equalcont{These authors contributed equally to this work.}

\author[3]{\fnm{Julian} \sur{Tugaut}}\email{julian.tugaut@univ-st-etienne.fr}
\equalcont{These authors contributed equally to this work.}

\affil[1]{ \orgname{Toulouse School of Economics}, \orgaddress{\city{Toulouse}, \postcode{31000}, \country{France}}} 

\affil[2]{\orgname{Universit\'e de Lorraine}, \orgdiv{CNRS, Institut Elie Cartan de Lorraine}, \city{Nancy}, \postcode{F-54000}, \country{France}}

\affil[3]{\orgname{Universit\'e Jean Monnet}, \orgdiv{CNRS, Institut Camille Jordan}, \orgaddress{\city{Saint-\'Etienne}, \postcode{42023}, \country{France}}}


\abstract{
We study the exit-time of a self-interacting diffusion from an open domain $G \subset \R^d$. In particular, we consider the equation $$ \dd{X_t} = - \left( \nabla V(X_t)  + \frac{1}{t}\int_0^t\nabla F (X_t - X_s)\dd{s} \right) \dd{t}  + \sigma \dd{W_t}.$$ We are interested in the small-noise ($\sigma \to 0$) behaviour of the exit-time from the potentials' domain of attraction. In this work rather weak assumptions on the potentials $V$ and $F$, and on the domain $G$ are considered. In particular, we do not assume $V$ nor $F$ to be either convex or concave, which covers a wide range of self-attracting and self-repelling stochastic processes possibly moving in a complex multi-well landscape. The Large Deviation Principle for the Self-interacting diffusion with generalized initial conditions is established. The main result of the paper states that, under some assumptions on the potentials $V$ and $F$, and on the domain $G$, the Kramers' type law for the exit-time holds. Finally, we provide a result concerning the exit-location of the diffusion.
}




\maketitle

\section{Introduction}\label{s:Introduction}

In this paper, we consider a stochastic process $(X_t^\sigma, \; t\!\geq\!0)$ living in $\R^d$ defined by a stochastic differential equation that includes an interaction of the process with its own passed trajectory in the drift term. Consider the following \textit{Self-Interacting Diffusion} (SID):
\begin{equation}\label{eq:def:main_sys_no_mu}
    \dd{X^\sigma_t} = - \nabla V(X^\sigma_t)\dd{t}  - \frac{1}{t} \int_0^t \nabla F  (X^\sigma_t - X^\sigma_s) \dd{s} \dd{t}  + \sigma \dd{W_t}.
\end{equation}
$V: \R^d \to \R$ is called \textit{the confinement potential} and represents the general geometry of the space. $F: \R^d \to \R$ is called \textit{the interaction potential} and represents the interaction forces arising between the process and its past. These forces depend only on the distance between the current position and the previous positions of the process and, moreover, we suppose that $\nabla F(0) = 0$ and, for simplicity, $F(0) = 0$. Lastly, $(W_t,\; t\!\geq\!0)$ denotes the standard $d$-dimensional Brownian motion and $\sigma \geq 0$ is a parameter controlling the noise of the system.

Let us now introduce the following group of main assumptions: 

\begin{custom_assumption}{A-1}\label{A-1}
    \hfill
    \begin{enumerate}
        \item \label{A-1.Cont} (regularity) Potentials $V$ and $F$ belong to the space $C^2(\R^d; \R)$.
        \item \label{A-1.Lip} (Lipschitz continuity) There exist $\text{Lip}_{\nabla V}, \text{Lip}_{\nabla F} < \infty$ such that $|\nabla V(x) - \nabla V(y)| \leq \text{Lip}_{\nabla V} |x - y|$ and $|\nabla F(x) - \nabla F(y)| \leq \text{Lip}_{\nabla F} |x - y|$ for any $x,\; y \in \R^d$.
    \end{enumerate}
\end{custom_assumption}

Note that the interaction term can be rewritten using \textit{the empirical measure} $\mu^\sigma_t = \frac{1}{t}\int_0^t \delta_{X_s^\sigma}\dd{s}$, where $\delta_x$ is the Dirac measure on $\R^d$ concentrated at the point $x$. In this work, we also fix the initial condition $x_0 \in \R^d$. That gives us the following representation of the equation \eqref{eq:def:main_sys_no_mu}, which will define the process $X^\sigma$ that we study here:
\begin{equation}\label{eq:SID_main_sys}
    \begin{cases}
    \dd{X^\sigma_t} \!\!\!\!&= - \nabla V(X^\sigma_t)\dd{t}  - \nabla F * \mu^\sigma_t (X^\sigma_t) \dd{t}  + \sigma \dd{W_t}, \\
    \mu^\sigma_t &= \frac{1}{t} \int_0^t \delta_{X^\sigma_s} \dd{s},\\
    X^\sigma_0 &= x_0 \in \R^d \; \text{a.s.}
    \end{cases}
\end{equation}
We emphasize that Assumption~\ref{A-1} implies the global existence and uniqueness of the strong solution to \eqref{eq:SID_main_sys} for any $\sigma \geq 0$. Further details can be found in Section~\ref{s:remarks_and_init_cond}.

In this paper, we consider the exit-time from an open neighbourhood $G \subset \R^d$ in the small-noise regime (i.e. when $\sigma \to 0$). We also suppose the existence of $a := \lim_{t \to \infty} X^0_t$, where $X^0$ is the process \eqref{eq:SID_main_sys} defined for $\sigma = 0$. Let us define {\it the effective potential} $U_a := V + F(\cdot - a)$. Consider the following assumptions on $G$, $a$ and $x_0$:

\begin{custom_assumption}{A-2}\label{A-2}
    \hfill
    \begin{enumerate}
        \item \label{A-2.G_open} (domain $G$) $G \subset \R^d$ is an open, bounded, connected set such that $\partial G = \partial \overline{G}$. The boundary $\partial G$ is a smooth $(d-1)$-dimensional hypersurface. The points $x_0, a \in G$.
        \item \label{A-2.Effect_pot_conv} (stability of $\overline{G}$ under the effective potential) Let $\phi$ be defined as $\phi_t^x = x - \int_0^t \nabla U_a(\phi_s^x) \dd{s}$. For any $x \in \overline{G}$, $\{\phi^x_t\}_{t > 0} \subset G$ and $\phi^x_t \xrightarrow[t\to \infty]{} a$. Moreover, $\nabla V(a) = 0$.
        \item \label{A-2.Strong_attraction_a} (strong attraction around $a$) There exist $\Delta_\mu, \Delta_x > 0$ small enough and constants $0 < K_2 < K_1<\infty$, such that for any $\mu \in  \mathcal{P}_2(\R^d)$ satisfying $\mathbb{W}_2(\mu;\delta_a) \leq \Delta_\mu$ and for any $x \in B_{\Delta_x}(a)$, we have $$\langle \nabla V(x) + \nabla F*\mu(x); x - a \rangle \geq K_1|x - a|^2-K_2|x-a|\mathbb{W}_2\left(\mu;\delta_a\right).$$
    \end{enumerate}
\end{custom_assumption}

Assumption \ref{A-2}.\ref{A-2.G_open} gives some regularity properties that the domain $G$ has to satisfy. Since we consider the small-noise behaviour of the process $X^\sigma$, we expect it to be close to its deterministic counterpart $X^0$, at least for some finite time. In particular, we expect that, with high probability, $X^\sigma$ should first approach the attractor $a$, before leaving the domain $G$. In this case, the drift term of the equation \eqref{eq:SID_main_sys}, $-\nabla V - \nabla F*\mu_t^\sigma$, should be close, at least for some finite time, to $-\nabla V - \nabla F*\delta_a$, which is the negative gradient of the effective potential $U_a$. Assumption \ref{A-2}.\ref{A-2.Effect_pot_conv} guarantees that, first, $a$ is also a stable equilibrium for the gradient of the effective potential $U_a$ and, second, that $G$ is a subset of the basin of attraction of $a$ for the flow generated by $- \nabla U_a$ (see Figure~\ref{fig:potentials_V_U_F}).

Assumption~\ref{A-2}.\ref{A-2.Strong_attraction_a} ensures that the attraction inside a small neighbourhood around $a$ is sufficiently strong. Verifying this assumption for general functions $V$ and $F$ can be challenging. However, we can identify a necessary condition:
\begin{equation}
\label{eq:aux:necessary_condition}
    \nabla^2 V(a) + 2\nabla^2 F(0) \succeq C \, {\rm Id},
\end{equation}
where $C > 0$ is a positive constant and ${\rm Id}$ is the identity matrix. To establish this, let us apply the conditions of Assumption~\ref{A-2}.\ref{A-2.Strong_attraction_a} to $\mu = \delta_{2a - x}$. This yields:
\begin{equation*}
    \int_0^1 \big\langle \big( \nabla^2 V(a + s(x - a)) + 2 \nabla^2 F(2s(x - a)) \big) (x - a); x - a \big\rangle \dd{s} \geq (K_1 - K_2)|x - a|^2,
\end{equation*}
for any $x \in B_{\Delta_x}(a)$. Consequently,
\begin{equation*}
    \sup_{s \in [0, 1]} \left\langle \Big( \nabla^2 V(a + s(x - a)) + 2 \nabla^2 F(2s(x - a)) \Big) \frac{x - a}{|x - a|}; \frac{x - a}{|x - a|} \right\rangle
\end{equation*}
is positive. Therefore, by the continuity of $\nabla^2 V$ and $\nabla^2 F$, we obtain \eqref{eq:aux:necessary_condition}.

For the case of a quadratic interaction potential, $F(x) = \alpha |x|^2/2$ with $\alpha \in \R\setminus\{0\}$, we can additionally derive the sufficient condition:
\begin{equation}
\label{eq:aux:suficient_condition}
    \nabla^2 V(a) \succeq (C - 2 \alpha \wedge 0) \, {\rm Id},
\end{equation}
where $C > 0$. Here, Assumption~\ref{A-2}.\ref{A-2.Strong_attraction_a} is satisfied because:
\begin{equation*}
\begin{aligned}
    \left\langle \nabla V(x) + \alpha \big(x - \int z \mu(\dd{z}) \big); x - a \right\rangle &= \langle \nabla V(x); x - a \rangle + \alpha |x - a|^2  \\
    & \quad - \alpha \left\langle \int (z - a) \mu(\dd{z}); x - a \right\rangle.
\end{aligned}
\end{equation*}
If $\alpha > 0$, this simply gives us the following lower bound in a small neighbourhood of $a$:
\begin{equation*}
\begin{aligned}
    \left\langle \nabla V(x) + \alpha \big(x - \int z \mu(\dd{z}) \big); x - a \right\rangle &\geq \left(\frac{C}{2} + \alpha\right)|x - a|^2 - \alpha \mathbb{W}_2\left(\mu;\delta_a\right) |x - a|,
\end{aligned}
\end{equation*}
which leads to Assumption~\ref{A-2}.\ref{A-2.Strong_attraction_a} by taking $K_1 = (\frac{C}{2} + \alpha)$ and $K_2 = \alpha$. 

In the case $\alpha < 0$, note that for any $r \in (0 , 1)$, by continuity of $\nabla^2 V(a)$, we can find a small enough radius $\Delta_x$ such that for any $x \in B_{\Delta_x}(a)$, we have:
\begin{equation*}
    \left\langle \nabla V(x) + \alpha \big(x - \int z \mu(\dd{z}) \big); x - a \right\rangle \geq \left(r(C - 2\alpha) + \alpha\right)|x - a|^2 + \alpha \mathbb{W}_2\left(\mu;\delta_a\right) |x - a|.
\end{equation*}
In particular, $r$ can be chosen to be equal to $\frac{C - 4\alpha}{2(C - 2 \alpha)} < 1$. That gives us $K_1 = \frac{C}{2} - \alpha$ and $K_2 = -\alpha$. Therefore, under the sufficient condition \eqref{eq:aux:suficient_condition}, Assumption~\ref{A-2}.\ref{A-2.Strong_attraction_a} holds in a small neighbourhood of $a$ for any $\mu \in \mathcal{P}_2(\mathbb{R}^d)$.

Under the assumptions \ref{A-1} and \ref{A-2}, we establish {\it the Kramers' type law} for the exit-time of the Self-interacting diffusion \eqref{eq:SID_main_sys} from an open bounded domain of attraction $G \subset \R^d$.
The following theorem is the main result of the paper. 

\begin{figure}[b]
\begin{minipage}{0.6\columnwidth}
    \centering
\begin{tikzpicture}[scale=1.4]
  \draw[->] (-2, 0) -- (2.4, 0) node[below] {$x$};
  \draw[->] (0, -2) -- (0, 2) node[left] {$ $};
  \draw[scale=1.5, domain=0:1.2, smooth, variable=\x, blue] plot ({\x}, {\x*\x*\x*\x + 4*\x*\x*\x/3 - 2*\x*\x});
  \draw[scale=1.5, domain=0:1.2, smooth, variable=\x, teal] plot ({\x}, {\x*\x*\x*\x + 4*\x*\x*\x/3 - 2*\x*\x - (\x - 0.618)*(\x - 0.618)});

  \draw[scale=0.5, domain=-2.45:0, smooth, variable=\x, blue] plot ({\x}, {3*\x*\x*\x*\x/4 + 3*\x*\x*\x/3 - 3*\x*\x/2 });
  \draw[scale=1.5, domain=-0.4:0, smooth, variable=\x, teal] plot ({\x}, {\x*\x*\x*\x + 4*\x*\x*\x/3 - 2*\x*\x - (\x - 0.618)*(\x - 0.618)});

  \draw (1.5, 1.9) node{$V$};
  \draw (1.95, 1.25) node{$U_a$};
  
  \draw[-, red, thick] (0.42, 0) -- (1.7, 0);
  \draw[dashed] (0.42, -2) -- (0.32, 2);
  \draw[dashed] (1.7, -2) -- (1.7, 2);
  \fill (0.95,0)  circle[radius=1pt];
  \draw (0.95,0.15) node {$a$};
  \draw[->] (0.99, -1) -- (1.5, -0.05) ;
  \draw (0.8, -1.1) node{$G$};
\end{tikzpicture}
\end{minipage}%
\begin{minipage}{0.4\columnwidth}
    \centering
\begin{tikzpicture}[scale=1.4]
  \draw[->] (-1.7, 0) -- (1.7, 0) node[below] {$x$};
  \draw[->] (0, -2) -- (0, 2) node[left] {$F$};
  \draw[scale=1, domain=-1.25:1.25, smooth, variable=\x, blue] plot ({\x}, {- \x*\x });
  \draw (-0.1,-0.15) node {$0$};
\end{tikzpicture}
\end{minipage}
\caption{Examples of possible $V$, $F$, and $G$ in dimension $d = 1$, with $U_a$ being the effective potential. 
}
\label{fig:potentials_V_U_F}

\end{figure}

\begin{theorem}\label{th:main_th}
    Let Assumptions \ref{A-1} and \ref{A-2} be fulfilled. Let the process $X^\sigma$ be the unique strong solution of the system \eqref{eq:SID_main_sys}. Let $\tau_G^\sigma := \inf \{t: X_t^\sigma \notin G\}$ denote the first time when $X^\sigma$ exits the domain $G$. Let $H:= \inf_{x \in \partial G} \{U_a(x) - U_a(a)\}$ be the height of the effective potential. Then, the following two results hold:
    \begin{enumerate}
        \item Kramers' type law: for any $\delta > 0$ we have
    \begin{equation*}
        \lim_{\sigma \to 0}\Prob_{x_0}\left( e^{\frac{2 (H - \delta)}{\sigma^2}} < \tau_G^\sigma < e^{\frac{2 (H + \delta)}{\sigma^2}}\right) = 1;
    \end{equation*}
    \item Exit-location: for any closed set $N \subset \partial G$ such that $\inf_{z \in N} \{U_a(x) - U_a(a)\} > H$ the following limit holds:
    \begin{equation*}
        \lim_{\sigma \to 0} \Prob_{x_0} (X^\sigma_{\tau_G^\sigma} \in N) = 0.
    \end{equation*}

    \end{enumerate}
\end{theorem}

\subsection{State of the art}

Similar systems with path-interaction behaviour have already been studied by numerous researchers for more than 30 years. In most of the papers, the long-time behaviour of the process is considered. One of the first mathematical descriptions of such a process under the name of Self-avoiding random walk is presented by J.R.\ Norris, L.C.G.\ Rogers, and D.\ Williams in \cite{NRW87}. Some stochastic properties of the process along with some long-time behaviour results in a particular setting were shown. The main difference between the system considered there and ours is that in \cite{NRW87} there is no renormalization of the interaction term with time. In \cite{DR92} R.T.\ Durrett and L.C.G.\ Rogers introduce a system that is similar to the one presented in \cite{NRW87} and aims to model the shape of a growing polymer:
\begin{equation*}
    \dd{X_t} =\left( \int_0^t f(X_t - X_s)\dd{s}\right) \dd{t} + \dd{W_t}.
\end{equation*}
Given the physical interpretation of the process, the authors called this model ``the Brownian Polymers''. In this paper, asymptotic bounds in dimension $d = 1$ with some assumptions on the interaction function were presented, along with some conjectures on more precise long-time behaviour. Among those, the authors conjectured that, in the symmetrical, \textit{repulsive} case with limited interaction ($f(-x) = -f(x)$, $x f(x) \geq 0$, and $f$ has compact support), one can show that $X_t/t \xrightarrow{} 0$ a.s. This repulsive case was later studied in the paper \cite{TTV12}, where the conjecture was partially proved (with an additional assumption on smoothness of $f$).

Another branch of models with path-interaction, but this time via the renorma\-lized-with-time occupation measure $\mu_t = \frac{1}{t} \int_0^t \delta_{X_s}\dd{s}$ living on a compact manifold, started to develop with a paper of M. Benaïm, M. Ledoux, and O. Raimond \cite{BLR02}. This work was continued in \cite{BR03}, \cite{BR05}, and others. In the first paper, the authors showed that the asymptotic behaviour of the empirical measure of the process $\mu_t$ can be linked to some deterministic dynamical flow. In the following paper \cite{BR03}, the authors used these techniques to describe some convergence in law properties of the process. Later, in \cite{BR05}, assuming the interaction to be symmetric, the authors gave an almost sure convergence result for the occupation measure $\mu_t$ and described its limit points.

The first step towards studying this process in the non-compact setting, which was $\R^d$, was done by A. Kurtzmann in \cite{kurtzmann2010ode}. The author considers the following model with an interaction that depends on the empirical measure of the process.
\begin{equation*}
     \dd{X_t} = -\left(\nabla V(X_t) + \frac{1}{t} \int_0^t \nabla_x F(X_t,X_s)\dd{s}\right)\dd{t} + \dd{W_t},
\end{equation*}
where $F$ is regular and $V$ is convex at infinity. In this paper, the ergodic properties of $X$ were studied, as well as certain conditions on $V$ and $F$ that guarantee almost sure convergence of $\mu_t$. This work generalizes the previous ones of Benaïm and co-authors.  

In the paper of V. Kleptsyn and A. Kurtzmann \cite{KK12}, the dynamic of the form \eqref{eq:SID_main_sys} with the constant $\sigma = \sqrt{2}$ is considered. The authors proved that, in the case of a convex confinement $V$ with $\lim_{|x| \to +\infty} V(x) = +\infty$, an attractive and symmetrical interaction (the potential $F$ is uniformly convex and spherically symmetrical, i.e. $F(x) = F(|x|)$), and some other regularity assumptions, the occupation measure $\mu_t$ converges (in *-weakly sense) almost surely to a density $\rho_\infty$. It means that there exists a unique density $\rho_\infty: \R^d \to \R_+$ such that
\begin{equation*}
    \mu_t = \frac{1}{t}\int_0^t \delta_{X_s} \dd{s} \xrightarrow[t \to \infty]{*-\text{weakly}} \rho_\infty(x)\dd{x} \quad \text{a.s.}
\end{equation*}
Moreover, $\rho_\infty$ is a function satisfying the following functional equation:
\begin{equation*}
    \rho_\infty = \Pi(\rho_\infty) := \frac{\exp{-V - F*\rho_\infty}}{\int_{\R^d} \exp{-V(y) - F*\rho_\infty (y)}\dd{y}}.
\end{equation*}

That brings us to the second part of the paper, the exit-time problem. The main goal of this work is to study first the exit-time of $X^\sigma$, given by \eqref{eq:SID_main_sys}, from some bounded regular domain $G \subset \R^d$, i.e., the following stopping time $\tau_G^\sigma:= \inf\{t \geq 0: X_t^\sigma \notin G\}$. In particular, we are interested in exits driven by the Brownian motion with a small noise $\sigma > 0$. A similar problem for It\^{o} diffusion was introduced and solved by M.\ Freidlin and A.\ Wentzell. The techniques that they used in a number of papers, starting with \cite{VF70} and presented in their book \cite{FW98}, are known under the name of the Freidlin-Wentzell theory (see also \cite[\S 5.6]{DZ10}). Consider an It\^{o} diffusion in the gradient form:
\begin{equation*}
    \dd{X_t} = - \nabla V(X_t)\dd{t} + \sigma \dd{W_t}.
\end{equation*}
Then, the main result of the Freidlin-Wentzell theory is that, given a domain $G \subset \R^d$ with only one attraction point of the gradient field $-\nabla V$, named $a \in G$, and some other regularity assumptions on $G$ and $V$, the so-called Kramers’ type law for the exit-time $\tau_G^\sigma$ can be established. Namely, if we denote $H := \inf_{z \in \partial G}\{ V(z) - V(a)\}$, then for any starting point $x_0 \in G$ and for any $\delta > 0$
\begin{equation}\label{eq:Kramer}
    \lim_{\sigma \to 0} \Prob_{x_0} (e^{\frac{2(H - \delta)}{\sigma^2}} < \tau^\sigma_G < e^{\frac{2(H + \delta)}{\sigma^2}}) = 1.
\end{equation}

It means that, with decreasing $\sigma > 0$, the exit-time from a stable (or positively invariant) by $-\nabla V$ domain $G$ grows exponentially with a rate that depends on the height $H$ of the potential $V$ inside the domain $G$. 

A similar question was posed by S.\ Herrmann, P.\ Imkeller, and D.\ Peithmann for McKean-Vlasov type diffusion in \cite{HIP08}. In particular, they considered a so-called \textit{Self-stabilizing diffusion} (SSD) of the form
\begin{equation}\label{eq:SSD}
    \begin{cases}
    \dd{X^\sigma_t} \!\!\!\!&= - \nabla V(X^\sigma_t)\dd{t}  - \nabla F * \nu^\sigma_t (X^\sigma_t) \dd{t}  + \sigma \dd{W_t}, \\
    \nu^\sigma_t &= \mathcal{L}(X_t^\sigma),\\
    X^\sigma_0 &= x_0 \in \R^d \; \text{a.s.},
    \end{cases}
\end{equation}
where $\mathcal{L}(X^\sigma_t)$ denotes the law of random variable $X_t^\sigma$. Apart from the obvious parallels that can be drawn between equations \eqref{eq:SID_main_sys} and \eqref{eq:SSD} that define SID and SSD respectively, the following fact makes the self-stabilizing case interesting for the current paper. As was shown in \cite{CGM08}, in the case of convex confinement and interaction and under some other regularity assumptions, the law $\nu_t^\sigma$ converges to the unique probability measure that is the solution of $\nu = \Pi(\nu)$, which is the same behaviour that was established in \cite{KK12}.

Following the Freidlin-Wentzell theory techniques, in \cite{HIP08} the authors first establish a large deviation principle for SSD with general assumptions on $V$ and $F$. After that, they had to restrict themselves to the case of convex confinement and convex interaction potentials in order to achieve, under some stability and regularity assumptions on $G$, the Kramers' type law \eqref{eq:Kramer}. Assuming $F(0) = 0$, the rate $H$ in this case have the form $H = \inf_{z \in \partial G}\{ U_a(z) - U_a(a)\}$. 

In \cite{T12} and \cite{T16}, J. Tugaut studies the exit-time problem for SSD in a convex landscape with a convex interaction, as it was presented in \cite{HIP08}. In these papers he establishes the Kramers' type law, avoiding using the large deviation principle. Instead, the author uses various analytical and coupling methods to deal with the problem, thus simplifying calculations. This work was continued in \cite{T18} and \cite{T19} where the same techniques were used to establish the Kramers' type law in the case of the confinement potential $V$ that is of the double-well form. However, the exit-time for the SSD with general assumptions (as in the Freidlin-Wentzell theory for It\^{o} diffusions) is still an open problem.

The exit-time problem was also studied for the case of the self-interacting diffusion. In \cite{ADMKT22}, A.~Aleksian, P.~Del~Moral, A.~Kurtzmann, and J.~Tugaut prove Kramers' type law for SID in which both the interaction and the confinement potentials $V$ and $F$ are convex. This nice property of the potentials was used by the authors in order to prove this exit-time result by applying analytical and coupling techniques similar to the ones used in \cite{T16}. Since the goal of the current paper is to establish some exit-time result for SID with more general, than in \cite{ADMKT22}, assumptions, we have to use a different approach. In our case, that consists in proving the Large Deviation Principle and restoring the logic of the Freidlin-Wentzell theory for SID.

\subsection{Some remarks on dynamics and initial condition}\label{s:remarks_and_init_cond}

Let us introduce some notation and notions that will be useful to describe laconically later derivations in this paper. 

The process $X^\sigma$ itself is non-Markov, yet the triple $(t, \mu^\sigma_t, X^\sigma_t)$ is. In order to see this, we will describe the notion of the previous path using three parameters. $t_0 \geq 0$ will represent its time-length, $\mu_0 \in \mathcal{P}_2(\R^d)$ will represent the occupation measure of the past trajectory and $x_0 \in \R^d$ its end point starting from which we continue the dynamic. We contract initial conditions into a vector $\mathbf{x} = (t_0, \mu_0, x_0) \in \mathfrak{X} = [0, \infty] \! \times \! \mathcal{P}_2(\R^d) \!\times\! \R^d$, and, in order to be able to operate with the three components of $\mathbf{x}$ separately, we introduce the following projection mappings. Let $\TP: \mathfrak{X} \to [0, \infty]$ be the projection on the first coordinate, returning initial time $t_0$. For the initial empirical measure, let $\OM: \mathfrak{X} \to \mathcal{P}_2(\R^d)$ be the mapping that returns the second coordinate of $\mathbf{x} \in \mathfrak{X}$. Finally, for the starting point of the process, respectively, define the mapping $\LP: \mathfrak{X} \to \R^d$ that returns the third coordinate of $\mathbf{x} \in \mathfrak{X}$.

Following the classical notation for diffusions, we introduce the following system of equations:
\begin{equation}\label{eq:SID_main_sys_gen}
    \begin{cases}
        \dd{\X^{\x, \sigma}_t} \!\!\!\!&= - \nabla V(\X^{\x,  \sigma}_t)\dd{t}  - \nabla F * \bmu^{\x, \sigma}_t (\X^{\x, \sigma}_t) \dd{t}  + \sigma \dd{W_t}, \\
        \bmu^{\x, \sigma}_t &= \frac{\TP{\mathbf{x}}}{\TP{\mathbf{x}} + t}\OM{\mathbf{x}} + \frac{1}{\TP{\mathbf{x}} + t} \int_0^t \delta_{\X^{\x, \sigma}_s}\dd{s},\\
        \X^{\x, \sigma}_0 &= \LP{\mathbf{x}}  \; \text{a.s.}
    \end{cases}
\end{equation}

Hereafter, we will drop the initial conditions from the superscript if it is clear which $\x \in \mathfrak{X}$ is meant. Though, to avoid confusion, we will always use the superscript $\sigma$, which can be equal to $0$ in the deterministic case. For $\x$ such that $\TP{\x} = \infty$ we naturally extend the definition of the processes \eqref{eq:SID_main_sys_gen} as
\begin{equation}\label{eq:SID_main_sys_infty_sigma}
    \begin{cases}
        \dd{\X^{\x, \sigma}_t} \!\!\!\!&= - \nabla V(\X^{\x, \sigma}_t)\dd{t}  - \nabla F * \OM{\x} (\X^{\x, \sigma}_t) \dd{t}  + \sigma \dd{W_t}, \\
        \X^{\x, \sigma}_0 &= \LP{\mathbf{x}}  \; \text{a.s.}
    \end{cases}
\end{equation}

Despite that the introduction of $\x$ aims to describe a notion of a ``previous path'' of the process, the set $\mathfrak{X}$ is defined to be more general. Firstly, we allow $t_0$ to be equal to infinity. Second, $\mu_0$ belongs to $\mathcal{P}_2(\R^d)$, which, in general, does not restrict it to be an empirical measure of any path.

Note that the topology naturally defined on $[0, \infty]$ by open intervals in usual sense and intervals of the form $[0, x)$, $(x, \infty]$ is metrizable. Let us denote $d_{[0, \infty]}$ some metric on this space. We also equip the set $\mathcal{P}_2(\R^d)$ with Wasserstein-2 metric (see e.g. \cite[Definition 6.1]{villani2008optimal}). Thus, the Cartesian product $\mathfrak{X}$ is a complete separable metric space with the following metric $$d_{\mathfrak{X}} (\x_1, \x_2) = \max(d_{[0, \infty]} (\TP{\x_1}, \TP{\x_2}); \mathbb{W}_2(\OM{\x_1}; \OM{\x_2});|\LP{\x_1} - \LP{\x_2}|),$$ where $\mathbb{W}_2$ is the Wasserstein distance (see references in \S\ref{ss:LDP-SID}).

Existence and uniqueness result for stochastic differential equations that depend on their path is standard for the case of Lipschitz continuous drift term (see e.g. \cite[Theorem 11.2]{RW00}). It is worth noting that introducing $\x \in \mathfrak{X}$ does not complicate matters in any way, as long as $\nabla V$ and $\nabla F$ remain Lipschitz continuous. Assumption \ref{A-1} ensures that we are working in this framework. Using the standard argument for linear diffusion processes, it can be shown that the stochastic process $t \mapsto (\TP\x + t, \bmu^{\x, \sigma}_t, \X_t^{\x, \sigma})$ is Markov and, moreover, being continuous, it is also a strong Markov process that takes its values in $\mathfrak{X}$.

\subsection{Outline of the paper}

Section \ref{s:LDP} is dedicated to the Large Deviation Principle. First, we define what we mean by the Large Deviation Principle (LDP). Then, under the set of assumptions~\ref{A-1}, we establish the LDP for self-interacting diffusion in Theorem \ref{th:LDP_gen}. Next, we present some results related to the LDP that are useful for the exit-time problem.

Section \ref{s:Exit-time} deals with the exit-time problem. More precisely, we prove the main result, which is Theorem \ref{th:main_th}. To do so, we present in Section \ref{s:aux_results} auxiliary lemmas that are later used in Section \ref{s:main_proof} to prove the main Theorem \ref{th:main_th}. These lemmas are then proved in Section \ref{s:aux_proofs}.

Section~\ref{s:Generalization} is focused on the generalization of Assumptions \ref{A-1} and \ref{A-2}. In this section, we examine the scenario of a locally Lipschitz continuous drift term ($\nabla V$ and $\nabla F$) and an unbounded domain $G$ from which we want to exit.


\subsection{List of notations}

The paper contains number of technical results. To make them more reader-friendly, we present here a list of notions that are used later in the paper.

\quad

\begin{center}
\begin{tabularx}{0.9\columnwidth} {| >{\centering\arraybackslash \hsize=.35\hsize}X
  | >{\raggedright\arraybackslash \hsize=.65\hsize}X 
  |} 
 \hline
    
 $|\cdot|$, $\langle\cdot \; ; \cdot \rangle$ 	& Euclidean norm and scalar product in $\R^d$ \\ \hline
 
 $\Int(\Gamma)$/$\accentset{\circ}{\Gamma}$, $\Cl(\Gamma)$/$\overline{\Gamma}$, $\partial \Gamma$ 
												& Interior, closure and boundary of a set $\Gamma \subset \R^d$ \\ \hline
    
 $C([0, T]; \R^d)$ 								& The set of continuous functions on $[0, T]$ with values in $\R^d$ equipped with the norm $\|\phi\|_\infty := \max_{t \in [0, T]} |\phi(t)|$ \\ \hline

 $C_0([0, T]; \R^d)$ 							& Subset of $C([0, T]; \R^d)$ consisting of functions $\psi$ such that $\psi(0) = 0$ \\ \hline 
    
 $C^2(\R^d; \R)$ 								& The set of functions on $\R^d$ with values in $\R$ with continuous partial derivatives up to order 2 \\ \hline

 $L([0, T]; \R^d)$ 								& The set of integrable w.r.t the Lebesgue measure functions on $[0, T]$ with values in $\R^d$ \\ \hline

 $L_2([0, T]; \R^d)$ 							& Subset of $L([0, T]; \R^d)$ consisting of functions $\psi$ such that $\int_0^T |\psi(s)|^2\dd{s} < \infty$ \\ \hline

 $\mathcal{P}_2(\R^d)$ 							& The set of probability measures with finite second moments, defined on $\R^d$, and equipped with the Wasserstein-2 distance $\mathbb{W}_2$ (see e.g. \cite[Definition 6.1]{villani2008optimal})\\ \hline

\end{tabularx}
\end{center}

\begin{center}
\begin{tabularx}{0.9\columnwidth} {| >{\centering\arraybackslash \hsize=.35\hsize}X
  | >{\raggedright\arraybackslash \hsize=.65\hsize}X 
  |} 
\hline

 $\mathcal{P}_2(\overline{G})$ 					& The set of probability measures with finite second moment, defined on the closure $\overline{G} \subset \R^d$, where $G$ is the same domain as in Assumption \ref{A-2}, page~\pageref{A-2} \\ \hline

 $(\mathfrak{X}, d_{\mathfrak{X}})$ 			& A metric space consisting of triples $\x = (t, \mu, x)$, $t \in [0, \infty]$, $\mu \in \mathcal{P}_2(\R^d)$, and $x \in \R^d$ (see Section~\ref{s:remarks_and_init_cond})\\ \hline

 $\TP$, $\OM$, $\LP$ 							& Projection maps on the first, second, and third coordinates of an $\x \in \mathfrak{X}$ (see Section~\ref{s:remarks_and_init_cond})\\ \hline

 $B_\rho(a) := \{x \in \R^d: |x - a| \leq \rho\}$, \quad
 $S_\rho(a) := \{x \in \R^d: |x - a| = \rho\}$	& A closed ball and a sphere of radius $\rho$ around $a$ in $\R^d$\\ \hline

 $\mathbb{B}_\rho(\mu) := \{\nu \in \mathcal{P}_2(\overline{G}): \mathbb{W}_2(\mu; \nu) \leq \rho\}$
												& A closed ball of radius $\rho$ around $\mu$ in $\mathcal{P}_2(\overline{G})$, where $G$ will be the same domain as in Assumption~\ref{A-2}, page~\pageref{A-2}\\ \hline
 
 $\mathbb{B}_\rho(\x) := \{\mathbf{y} \in \mathfrak{X}: d_{\mathfrak{X}}(\mathbf{y}, \x) \leq \rho\}$ 
												& A closed ball of radius $\rho$ around $\x$ in $\mathfrak{X}$\\ \hline
    
 $\#$ 											& A push-forward symbol (see e.g. \cite[p.11]{villani2008optimal}) \\ \hline   

\end{tabularx}
\end{center}

\quad

All the equations in the paper are meant to hold a.s. unless stated otherwise. If we consider the probability of events involving $\X^{\sigma}$ or other random variables that depend on the initial condition $\x$, this initial condition is denoted in the subscript of $\Prob$ as $\Prob_\x$.

\section{Large Deviation Principle}\label{s:LDP}

In this paper, the Large Deviations techniques are widely used. Here, by the Large Deviation Principle (LDP) we mean the following asymptotic behaviour of measures. 

\begin{definition}
Family of measures $(\nu^\sigma)_{\sigma > 0}$ defined on some Banach space $B$ equipped with Borel sigma-algebra $\mathcal{B}$ is said to satisfy the Large Deviation Principle with a good rate function $I$ if for any measurable set $\Gamma \in \mathcal{B}$:
    \begin{equation*}
        -\inf_{x \in \accentset{\circ}{\Gamma}} I(x) \leq \liminf_{\sigma \xrightarrow{} 0} \frac{\sigma^2}{2} \log \nu^\sigma(\Gamma) \leq \limsup_{\sigma \xrightarrow{} 0} \frac{\sigma^2}{2} \log \nu^\sigma(\Gamma) \leq -\inf_{x \in \overline{\Gamma}} I(x),
    \end{equation*}
where $I: B \to [0, \infty]$ is a lower-semicontinuous function (this property defines \textit{rate function}) whose level sets $\{x: I(x) \leq \alpha\}$ are compact subsets of $B$ for any $0 \leq \alpha < \infty$ (which means by definition that the rate function is \textit{good}).
\end{definition}

For more information about the LDP, see \cite{DZ10}. Note that our definition of the LDP by its appearance deviates from the conventional one, but they are equivalent up to multiplication of the ``conventional'' rate function by $1/2$. 

\subsection{Establishing the LDP for the SID}\label{ss:LDP-SID}

In this section, we prove the LDP for the Self-Interacting diffusions of the type \eqref{eq:SID_main_sys_gen} under assumptions \ref{A-1}. Consider the following theorem that is the main result of the section.

\begin{theorem}\label{th:LDP_gen}
    Under Assumptions \ref{A-1}, for any $\x \in \mathfrak{X}$ the probability measures $(\nu^{\x, \sigma})_{\sigma > 0}$ induced on $C([0, T]; \R^d)$ by the process $(\X_t^{\sigma})_{0\leq t \leq T}$, which is the unique solution of the system \eqref{eq:SID_main_sys_gen}, satisfy a LDP with the following good rate function:
    \begin{equation}\label{eq:LDP_gen_rate}
    \begin{aligned}
        I_{T}^{\x}(f) &:=  \frac{1}{4}  \int_0^T \Big|\dot{f}(t) + \nabla V(f(t)) + \frac{\TP{\x}}{\TP{\x} + t} \nabla F * \OM{\x}(f(t)) \\
        &\quad + \frac{1}{\TP{\x} + t}\int_0^t \nabla F(f(t) - f(s))\dd{s} \Big|^2\dd{t} ,\; \text{for } f \in H_1^{\LP{\x}} 
    \end{aligned}
    \end{equation}
    and $I_{T}^{\x}(f) := \infty$ otherwise. Here, for $x \in \R^d$, we define $H_1^{x} := \{x + \int_0^{\cdot} g(s) \dd{s}: g \in L_2([0, T]; \R^d) \}$.
\end{theorem}

\begin{proof}
    Let us define the function $G: C_0([0, T]; \R^d) \to C([0, T]; \R^d)$ that maps every function $g$ to the unique solution of the following equation:
    \begin{equation*}
    \begin{aligned}
        f(t) = \LP{\x} -  \int_0^t\nabla V(f(s))\dd{s} - \int_0^t \frac{\TP{\x}}{\TP{\x} + s} \nabla F*\OM{\x} (f(s))\dd{s} \\
        - \int_0^t\frac{1}{\TP{\x} + s}\int_0^s \nabla F(f(s) - f(u))\dd{u}\dd{s} + g(t).
    \end{aligned}
    \end{equation*}
     Let us show that $G$ is continuous. Fix some $\delta > 0$ and two functions $g_1, g_2 \in C_0([0,T]; \R^d)$ such that $\|g_1 - g_2\|_{\infty} \leq \delta$. Let $f_1 = G(g_1)$, $f_2 = G(g_2)$. Then, using Lipschitz continuity of $\nabla V$ and $\nabla F$ (Assumption \ref{A-1}.\ref{A-1.Lip}) we can get:
    \begin{equation*}
        \begin{aligned}
        |f_1(t) - f_2(t)| & \leq \left( \text{Lip}_{\nabla V} + \text{Lip}_{\nabla F} \right) \int_0^t \big| f_1(s) - f_2(s) \big|\dd{s} \\
        & \quad + \int_0^t \frac{\text{Lip}_{\nabla F}}{\TP{\x} + s} \int_0^s \big| f_1(u) - f_2(u)\big|\dd{u}\dd{s} + \delta \\
        &\leq \left( \text{Lip}_{\nabla V} + \left(1 + \frac{T}{\TP{\x}} \right)\text{Lip}_{\nabla F} \right) \int_0^t |f_1(s) - f_2(s)|\dd{s} + \delta.
        \end{aligned}
    \end{equation*}
  Therefore, by Grönwall's inequality 
    \begin{equation*}
        \|f_1 - f_2\|_{\infty} \leq \delta \exp\left\{\left( \text{Lip}_{\nabla V} + \left(1 + \frac{T}{\TP{\x}} \right)\text{Lip}_{\nabla F} \right)T\right\};
    \end{equation*}
    it means the continuity of the map $G$ for any possible $\x \in \mathfrak{X}$.
    
    By the uniqueness of the solution to \eqref{eq:SID_main_sys_gen}, we can express $\nu^{\x, \sigma}$ as $\nu^{\x, \sigma} = G_{\#} \eta^\sigma$, where $\eta^\sigma$ is the probability measure induced on $C_0([0, T]; \R^d)$ by the path of Brownian motion $W^\sigma = \sigma W$. The LDP for the path of Brownian motion with vanishing noise is known under the name of Schilder theorem (\cite[Theorem~5.2.3]{DZ10}). Since $G$ is continuous, we can apply the Contraction principle (see \cite[Theorem~4.2.1]{DZ10}) and conclude that the family of measures $(\nu^{\x, \sigma})_{\sigma \geq 0}$ satisfies a LDP with the good rate function
    \begin{equation*}
    \begin{aligned}
        I_{T}^{\x}(f) &= 
            \inf_{g \in G^{-1}(f)} \frac{1}{4} \int^T_0 |\dot{g}(t)|^2\dd{t} \\
            &= 
        \begin{cases}
            \frac{1}{4} \int_0^T |\dot{f}(t) + \nabla V(f(t)) + \frac{\TP{\x}}{\TP{\x} + t}\int_{\R^d} \nabla F(f(t) - u) \OM{\x}(\dd{u})\\
            \quad\quad + \frac{1}{\TP{\x} + t}\int_0^t \nabla F(f(t) - f(u))\dd{u}|^2\dd{t} ,\; \text{for } f \in H_1^{\LP{\x}}, \\
            \infty, \quad\text{otherwise}.
        \end{cases}
    \end{aligned}
    \end{equation*}
    
\end{proof}

\begin{rem}
    For simplicity of the notation, we define the rate function corresponding to the system \eqref{eq:SID_main_sys} as $I_T^{x_0}$, i.e.
    \begin{equation*}
        I_T^{x_0}(f) := \begin{cases}
            \frac{1}{4} \int_0^T |\dot{f}(t) + \nabla V(f(t)) + \frac{1}{t}\int_0^t \nabla F(f(t) - f(s)) \dd{s}|^2\dd{t} ,\; \text{for } f \in H_1^{x_0}, \\
            \infty, \quad\text{otherwise}.
            \end{cases}
    \end{equation*}
\end{rem}

\subsection{Results related to the LDP}

Let us first introduce another assumption on the interaction potential $F$. Consider:  

\begin{custom_assumption}{A-3}\label{A-3}     
    There exists $C_{\nabla F} < \infty$ such that $|\nabla F(x)| \leq C_{\nabla F}$ for any $x \in \R^d$.
\end{custom_assumption}

This assumption will be only used in Lemma~\ref{lm:LDP_conv_init_cond} and can be substituted, for example, with a condition that the diffusion $\X^{\sigma}$ does not leave some bounded domain in $\R^d$. Since it is indeed the case in the main Theorem~\ref{th:main_th}, one can suppose that Assumption \ref{A-3} holds without loss of generality. However, for clarity, we opt to explicitly state it whenever it is used.


The following lemma generalizes the large deviation principle for the case of converging initial conditions. Under assumptions \ref{A-1} and \ref{A-3}, the following result holds.

\begin{lemma}\label{lm:LDP_conv_init_cond}
    For any family $\{\x_\sigma\}_{\sigma > 0}$ such that $\x_\sigma \xrightarrow[\sigma \to 0]{d_\mathfrak{X}}~\x$, the following inequalities hold:

     \begin{enumerate}
         \item For any closed $\Phi \subset C([0, T]; \R^d)$
         \begin{equation*}
             \limsup_{\substack{\sigma \to 0}} \frac{\sigma^2}{2} \log \Prob_{\x_\sigma}(\X^{\sigma} \in \Phi) \leq - \inf_{\phi \in \Phi} I_T^{\x} (\phi).
         \end{equation*}
         \item For any open $\Psi \subset C([0, T]; \R^d)$
         \begin{equation*}
             \liminf_{\substack{\sigma \to 0}} \frac{\sigma^2}{2} \log \Prob_{\x_\sigma}(\X^{\sigma} \in \Psi) \geq - \inf_{\phi \in \Psi} I_T^{\x} (\phi).
         \end{equation*}
     \end{enumerate}
\end{lemma}
\begin{proof}
    We will show that the families of measures $(\nu^{\x_\sigma, \sigma})_{\sigma > 0}$ and $(\nu^{\x, \sigma})_{\sigma > 0}$, which are families of probability measures induced on $C([0, T]; \R^d)$ by $\X^{\x_\sigma, \sigma}$ and $\X^{\x, \sigma}$ respectively, are exponentially equivalent (see \cite[Definition~4.2.10]{DZ10}) for any $\{\x_\sigma\}_{\sigma > 0}$ such that $\x_\sigma \xrightarrow[\sigma \to 0]{d_\mathfrak{X}} \x$.
    
    Indeed, if we define $Z^\sigma_t = \X^{\x_\sigma, \sigma}_t - \X^{\x, \sigma}_t$, where $\X^{\x_\sigma, \sigma}_t$ and $\X^{\x, \sigma}_t$ are driven by the same Brownian motion, then
    \begin{equation*}
        \begin{aligned}
            |Z_t^\sigma|  &\leq |\LP{\x_\sigma} - \LP{\x}| + \int_0^t |\nabla V(\X_s^{\x_\sigma, \sigma}) - \nabla V(\X_s^{\x, \sigma})|\dd{s} \\
            &\quad + \int_0^t |\frac{\TP{\x_\sigma}}{\TP{\x_\sigma} + s} \nabla F * \OM{\x_\sigma} (\X_s^{\x_\sigma, \sigma}) - \frac{\TP{\x}}{\TP{\x} + s} \nabla F * \OM{\x}(\X_s^{\x, \sigma})|\dd{s}\\
            &\quad + \int_0^t \int_0^s \big| \frac{1}{\TP{\x_\sigma} + s}  \nabla F(\X_s^{\x_\sigma, \sigma} - \X_z^{\x_\sigma, \sigma}) - \frac{1}{\TP{\x} + s} \nabla F(\X_s^{\x, \sigma} - \X_z^{\x, \sigma})\big|\dd{z}\dd{s}.
        \end{aligned}
    \end{equation*}
    In order to separate the effect of closeness of $\TP{\x_\sigma}$ to $\TP{\x}$ from $Z_t^\sigma$ in the last two integrals, we add and subtract $\frac{\TP{\x_\sigma}}{\TP{\x_\sigma} + s} \nabla F*\OM{\x}(\X_s^{\x, \sigma})$ and $\frac{1}{\TP{\x_\sigma} + s} \nabla F(\X_s^{\x, \sigma} - \X_z^{\x, \sigma})$ in the corresponding integrals. Since $\nabla V$ and $\nabla F$ are Lipschitz continuous, we get
    \begin{equation*}
        \begin{aligned}
            |Z_t^\sigma| & \leq (1 + \text{Lip}_{\nabla F})d_{\mathfrak{X}}(\x_\sigma, \x) + (\text{Lip}_{\nabla V} + \text{Lip}_{\nabla F}) \int_0^t |Z_s^\sigma| \dd{s} \\
            & \quad + \int_0^t \int_{\R^d}  \left| \left(\frac{\TP{\x_\sigma}}{\TP{\x_\sigma} + s} - \frac{\TP{\x}}{\TP{\x} + s} \right) \nabla F(\X_s^{\x, \sigma} - z) \right| \OM{\x}(\dd{z})\dd{s} \\
            & \quad + \text{Lip}_{\nabla F} \int_0^t \frac{1}{\TP{\x_\sigma} + s} \int_0^s \big( |Z_s^\sigma| + |Z_z^\sigma|\big) \dd{z}\dd{s} \\
            & \leq (1 + \text{Lip}_{\nabla F}) d_{\mathfrak{X}}(\x_\sigma, \x) + \left(\text{Lip}_{\nabla V} + 2 \text{Lip}_{\nabla F} + \frac{T}{\TP{\x_\sigma}}\right) \int_0^t |Z_s^\sigma|\dd{s} \\
            & \quad + C_{\nabla F}\int_0^t \left|\frac{\TP{\x_\sigma}}{\TP{\x_\sigma} + s} - \frac{\TP{\x}}{\TP{\x} + s}  \right| \dd{s},
        \end{aligned}
    \end{equation*}
    where we get the last inequality by applying assumption \ref{A-3} to  $\int_0^t \int_{\R^d}  \left| \left(\frac{\TP{\x_\sigma}}{\TP{\x_\sigma} + s} - \frac{\TP{\x}}{\TP{\x} + s} \right) \nabla F(\X_s^{\x, \sigma} - z) \right| \OM{\x}(\dd{z})\dd{s}$.
    
    
    
    
    Since the expression inside the integral is bounded, we have $\int_0^t \left|\frac{\TP{\x_\sigma}}{\TP{\x_\sigma} + s} - \frac{\TP{\x}}{\TP{\x} + s}  \right| \dd{s} = O(|\TP{\x_\sigma} - \TP{\x}|) = O(d_{\mathfrak{X}}(\x_\sigma, \x))$. Thus, by Grönwall's inequality,
    \begin{equation*}
        |Z_t^\sigma| \leq \left(1 + \text{Lip}_{\nabla F} +  C_{\nabla F} \right)O(d_{\mathfrak{X}}\big(\x_\sigma; \x) \big) \exp{\left(\text{Lip}_{\nabla V} + 2 \text{Lip}_{\nabla F} + \frac{T}{\TP{\x_\sigma}} \right) T}.
    \end{equation*}
    
    It means that $\Prob (|Z_t^\sigma| \geq \delta) = 0$ for any $\delta$ if we choose $\sigma$ to be small enough. That proves exponential equivalence of $(\nu^{\x_\sigma,\sigma})_{\sigma > 0}$ and $(\nu^{\x,\sigma})_{\sigma > 0}$, and, by the contraction principle (see \cite[Theorem~4.2.13]{DZ10}), the lemma itself.
    
\end{proof}

The following lemma elaborates on the idea of the LDP and provides a tool to study the asymptotic behaviour ($\sigma \to 0$) of the process $\X^{\sigma}$ with respect to its initial conditions. This and all the following lemmas of this section hold under Assumption \ref{A-1}.

\begin{lemma}\label{lm:LDP_compact_init_cond}
    For any compact subset $\mathcal{C} \subset \mathfrak{X}$ the following inequalities hold:
    
    \begin{enumerate}
    \item For any closed $\Phi \subset C([0, T]; \R^d)$    
        \begin{equation*}
            \limsup_{\sigma \to 0} \frac{\sigma^2}{2} \log \sup_{\x \in \mathcal{C}} \Prob_{\x}(\X^{\sigma} \in \Phi ) \leq - \inf_{\x \in \mathcal{C}} \inf_{\phi \in \Phi} I_T^{\x}(\phi).
        \end{equation*}
   \item Similarly, for any open $\Psi \subset C([0, T]; \R^d)$
    
        \begin{equation*}
            \liminf_{\sigma \to 0} \frac{\sigma^2}{2} \log \inf_{\x \in \mathcal{C}} \Prob_{\x}(\X^{\sigma} \in \Psi ) \geq -\sup_{\x \in \mathcal{C}} \inf_{\phi \in \Psi} I_T^{\x}(\phi).
        \end{equation*}
    \end{enumerate}
\end{lemma}

\begin{proof}
    For the first inequality, for each fixed $\varepsilon > 0$ let us define $$I^\varepsilon := \min \left\{\inf_{\x \in \mathcal{C}} \inf_{\phi \in \Phi} I_T^{\x}(\phi) - \varepsilon; \frac{1}{\varepsilon} \right\}.$$
    
    By Lemma \ref{lm:LDP_conv_init_cond} for any $\x \in \mathcal{C}$ there exists small enough $\sigma_{\x} > 0$ such that for any $\sigma < \sigma_{\x}$
    \begin{equation*}
        \frac{\sigma^2}{2} \log \sup_{\mathbf{y} \in B_{\sigma_\x}(\x)} \Prob_{\mathbf{y}}(\X^\sigma \in \Phi) \leq - I^\varepsilon.
    \end{equation*}
    
    Since $\mathcal{C}$ is compact, we can cover it by finite amount of $B_{\sigma_{\x_i}}(\x_i)$ for some $\x_1, \dots, \x_m \in \mathcal{C}$. Then, for any $\sigma < \min_{1 \leq i \leq m}\sigma_{\x_i}$
    \begin{equation*}
        \frac{\sigma^2}{2} \log \sup_{\x \in \mathcal{C}} \Prob_{\x}(\X^{\sigma} \in \Phi ) \leq - I^\varepsilon,
    \end{equation*}
    and that proves the first inequality. One can prove the second inequality the same way.
\end{proof}

By definition, rate functions are lower semicontinuous, i.e. for any $\x \in \mathfrak{X}$, the inequality $\liminf_{\phi \to \phi_0} I^\x_T(\phi) \geq I^\x_T(\phi_0)$ holds, or, equivalently, all level sets $L_\alpha := \{\phi \in C([0, T]; \R^d): I^\x_T(\phi) \leq \alpha\}$ are closed for $0 \leq \alpha < \infty$. An immediate consequence of this property is that infima of $I^\x_T$ are achieved over compact sets. The following lemma extends this lower semicontinuity property of our rate functions to the case of converging initial conditions besides the argument of $I_T$. In other words, $I_T$ as a function of two arguments $(\x, \phi) \in \mathfrak{X} \times C([0, T]; \R^d)$ still possesses lower semicontinuity property.

\begin{lemma}\label{lm:lower_semicont_I}
For any $T > 0$, $\x \in \mathfrak{X}$, $\phi \in C([0, T]; \R^d)$, and any sequences $\{\x_n\}$, $\{\phi_n\}$, where $\x_n \in \mathfrak{X}$, $\phi_n \in C([0, T]; \R^d)$, such that $\x_n \xrightarrow[n \to \infty]{d_\mathfrak{X}} \x$ and $\phi_n \xrightarrow[n \to \infty]{} \phi$ the following inequality holds:
\begin{equation*}
    \liminf_{n \to \infty} I^{\x_n}_T(\phi_n) \geq I^{\x}_T(\phi).
\end{equation*}
\end{lemma}

\begin{proof}
        Without loss of generality, let us assume that all $\phi_n \in H_1^{\LP{\x}}$. Those $\phi_n$ for which it is not true do not influence the $\liminf$ since $I^{\x_n}_T(\phi_n) = \infty$). Then
	\begin{equation}\label{eq:aux:I_x_n_phi_n}
	\begin{aligned}
        4 I_{T}^{\x_n}(\phi_n) &= \int_0^T \big|\dot{\phi_n}(t) + \nabla V(\phi_n(t)) + \frac{\TP{\x_n}}{\TP{\x_n} + t}\int_{\R^d} \nabla F(\phi_n(t) - u) \OM{\x_n}(\dd{u})\\
        &\quad + \frac{1}{\TP{\x_n} + t}\int_0^t \nabla F(\phi_n(t) - \phi_n(u))\dd{u}\big|^2\dd{t}.
    \end{aligned}
	\end{equation}
	Let us add and subtract two terms of the form $A:= \frac{\TP{\x}}{\TP{\x} + t}\int_{\R^d} \nabla F(\phi_n(t) - u) \OM{\x}(\dd{u})$ and $B:= \frac{1}{\TP{\x} + t}\int_0^t \nabla F(\phi_n(t) - \phi_n(u))\dd{u}$. Note that for any $c \in (0, 1)$ and for any $a, b \in \R^d$ the following inequality holds: $|a + b|^2 \geq (1 - c) |a|^2 + (1 - \frac{1}{c}) |b|^2$. Using two previous statements, let us split \eqref{eq:aux:I_x_n_phi_n} into two parts, one of which does not depend on $\x_n$, but depends on $\x$ instead. We then have
	\begin{equation*}
	\begin{aligned}
		4 I_{T}^{\x_n}(\phi_n) &\geq (1 - c) 4 I_T^{\x}(\phi_n) + \left(1 - \frac{1}{c}\right) \int_0^T\Big|\frac{\TP{\x_n}}{\TP{\x_n} + t}\int_{\R^d} \nabla F(\phi_n(t) - u) \OM{\x_n}(\dd{u}) - A\\
		&\quad  + \frac{1}{\TP{\x_n} + t}\int_0^t \nabla F(\phi_n(t) - \phi_n(u))\dd{u} - B\Big|^2 \dd{t} \\
		& = : (1 - c)4 I_T^{\x}(\phi_n) + \left(1 - \frac{1}{c} \right)\textbf{I}.
	\end{aligned}	
	\end{equation*}
    Since $I_T^{\x}$ is a good rate function, by taking $\liminf$ from both sides of the inequality above we get
	\begin{equation}\label{eq:aux:I_x_n_phi_n_2}
	   \liminf_{n \to \infty} I_{T}^{\x_n}(\phi_n) \geq (1 - c) I_T^{\x}(\phi) + \frac{1}{4}\left(1 - \frac{1}{c}\right) \liminf_{n \to \infty} \textbf{I},
	\end{equation}
	for all $c \in (0, 1)$. Therefore, in order to prove the lemma, it is enough to show that $\liminf_{n \to \infty} \textbf{I} = 0$. Then we could take the limit $c \to 0$ from both sides of \eqref{eq:aux:I_x_n_phi_n_2} and obtain the necessary result.
	
	Let us consider $\textbf{I}$. We use inequality of the form $|a + b|^2 \leq 2|a|^2 + 2|b|^2$ for $\textbf{I}$ and get an upper bound of the form. $\textbf{I}$ is less or equal than:
	\begin{equation*}
	\begin{aligned} 
        &2\int_0^T \Big| \frac{\TP{\x_n}}{\TP{\x_n} + t}\int_{\R^d} \nabla F(\phi_n(t) - u) \OM{\x_n}(\dd{u}) - \frac{\TP{\x}}{\TP{\x} + t}\int_{\R^d} \nabla F(\phi_n(t) - u) \OM{\x}(\dd{u})  \Big|^2\dd{t} \\
		& + 2\int_0^T \Big|\frac{1}{\TP{\x_n} + t}\int_0^t \nabla F(\phi_n(t) - \phi_n(u))\dd{u} - \frac{1}{\TP{\x} + t}\int_0^t \nabla F(\phi_n(t) - \phi_n(u))\dd{u} \Big|^2\dd{t} \\
		&=: 2\text{I}\text{I}_1 + 2\text{I}\text{I}_2.
	\end{aligned}
	\end{equation*}

    Let us apply dominated convergence theorem to both $\text{I}\text{I}_1$ and $\text{I}\text{I}_2$. In order to do so, we prove that integrands are uniformly bounded.
    
    First, consider $\text{I}\text{I}_1$. Of course, fractions $\frac{\TP{\x_n}}{\TP{\x_n} + t}$ and $\frac{\TP{\x}}{\TP{\x} + t}$ are bounded by $1$. As for $ \int_{\R^d} |\nabla F(\phi_n(t) - u)| \OM{\x_n}(\dd{u})$, we introduce the following decomposition
    \begin{equation} \label{eq:aux:nabla_F(phi_n)}
    \begin{aligned}
        \int_{\R^d} &|\nabla F(\phi_n(t) - u)|  \OM{\x_n}(\dd{u}) \\
        &= \int_{\R^d} |\nabla F(\phi_n(t) - u) - \nabla F(\phi(t) - u) + \nabla F(\phi(t) - u)| \OM{\x_n}(\dd{u}) \\
        & \leq \text{Lip}_{\nabla F} |\phi_n(t) - \phi(t)| + \int_{\R^d} |\nabla F(\phi(t) - u) |\OM{\x_n} \\
        & \leq \text{Lip}_{\nabla F} \|\phi_n - \phi\|_{\infty} + C_{\nabla F} \left( \max_{t \in [0, T]} |\phi(t)| + \left(\int_{\R^d} |u|^2 \OM{\x_n}\right)^{1/2} \right),
    \end{aligned}
    \end{equation}
    and since $\int |u|^2 \OM{\x_n} \xrightarrow[n \to \infty]{} \int |u|^2 \OM{\x} $ (see \cite[Theorem~6.9]{villani2008optimal}), the integral $ \int_{\R^d} |\nabla F(\phi_n(t) - u)| \OM{\x_n}(\dd{u})$ and, as a consequence, $ \int_{\R^d} |\nabla F(\phi_n(t) - u)| \OM{\x}(\dd{u})$ are uniformly (in $t$) bounded.
    
    For $\text{I}\text{I}_2$, we also use lipschitzness of $\nabla F$ when needed, convergence of $\phi_n$ in uniform norm topology towards $\phi$, as well as bounds of the form $|\phi(t)| \leq \max_{t \in [0, T]} |\phi(t)|$. That easily gives us uniform (in $t$) boundedness of the integrand of the integral $\text{I}\text{I}_2$.
    
    Thus, to calculate $\lim_{n \to \infty} \textbf{I}$, we can use the dominated convergence theorem and pass the limit inside both of the integrals $\text{I}\text{I}_1$ and $\text{I}\text{I}_2$.
    
    Since $\nabla F$ is continuous and since the uniform boundedness in $t$ and in $u$ on finite time interval $[0, T]$ of $\nabla F(\phi_n(t) - \phi_n(u))$ can be easily established, $\lim_{n \to \infty} \text{I}\text{I}_2$ is clearly equal to $0$. Limits of components of $\text{I}\text{I}_1$ are also obvious, except for $\int_{\R^d} |\nabla F(\phi_n(t) - u)| \OM{\x_n}(\dd{u})$ that needs some attention. Let us show that expression $\int_{\R^d} |\nabla F(\phi(t) - u)| \OM{\x}(\dd{u})$ is actually its limit with $n \to \infty$. Similarly to computations in \eqref{eq:aux:nabla_F(phi_n)}, we get 
    \begin{equation*}
    \begin{aligned}
        &\Big| \int_{\R^d} \nabla F(\phi_n(t) - u)  \OM{\x_n}(\dd{u}) - \int_{\R^d} \nabla F(\phi(t) - u)  \OM{\x}(\dd{u})\Big| \\
        & \qquad \leq \Big| \int_{\R^d} \nabla F(\phi(t) - u)  \OM{\x_n}(\dd{u}) - \int_{\R^d} \nabla F(\phi(t) - u)  \OM{\x}(\dd{u})\Big| + \text{Lip}_{\nabla F} \|\phi_n - \phi\|_{\infty}.
    \end{aligned}
    \end{equation*}
    
    As was pointed out before, convergence of measures in Wasserstein distance gives convergence of respective integrals, since $\nabla F$ is Lipschitz continuous \cite[Theorem 6.9]{villani2008optimal}. 
    
    This is the last remark needed to observe that $\lim_{n \to \infty} \textbf{I} = 0$. Thus, the lemma is proved by \eqref{eq:aux:I_x_n_phi_n_2}. 
\end{proof}

As was pointed out before, lower semicontinuity guarantees that infima of a function are achieved over compact sets. We summarise this property by the following corollary.

\begin{corollary}\label{cor:inf_I_over_compact}
    For any $T > 0$, any compact set $\Phi \subset C([0, T]; \R^d)$, and any compact set $\mathcal{C} \subset \mathfrak{X}$ there exist $\phi^* \in \Phi$ and $\x^* \in \mathcal{C}$ such that
    \begin{equation*}
        \inf_{\x \in \mathcal{C}} \inf_{\phi \in \Phi} I^{\x}_T(\phi) = I^{\x^*}_T(\phi^*).
    \end{equation*}
\end{corollary}

\subsection{Compactness results}

In this paper, compact subsets of $\mathfrak{X}\times C([0,T]; \R^d)$ of particular form are considered. Let us present two results about compactness of some sets that will be used later in the proof. The results of this section hold under Assumptions~\ref{A-1}.

\begin{lemma}\label{lm:K_is_compact}
    For any $T > 0$ and for any compact subsets $C_1, C_2 \subset \R^d$ the following set $$\mathcal{C} := \{\x \in \mathfrak{X}: T \leq \TP{\x} \leq \infty, \; \LP{\x} \in C_1, \; \OM{\x} \in \mathcal{P}_2(C_2)\}.$$ is a compact subset of $\mathfrak{X}$. 
\end{lemma}

\begin{proof}
    Projections of $\mathcal{C}$ on the first two axes are obviously compact subsets in $[0, \infty]$ and $\R^d$. For compactness of the projection on the third axis, note that by Prokhorov's theorem this set is compact in weak topology which is metrizable by Wasserstein-2 distance (see \cite[Theorem~6.9]{villani2008optimal}).
\end{proof}

\begin{lemma}\label{lm:A_x_compact}
    Let $\mathcal{C}$ be a compact subset of $\mathfrak{X}$ such that $\inf\{\TP{\x} : \x \in \mathcal{C}\} > 0$. For any $\x \in \mathcal{C}$ define $\Phi^\x := \{\phi \in C([0, T]; \R^d): I^\x_T(\phi) \leq 1\}$. Then $$\Phi = \bigcup_{\x \in \mathcal{C}} \Phi^\x$$ is a relatively compact subset of $C([0, T]; \R^d)$.
\end{lemma}

\begin{proof}
    In the case of complete metric spaces, the notion of relative compactness is equivalent to totally boundedness. By definition, the set $\Phi$ is totally bounded if for any $\varepsilon > 0$ there exists a finite cover of $\Phi$ with open balls of radius $\varepsilon$. The strategy of the proof is the following. We prove that for any $\varepsilon > 0$ and any $\x \in \mathcal{C}$ there exists $\delta_\x^\varepsilon > 0$ small enough such that $\bigcup_{\mathbf{y} \in \mathbb{B}_{\delta_\x^\varepsilon}(\x)} \Phi^{\mathbf{y}}$ is covered by a finite number of balls of radius $\varepsilon$. Since $\mathcal{C}$ is itself a compact subset of $\mathfrak{X}$ and $\{\mathbb{B}_{\delta_{\x}^\varepsilon}(\x)\}_{\x \in \mathcal{C}}$ is its cover by open sets, we can extract finite subcover $\{\mathbb{B}_{\delta_{\x_i}^\varepsilon}(\x_i)\}_1^n$ and thus prove the lemma, since $\Phi \subset \bigcup_{i = 1, \dots, n} \bigcup_{\mathbf{y} \in \mathbb{B}_{\delta_{\x_i}^\varepsilon}(\x_i)} \Phi^{\mathbf{y}}$ is covered by a finite number of open balls of radius $\varepsilon$. 

    \begin{figure}[b]
        \centering
        
        \begin{tikzpicture}
            \draw[thick] (-1.6, 0) ellipse (1 and 2);
            \draw[thick] (2.6, 0) ellipse (1.4 and 2);

            \draw (-1.6, -2.25) node {$\mathcal{C}$};
            \draw (2.6, -2.25) node {$\Phi$};

            \fill (-1.6, -0.4)  circle[radius=1pt];
            \draw (-1.45, -0.55) node {$\x$};
            \draw[color=blue] (-1.6, -0.4) circle[radius=0.6];
            \draw[color=blue] (-1.85, 0.45) node {$\mathbb{B}_{\delta_\x}(\x)$};

            \draw[rotate=25] (2.8, -0.4) ellipse (0.3 and 0.8);
            \draw (2.4, 0.2) node {$\Phi^\x$};

            \draw[->] (-1.6, -0.4) -- (2.6, 0.6);

            \draw[color=blue] (2.62, 0.5) ellipse (0.57 and 1.3);

            \draw[color=blue] (2.61, -1) node {$\bigcup_{\mathbf{y} \in \mathbb{B}_{\delta_\x}(\x)} \Phi^{\mathbf{y}}$};

            \draw[dashed, color=blue] (-1.6, -0.4 + 0.6) .. controls (0, 0.5) .. (2.62, 0.5 + 1.3);
            \draw[dashed, color=blue] (-1.6, -0.4 - 0.6) .. controls (0, -0.5) .. (2.62, 0.5 - 1.3);
            
        \end{tikzpicture}        
        
        \caption{Illustration of the relationship between some objects used in the proof, namely $\x$, $\Phi^\x$, $\mathbb{B}_{\delta_\x}(\x)$, and $\bigcup_{\mathbf{y} \in \mathbb{B}_{\delta_\x}(\x)} \Phi^{\mathbf{y}}$.}
        \label{fig:Phi_proof}
    \end{figure}
    
    As a result, all we have to prove is that for any $\varepsilon > 0$ and any $\x \in \mathcal{C}$ there exists $\delta_\x^\varepsilon > 0$ small enough such that $\bigcup_{\mathbf{y} \in \mathbb{B}_{\delta_\x^\varepsilon}(\x)} \Phi^{\mathbf{y}}$ is a totally bounded subset of $C([0, T]; \R^d)$.
    
    First of all, for any function $\phi \in \bigcup_{\mathbf{y} \in \mathbb{B}_{\delta_\x^\varepsilon}(\x)} \Phi^{\mathbf{y}}$ if $\delta_{\x}^\varepsilon$ is small enough, the following integral is bounded by a positive constant $C_1$: 
    \begin{equation*}
        \int_0^T |\dot{\phi}(s)|^2 \dd{s} \leq C_1.
    \end{equation*}
    
    Indeed, by adding and subtracting $\nabla V(\phi(s)) + \frac{\TP{\mathbf{y}}}{\TP{\mathbf{y}} + s} \nabla F * \OM{\mathbf{y}}(\phi(s)) + \\ \frac{1}{\TP{\mathbf{y}} + s}\int_0^s \nabla F(\phi(s) - \phi(z))\dd{z}$ inside the absolute value, we get: 
    
    \begin{equation*}
    \begin{aligned}
        \int_0^t |\dot{\phi}(s)|^2 \dd{s} & \leq 4 \int_{0}^{t} \Big|\dot{\phi}(s) + \nabla V(\phi(s)) + \frac{\TP{\mathbf{y}}}{\TP{\mathbf{y}} + s} \nabla F * \OM{\mathbf{y}}(\phi(s)) \\
		& \quad + \frac{1}{\TP{\mathbf{y}} + s}\int_0^s \nabla F(\phi(s) - \phi(z))\dd{z} \Big|^2\dd{s} + 4 \int_0^t |\nabla V(\phi(s))|^2\dd{s} \\
		& \quad + 4\int_0^t \left|\frac{\TP{\mathbf{y}}}{\TP{\mathbf{y}} + s} \nabla F * \OM{\mathbf{y}}(\phi(s)) \right|^2 \dd{s} \\
		& \quad + 4\int_0^t \left|\frac{1}{\TP{\mathbf{y}} + s}\int_0^s \nabla F(\phi(s) - \phi(z))\dd{z}\right|^2 \dd{s}.
    \end{aligned}
    \end{equation*}
    
    Let $\text{C} > 0$ be a generic positive constant. Since $I_{T}^{\mathbf{y}}(\phi) \leq 1$ for all $\mathbf{y} \in \mathbb{B}_{\delta_\x^\varepsilon}(\x)$ and using Lipschitz continuity of $\nabla V$ and $\nabla F$ (Assumption~\ref{A-1}.\ref{A-1.Lip}), we get
    \begin{equation*}
        \begin{aligned}
        \int_0^t &|\dot{\phi}(s)|^2 \dd{s}  \leq 4 + 8 \text{Lip}^2_{\nabla V} \!\int_0^t |\phi(s)|^2 \dd{s} + 8T|\nabla V(0)|^2 \\
		& + 4 \text{Lip}_{\nabla F}^2 \int_0^t \!\!\int_{\R^d} |\phi(s) - z|^2 \OM{\mathbf{y}}(\dd{z})\dd{s} + \frac{4 t \text{Lip}^2_{\nabla F}  }{\TP{\mathbf{y}}^2} \int_0^t \!\! \int_0^s |\phi(s) - \phi(z)|^2 \dd{z} \dd{s} \\
		& \qquad\qquad \leq  \text{C}  + \text{C}\left(1  + \frac{\text{C}}{\TP{\mathbf{y}}^2} \right) \int_0^t |\phi(s)|^2 \dd{s} + \text{C} \int_{\R^d} |z|^2 \OM{\mathbf{y}}(\dd{z}).
        \end{aligned}
    \end{equation*}

    Since all $\mathbf{y}$ belong to a ball of radius $\delta_{\x}^\varepsilon$ of $\x$, then the following inequalities hold: $|\TP{\mathbf{y}} - \TP{\x}| \leq \delta_{\x}^\varepsilon$, $\left|\int |z|^2 \OM{\mathbf{y}}(\dd{z}) - \int |z|^2 \OM{\x}(\dd{z}) \right| \leq \delta_{\x}^\varepsilon$ and $|\LP{\mathbf{y}} - \LP{\x}| \leq \delta_{\x}^\varepsilon$. Moreover, without loss of generality, we can choose $\delta_\x^\varepsilon$ to be small enough such that $\TP \mathbf{y} > \TP\x - \delta_\x^\varepsilon > 0$. 
    
    In addition, since $\phi \in H_1^{\LP\x}$ and by Hölder's inequality, we can bound $\int_0^t |\phi(s)|^2 \dd{s} \leq 2T \big(|\LP{\x}|^2 + \int_0^t\int_0^s|\dot{\phi}(u)|^2\dd{u}\dd{s} \big)$. As a result, there exist constants $C_2, C_3 > 0$ such that
    \begin{equation*}
        \int_0^t |\dot{\phi}(s)|^2 \dd{s} \leq C_2 + C_3 \int_0^t \!\! \int_0^s |\dot{\phi}(u)|^2 \dd{u} \dd{s}.
    \end{equation*}
    So, by Grönwall's inequality, we get
    \begin{equation}\label{eq:aux:phi_s}
        \int_0^T |\dot{\phi}(s)|^2 \dd{s} \leq C_2 e^{C_3 T} =: C_1.
    \end{equation}
    
    Finally, we use the Arzelà–Ascoli theorem to prove that the set $\bigcup_{\mathbf{y} \in \mathbb{B}_{\delta_\x^\varepsilon}(\x)} \Phi^{\mathbf{y}}$ is totally bounded. We use the bound \eqref{eq:aux:phi_s} and get the following inequalities showing uniform equicontinuity and pointwise boundedness of functions that belong to $\bigcup_{\mathbf{y} \in \mathbb{B}_{\delta_\x^\varepsilon}(\x)} \Phi^{\mathbf{y}}$. Namely, for any $0 \leq t_1 \leq t_2 \leq T$, we have
    \begin{equation*}
        |\phi(t_2) - \phi(t_1)|^2 \leq (t_2 - t_1) \int_{t_1}^{t_2} \left| \dot{\phi}(s) \right|^2 \dd{s} \leq C_1(t_2 - t_1).
    \end{equation*}
	This completes the proof.
    
\end{proof}

\section{Proof of the main theorem}\label{s:Exit-time}



In this section, we prove the main result of the paper, that is Theorem~\ref{th:main_th}, under Assumptions~\ref{A-1} and \ref{A-2}. Let us first recall that, by the definition of $a$, $X^0_t \xrightarrow[t \to \infty]{} a$. Assumption \ref{A-2} states that the domain $G$ is a regular enough neighbourhood of the point $a$ and a subset of its basin of attraction for the flow generated by $-\nabla U_a$. Since, as it was proved in Theorem \ref{th:LDP_gen}, the process $X^\sigma$ satisfies the LDP for any finite time interval $[0, T]$, we expect it to be close to its deterministic limit and also to converge towards $a$ with high probability. As a consequence, we also expect $\bmu_t^\sigma$ to converge towards $\delta_a$ in a finite time (this is shown in Lemma \ref{lm:conv_to_delta_a}). Assumption \ref{A-2}.\ref{A-2.Effect_pot_conv} ensures that for $\x \in \mathfrak{X}$ such that $\OM{\x}$ is close to $\delta_a$, starting from any $\LP{\x} \in G$, deterministic process $\X^{0}$ will be still attracted to $a$. Assumption \ref{A-2}.\ref{A-2.Strong_attraction_a} elaborates on this idea and suggests that the attracting forces around $a$ will be stronger than interaction forces. These facts combined produce Lemma \ref{lm:X0_conv}. These effects also suggest that $\bmu^\sigma$ for small $\sigma$ should stay close to $\delta_a$ with high probability at least until exit-time, since inside $G$ there will always be a force that pushes $X^\sigma$ towards $a$. That is shown in Lemma \ref{lm:gamma_gr_tau}.


In this section, we first state auxiliary Lemmas \ref{lm:conv_to_delta_a}-\ref{lm:unif_lower_bound} before proving the main theorem in Section~\ref{s:main_proof}. Finally, in the Section~\ref{s:aux_proofs}, we provide the proofs of the auxiliary lemmas.

\subsection{Auxiliary results}\label{s:aux_results}


\subsubsection{Initial descent to the point of attraction \texorpdfstring{$a$}{a}}
The first result that we will use for proving the main theorem describes the convergence of the solution to the system \eqref{eq:SID_main_sys} with vanishing noise towards the point $a$ in a constant time. Moreover, we show that its occupation measure also converges towards $\delta_a$ at some constant time with high probability for small $\sigma$. As a matter of fact, as will be shown in Lemma \ref{lm:gamma_gr_tau}, with $\sigma \to 0$, not only $\mu^\sigma$ converges towards $\delta_a$ in some constant time, but it also stays in the neighbourhood of $\delta_a$ at least until exit-time of $X^\sigma$ from the domain $G$. We call this ``effect of stabilization of the occupation measure''.

Consider the following lemma:

\begin{lemma}\label{lm:conv_to_delta_a}
    There exists $\rhobar > 0$ such that for any $0 < \rho < \rhobar$ there exists a finite time $T_{\text{st}}^\rho > 0$ such that for any $t \geq T_{\text{st}}^\rho$ we have:
    \begin{equation*}
        \lim_{\sigma \to 0}\Prob_{x_0} \left( X^\sigma_{t} \notin B_{\rho}(a) \right) = 0
    \end{equation*}
    and 
    \begin{equation*}
        \lim_{\sigma \to 0}\Prob_{x_0} \left(\mu^\sigma_{t} \notin \mathbb{B}_{\rho}(\delta_a)  \right) = 0.
    \end{equation*}
    Moreover, we have:
    \begin{equation*}
        \lim_{\sigma \to 0} \Prob_{x_0} (\tau_G^\sigma \leq t) = 0.
    \end{equation*}
\end{lemma}

\subsubsection{Descent of the deterministic process towards \texorpdfstring{$a$}{a}}
The following lemma claims that the deterministic process $\X^{0}$, with some suitable initial conditions, descends inside a small ball around $a$ at most in some finite time that depends on how close the process should get to the point $a$. This result is shown for any starting point $\LP\x \in G$, any $\OM\x$ close enough to $\delta_a$ ($\OM\x \in \mathbb{B}_{(1 + 2\varepsilon)\rho}(\delta_a)$), and any $\TP{\x}$ big enough ($\TP{\x} \geq T_{\text{st}}^\rho$). We also show that the occupation measure $\bmu^0$ will not move far from $\delta_a$ in the process of convergence.

Let $\X^{0}$ be the solution of \eqref{eq:SID_main_sys_gen} with $\sigma=0$ and define
\begin{equation}\label{eq:def:set_C_1}
    \mathcal{C}^1_{\varepsilon, \rho} := \{\x \in \mathfrak{X}: T_{\text{st}}^\rho \leq  \TP{\x} \leq \infty ,\; \OM{\x} \in \mathbb{B}_{(1 + 2\varepsilon)\rho}(\delta_a) \text{, and } \LP{\x} \in G\}.    
\end{equation}

\begin{lemma}\label{lm:X0_conv}
    There exist $\rhobar, \epsilonbar, \thetabar > 0$ such that for any $0 < \rho < \rhobar$ there exists a finite time $T^\rho_1 > 0$ such that for any $0 < \varepsilon < \epsilonbar$, for any $0 < \vartheta < \thetabar$ and for any $\x \in \mathcal{C}^1_{\varepsilon, \rho}$ the following estimations hold:
    \begin{enumerate}
        \item $\X^{0}_s \in G$ and $\bmu_s^{0} \in \mathbb{B}_{(1 + 3\varepsilon)\rho}(\delta_a)$ for any $0 \leq s < T_1^\rho$.
        \item $\X^{0}_s \in B_{(1 - \vartheta )\rho}(a)$ and $\bmu_s^{0} \in \mathbb{B}_{(1 + 2\varepsilon)\rho}(\delta_a)$ for any $s \geq T_1^\rho$.
    \end{enumerate}
\end{lemma}

\subsubsection{Attraction of the stochastic process towards \texorpdfstring{$a$}{a}}

The following lemma claims that at most at time $T_1^\rho$, which was introduced in Lemma~\ref{lm:X0_conv}, with high probability the following event happens. Starting at any point $\LP{\x}$ inside $G$ with some suitable initial conditions $\OM{\x}$ and $\TP{\x}$, the stochastic process $\X^{\sigma}$  either comes sufficiently close to the point $a$ or leaves the domain $G$ entirely. 

Define 
the first time when the process either comes close enough to point $a$ or leaves the domain $G$ as 
\begin{equation}\label{eq:def:tau_0^x}
    \tau_0 := \inf\{t \geq 0: \X^{\sigma}_t \in B_{\rho}(a) \cup \partial G\}.    
\end{equation}
Then the following lemma holds.

\begin{lemma}\label{lm:tau_0_<=T1}
       There exist $\rhobar, \epsilonbar > 0$ such that for any $0 < \rho < \rhobar$  and for any $0 < \varepsilon < \epsilonbar$  we have: 
    \begin{equation}
        \limsup_{\sigma \to 0} \frac{\sigma^2}{2} \log \sup_{\x \in \mathcal{C}^1_{\varepsilon, \rho}} \Prob_{\x} (\tau_0 > T_1^\rho) < 0,
    \end{equation}
    where $T_1^\rho> 0$ is the finite time defined in Lemma~\ref{lm:X0_conv} and $\mathcal{C}^1_{\varepsilon, \rho}$ is defined in \eqref{eq:def:set_C_1}.
\end{lemma}

\subsubsection{Behaviour in the annulus between \texorpdfstring{$B_{\rho}(a)$}{Brho(a)} and \texorpdfstring{$\partial G$}{boundary of G}}

The following lemma claims that the probability that, with $t$ big enough, the process $\X^{\sigma}$ stays in between $B_{\rho}(a)$ and $\R^d \setminus G$ without touching any of those sets, decays exponentially with $\sigma \to 0$. Moreover, with $t \to +\infty$ the rate of this decay tends to $- \infty$.

Introduce the following stopping time
\begin{equation}\label{eq:def:gamma}
    \gamma := \inf\{t \geq 0: \bmu_{t}^{\sigma} \notin \mathbb{B}_{(1 + 2\varepsilon)\rho}(\delta_a)\},    
\end{equation}
where $ \bmu^\sigma$ is part of the solution to equation \eqref{eq:SID_main_sys_gen}. 
We recall that $\tau_0$ is defined in \eqref{eq:def:tau_0^x}, and $\mathcal{C}^1_{\varepsilon, \rho}$ in \eqref{eq:def:set_C_1}. Consider the following lemma.

\begin{lemma}\label{lm:rest_inside_annulus_prob}
    There exist $\rhobar, \epsilonbar > 0$ such that for any $0 < \rho < \rhobar$ and for any $0 < \varepsilon < \epsilonbar$ the following limit holds:
    \begin{equation}
        \lim_{t \to \infty} \limsup_{\sigma \to 0} \frac{\sigma^2}{2} \log \sup_{\x \in \mathcal{C}^1_{\varepsilon, \rho}} \Prob_{\x}(t < \tau_0 < \gamma) = -\infty.
    \end{equation}
\end{lemma}

\subsubsection{Stabilization of the occupation measure}

 The following lemma establishes control of the occupation measure $\bmu^\sigma$ until the exit-time $\tau_G^\sigma$. We recall that 
 \begin{equation}\label{eq:def:H}
    H = \inf_{z \in \partial G} \{U_a(z) - U_a(a)\} 
 \end{equation}
 and present the following lemma.

\begin{lemma}\label{lm:gamma_gr_tau}
    There exist $\rhobar, \epsilonbar > 0$ such that for any $0 < \rho < \rhobar$ and for any $0 < \varepsilon < \epsilonbar$ the following limit holds:
    \begin{equation*}
        \lim_{\sigma \to 0} \sup_{\x \in \mathcal{C}^2_\rho} \Prob_{\x} \left(\gamma \leq \tau_G^\sigma \wedge \exp{\frac{2(H + 1)}{\sigma^2}} \right) = 0,
    \end{equation*}
    where 
    \begin{equation*}
        \mathcal{C}^2_\rho := \{\x \in \mathfrak{X}: T_{\text{st}}^\rho \leq \TP{\x} \leq \infty,\; \OM{\x} \in \mathbb{B}_{\rho}(\delta_a) \text{, and } \LP{\x} \in S_{\rho}(a)\}
    \end{equation*}
\end{lemma}

\subsubsection{Exit before nearing \texorpdfstring{$a$}{a}}

The following lemma states that a certain asymptotic upper bound for the probability of exiting the domain $G$ before approaching even smaller neighbourhood of $a$ holds. Let us define the following set 
\begin{equation}\label{eq:def:set_C_2}
    \mathcal{C}^3_{\varepsilon, \rho} := \{\x \in \mathfrak{X}: T_{\text{st}}^\rho \leq \TP{\x} \leq \infty,\; \OM{\x} \in \mathbb{B}_{(1 + 2\varepsilon)\rho}(\delta_a) \text{, and } \LP{\x} \in S_{(1 + \varepsilon)\rho}(a)\}.
\end{equation}
Consider the following lemma:

\begin{lemma}\label{lm:tau_0=tau_G}
    There exists $\epsilonbar > 0$ such that for any $0 < \varepsilon < \epsilonbar$ we have:
    \begin{equation*}
        \lim_{\rho \to 0} \limsup_{\sigma \to 0} \frac{\sigma^2}{2} \log \sup_{\x \in \mathcal{C}^3_{\varepsilon, \rho}} \Prob_{\x} (\tau_0 = \tau_G^\sigma, \tau_0 < \gamma)\leq -H,
    \end{equation*}
    where $H$ is defined as in \eqref{eq:def:H}.
\end{lemma}

\subsubsection{Control of dynamics for small time intervals}

The following result claims that for any $\zeta > 0$ and $c > 0$ there exists a small enough time $T(\zeta, c)$ such that during this time the probability that the diffusion~\eqref{eq:SID_main_sys_gen} drifts farther away than $\zeta$ from any starting point in $\mathcal{C}^1_{\varepsilon, \rho}$ decreases exponentially with the given rate $c$. Note that this time is chosen independently of $\varepsilon, \rho > 0$ small enough.
\begin{lemma}\label{lm:T(eps,c)}
    For any $0 < \rho, \varepsilon < 1$ and for any $\zeta, c > 0$ there exists a time $T(\zeta, c) > 0$, independent of $\rho$ and $\varepsilon$, such that 
    \begin{equation*}
        \limsup_{\sigma \to 0} \frac{\sigma^2}{2} \log \sup_{ \x \in \mathcal{C}^1_{\varepsilon, \rho}} \Prob_{\x} \left(\sup_{t \in [0, T(\zeta, c)]} |\X^{\sigma}_t - \LP{\x}| \geq \zeta \right) \leq -c,
    \end{equation*}
    where $\mathcal{C}^1_{\varepsilon, \rho}$ is defined in \eqref{eq:def:set_C_1}.
\end{lemma}

\subsubsection{Uniform lower bound for the probability of exit from \texorpdfstring{$G$}{G}}

The following lemma provides a uniform lower bound for the  probability of exit from domain $G$ starting from a  position that is close to $a$, given that the empirical measure of the process does not move far away from $\delta_a$.
\begin{lemma}\label{lm:unif_lower_bound}
    Let $\mathcal{C}^3_{\varepsilon, \rho}$ be defined as in \eqref{eq:def:set_C_2}. Then for any $\eta > 0$ there exists a finite time $T_0 > 0$ and $\rhobar(\eta), \epsilonbar(\eta) > 0$ such that for any $ 0 < \rho < \rhobar(\eta)$ and for any $0 < \varepsilon < \epsilonbar(\eta)$ we have:  
    \begin{equation*}
        \liminf_{\sigma \to 0} \frac{\sigma^2}{2} \log \inf_{\x \in \mathcal{C}^3_{\varepsilon, \rho}} \Prob_{\x}(\tau_{G}^\sigma \leq T_0,  \tau_G^\sigma < \gamma) > - (H + \eta).
    \end{equation*}
\end{lemma}

\subsection{Proof of the main theorem}\label{s:main_proof}

Let us first fix the parameters that are used in the proof. Take some $0 < \delta < H \wedge 1$ and fix the parameters $0 < \rhobar, \epsilonbar, \thetabar < 1$ to be small enough such that the conditions of the auxiliary lemmas~\ref{lm:X0_conv}--\ref{lm:tau_0=tau_G} hold. Decrease $\rhobar$ if necessary, so that for any $\rho < \rhobar$, due to Lemma~\ref{lm:tau_0=tau_G}, we have:
\begin{equation}\label{eq:aux:Lemma:tau_0=tau_G}
     \limsup_{\sigma \to 0} \frac{\sigma^2}{2} \log \sup_{ \x \in \mathcal{C}^3_{\varepsilon, \rho} } \Prob_{\x}(\tau_G^\sigma = \tau_0, \gamma > \tau_G^\sigma) \leq - \left(H - \frac{\delta}{4} \right).
\end{equation}
Decrease $\rhobar$ and $\epsilonbar$ one more time if necessary, so that Lemma~\ref{lm:unif_lower_bound} holds with $\eta = \delta/6$. Finally, fix some $0 < \rho < \rhobar$, $0 < \varepsilon < \epsilonbar$, and $0 < \vartheta < \thetabar$.

\subsubsection{Kramers' type law}
    Given the results of Lemmas \ref{lm:X0_conv} and \ref{lm:gamma_gr_tau}, we expect our process to spend most of its time near $a$ with $\sigma$ small enough. In order to have more information about this behaviour,
    we introduce the following stopping times
    \begin{equation}\label{eq:def_of_tau_theta}
        \begin{aligned}
            \tau_1 &:= \inf\{t \geq T_{\text{st}}^\rho :\; X_t^\sigma \in B_{\rho}(a) \cup \partial G \},\\
            \theta_{m} &:= \inf\{t \geq \tau_m: X_t^\sigma \in S_{(1 + \varepsilon)\rho}(a)\}, \\
            \tau_{m + 1} &:= \inf\{t\geq \theta_{m} :X_t^\sigma \in B_{\rho}(a) \cup \partial G \};
        \end{aligned}
    \end{equation}
    for $m \in \N$ with the convention that $\theta_{m + 1} = \infty$ if $\tau_m = \tau_G^\sigma$. We hereby can consider separately the intervals where $X^{\sigma}$ is close to point $a$ and those where it is not the case, i.e. the intervals of the type $[\tau_k, \theta_k]$ and $[\theta_k, \tau_{k + 1}]$. Besides, parameter $\rho$ controls the desired closeness of $X^\sigma$ to $a$. By $\tau_0$ and $\theta_0$ we denote the corresponding stopping times regardless of the number of exits that we had before, i.e.
    \begin{equation}\label{eq:def_of_tau_theta_0}
        \begin{aligned}
            \tau_0 &:= \inf\{t \geq 0: \X^{\sigma}_t \in B_{\rho}(a) \cup \partial G\}; \\
            \theta_0 &:= \inf\{t \geq 0 :\; \X_t^{\sigma} \in S_{(1 + \varepsilon)\rho}(a) \}.
        \end{aligned}
    \end{equation}

    \begin{figure}[t]
    \centering
    \begin{tikzpicture}
        \draw (0, 0) node[inner sep=0]{\includegraphics[width=0.37\columnwidth]{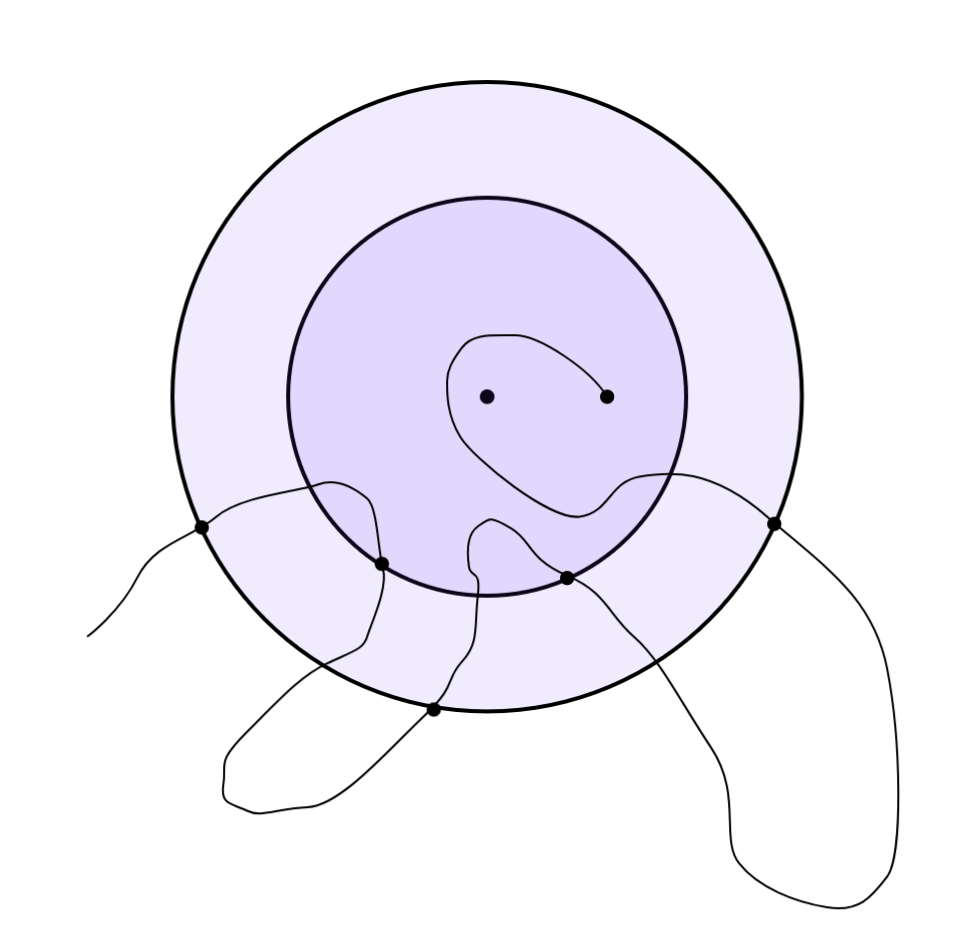}};
        \draw (0.1, 0.25) node {$\color{blue} a$};
        \draw (0, 0.9) node {$\color{blue} B_{\rho}(a)$};
        \draw (0.08, 1.6) node {$\color{blue} B_{(1 + \varepsilon)\rho}(a)$};
        \draw (-1.8, 0.8) node {$\color{blue} G$};
        
        \draw (0.8, 0.25) node {$\tau_1$};
        \draw (1.7, -0.25) node {$\theta_1$};
        \draw (0.7, -0.55) node {$\tau_2$};
        \draw (-0.06, -1.4) node {$\theta_2$};
        \draw (-0.35, -0.7) node {$\tau_3$};
        \draw (-1.15, -0.35) node {$\theta_3$};
        \draw (-1.9, -0.9) node {$t$};
    \end{tikzpicture}
    \caption{Illustration of the definitions of $\tau_k$ and $\theta_k$ for $d = 2$.}
    \end{figure}

    \textit{Lower bound}. Take the time $T_0 := T(\varepsilon\rho, H )$, where $T(\zeta, c)$ is defined by Lemma~\ref{lm:T(eps,c)}, i.e., it is a small enough time such that the probability that the process $\X^\sigma$ moves further away than $\varepsilon \rho$ from its starting point within time $T_0$ decays exponentially fast with a rate at least $H$. Note that for any $\x \in \mathfrak{X}$ such that $\LP{\x} \in B_\rho(a)$, we have:
    \begin{equation*}
        \Prob_\x (\theta_0 < T_0) \leq \Prob_\x \big(\sup_{t \in [0, T_0]} |\X_t^\sigma - \LP{\x}| \geq \varepsilon \rho \big).
    \end{equation*}
    In particular, if we denote
    \begin{equation*}
        \mathcal{C}^4_{\varepsilon, \rho} := \{ \x \in \mathfrak{X}: T^\rho_{\text{st}} \leq \TP{\x} \leq \infty;  \OM{\x} \in \mathbb{B}_{(1 + 2\varepsilon)\rho}(\delta_a); \LP{\x} \in B_{\rho}(a)\}, 
    \end{equation*}
    then by Lemma~\ref{lm:T(eps,c)}, for any $\sigma > 0$ small enough, we have:
    \begin{equation}
    \label{eq:aux:theta_0<T_0}
        \sup_{\x \in \mathcal{C}^4_{\varepsilon, \rho}} \Prob_\x (\theta_0 < T_0) \leq e^{-\frac{2H}{\sigma^2}}.
    \end{equation} 
    
    Consider the following event $\{\tau_G^\sigma \leq k T_0\}$. Using the Markov property of the process $(t, \mu^\sigma_t, X^\sigma_t)$ at time $T_{\text{st}}^\rho$, we obtain the following inequality:
    \begin{equation*}
    \begin{aligned}
        \Prob_{x_0}(\tau_G^\sigma \leq k T_0) &\leq \Prob_{x_0}(\tau_G^\sigma \leq T_{\text{st}}^\rho) + \Prob_{x_0}(X_{T_{\text{st}}^\rho}^\sigma \notin B_{\rho}(a)) + \Prob_{x_0} (\mu^\sigma_{T_{\text{st}}^\rho} \notin \mathbb{B}_{\rho}(\delta_a)) \\
        & \quad + \sup_{\x \in \mathcal{C}^2_\rho} \Prob_\x (\tau_G^\sigma \leq kT_0).
    \end{aligned}
    \end{equation*}
    
    We can split the last probability into two parts accordingly to whether $\{\gamma \leq \tau_G^\sigma\}$ holds or not, where $\gamma$ is the stopping time defined in \eqref{eq:def:gamma}. Recall that $\mathcal{C}^3_{\varepsilon, \rho}$ is defined in \eqref{eq:def:set_C_2}. Note that if the event $\{\tau_G^\sigma \leq k T_0, \gamma > \tau_G^\sigma\}$ takes place, then either first $k$ of disjoint events $\{\tau_i = \tau_G^\sigma,  \gamma > \tau_G^\sigma\}$ occur, or, otherwise, $\{\tau_k < \tau_G^\sigma \leq k T_0, \gamma > \tau_G^\sigma\}$. From the second event, it follows that at least one period of time $\theta_{i+1} - \tau_i$ was less than $T_0$. That gives us the following inequality:
    \begin{equation*}
    \begin{aligned}
        \Prob_{\x}(\tau_G^\sigma \leq k T_0) &\leq  \sum_{i = 1}^k \Prob_{\x}(\tau_G^\sigma = \tau_i, \gamma > \tau_G^\sigma) + \Prob_{\x}\left(\min_{1 \leq i \leq k} \{\theta_{i+1} - \tau_{i}\} < T_0, \gamma > \tau_G^\sigma \right)\\
        &\quad + \Prob_{\x} (\gamma \leq \tau_G^\sigma \leq kT_0). 
    \end{aligned}
    \end{equation*}
    Using the strong Markov property of $(t, \bmu^\sigma, \X^\sigma)$, that finally gives us:
        
    \begin{equation*}
    \begin{aligned}
        \Prob_{x_0}(\tau_G^\sigma \leq k T_0) & \leq \Prob_{x_0}(\tau_G^\sigma \leq T_{\text{st}}^\rho) + \Prob_{x_0}(X_{T_{\text{st}}^\rho}^\sigma \notin B_{\rho}(a)) + \Prob_{x_0} (\mu^\sigma_{T_{\text{st}}^\rho} \notin \mathbb{B}_{\rho}(\delta_a)) \\
        &\quad + k \sup_{ \x \in \mathcal{C}^3_{\varepsilon, \rho} } \Prob_{\x}(\tau_G^\sigma = \tau_0, \gamma > \tau_G^\sigma)  + k \sup_{\x \in \mathcal{C}^4_{\varepsilon, \rho}} \! \Prob_{\x}(\theta_0 < T_0, \gamma > \tau_G^\sigma) \\
        &\quad + \sup_{\x \in \mathcal{C}^2_\rho}\Prob_{\x}(\gamma \leq \tau_G^\sigma \leq kT_0).
    \end{aligned}
    \end{equation*}
    By \eqref{eq:aux:Lemma:tau_0=tau_G} and \eqref{eq:aux:theta_0<T_0}, we can choose $\sigma$ small enough such that the probabilities $\sup_{\x \in \mathcal{C}^3_{\varepsilon, \rho}} \Prob_{\x}(\tau_G^\sigma = \tau_0, \gamma > \tau_G^\sigma)$ and $\sup_{\x \in \mathcal{C}^4_{\varepsilon, \rho}} \Prob_{\x}(\theta_0 < T_0, \gamma > \tau_G^\sigma)$ above are less or equal than $\exp{-\frac{2(H - \delta/2)}{\sigma^2}}$. We also use Lemma \ref{lm:conv_to_delta_a} in order to deal with the first three probabilities and Lemma \ref{lm:gamma_gr_tau} for the $\Prob_{\x}(\gamma \leq \tau_G^\sigma \leq kT_0)$. Therefore, we establish the following upper bound:
    \begin{equation*}
        \Prob_{x_0}(\tau_G^\sigma \leq k T_0) \leq  o_{\sigma} + 2 k e^{\frac{-2(H - \delta/2)}{\sigma^2}},
    \end{equation*}
    where $o_\sigma$ is an infinitesimal w.r.t. $\sigma$. Finally, by choosing $k = \left\lfloor \frac{1}{T_0} \exp{\frac{2 (H - \delta)}{\sigma^2}} \right\rfloor$ we establish the following result:
    \begin{equation}
        \lim_{\sigma \to 0}\Prob_{x_0}(\tau_G^\sigma \leq e^{\frac{2 (H - \delta)}{\sigma^2}}) = 0.
    \end{equation}

    \textit{Upper bound}. Let $T_0 > 0$ be defined by Lemma~\ref{lm:unif_lower_bound}. Since $\{\tau_G^\sigma \leq T_0, \tau_G^{\sigma} < \gamma\} \subset \{\tau_G^\sigma \wedge \gamma \leq T_0\}$, and due to the way how $\rho$ and $\varepsilon$ were fixed, the result of Lemma~\ref{lm:unif_lower_bound} takes the form:
    \begin{equation}\label{eq:main_proof:tau_T0}
        \liminf_{\sigma \to 0} \frac{\sigma^2}{2} \log \inf_{\x \in \mathcal{C}^3_{\varepsilon, \rho}} \Prob_{\x}(\tau_{G}^\sigma \wedge \gamma \leq T_0) > - (H + \eta),
    \end{equation}
    where we recall that $\eta$ was fixed to be equal to $\delta/6$. By Lemma \ref{lm:tau_0_<=T1}, there exists $T_1^\rho > 0$ such that 
    \begin{equation}\label{eq:main_proof:Lemma3.3}
        \limsup_{\sigma \to 0} \frac{\sigma^2}{2} \log \sup_{\x \in \mathcal{C}^1_{\varepsilon, \rho}} \Prob_{\x} (\tau_0 > T_1^\rho) < 0,
    \end{equation}
    where $\mathcal{C}^1_{\varepsilon, \rho}$ is defined in \eqref{eq:def:set_C_1} and $\tau_0$ in \eqref{eq:def_of_tau_theta_0}.
    
    Let $T := T_0 + T_1^\rho$ and let us now introduce the following decomposition $\mathcal{C}^1_{\varepsilon, \rho} = A \cup B$, where $A$ is defined as a subset of $\mathfrak{X}$ such that for any $\x \in A$ we have
    \begin{equation*}
        \Prob_\x (\gamma \leq T) \geq e^{-\frac{2 \eta}{\sigma^2}},
    \end{equation*}
    and $B \subset \mathfrak{X}$ as its complement in $\mathcal{C}^1_{\varepsilon, \rho}$. Consider the following equation:
    \begin{equation} \label{eq:aux:bound_on_P}
        \begin{aligned}
            \inf_{\x \in \mathcal{C}^1_{\varepsilon, \rho}} \Prob_{\x} (\tau_G^\sigma \wedge \gamma \leq T ) = \min \left( \inf_{\x \in A} \Prob_{\x} (\tau_G^\sigma \wedge \gamma \leq T ); \;\;\; \inf_{\x \in B} \Prob_{\x} (\tau_G^\sigma \wedge \gamma \leq T )\right).
        \end{aligned}
    \end{equation}
    For the first infumum, by the definition of $A$, we have: 
    \begin{equation*}
        \inf_{\x \in A} \Prob_{\x} (\tau_G^\sigma \wedge \gamma \leq T ) \geq \Prob_\x (\gamma \leq T) \geq e^{-\frac{2 \eta}{\sigma^2}}.
    \end{equation*}
    For the second probability, note that the following inclusion of the events takes place: $\{\tau_0 \wedge \gamma \leq T_1^\rho, (\tau_G^\sigma - \tau_0\wedge\gamma) \wedge (\gamma - \tau_0 \wedge \gamma) \leq T_0\} \subset \{\tau_G^\sigma \wedge \gamma  \leq T\}$. Thus, we can use the strong Markov property of the triple $(t, \bmu^\sigma_t, \X^\sigma_t)$, evaluated at the stopping time $\tau_0 \wedge \gamma$, and thus get the following lower bound:
    \begin{equation*}
    \begin{aligned}
        \inf_{\x \in B} \Prob_{\x} (\tau_G^\sigma \wedge \gamma \leq T ) & \geq \inf_{\x \in B} \Prob_{\x} (\tau_0 \wedge \gamma \leq T_1^\rho) \inf_{\x \in \mathcal{C}^3_{\varepsilon, \rho}} \Prob_\x( \tau_G^\sigma \wedge \gamma \leq T_0) \\   
        & > \inf_{\x \in B} \Prob_{\x} (\tau_0 \leq T_1^\rho, \gamma > \tau_0) \; e^{-\frac{2(H + 2\eta)}{\sigma^2}},
    \end{aligned}
   \end{equation*}
   where the last inequality holds due to \eqref{eq:main_proof:tau_T0}, for $\sigma > 0$ small enough. Let us now consider the first probability:
   \begin{equation*}
   \begin{aligned}
       \inf_{\x \in B} \Prob_{\x} (\tau_0 \leq T_1^\rho, \gamma > \tau_0) &\geq \inf_{\x \in B} \Prob_\x(\tau_0 \leq T_1^\rho < \gamma) \\
       & = \inf_{\x \in B} \Big\{ \Prob_\x (\gamma > T_1^\rho) - \Prob_\x (\tau_0 > T_1^\rho, \gamma > T_1^\rho)\Big\} \\
       & \geq 1 - \sup_{\x \in B} \Prob_\x(\gamma \leq T) - \sup_{\x \in B} \Prob_\x (\tau_0 > T_1^\rho, \gamma > T_1^\rho) \geq \frac{1}{2},
   \end{aligned}
   \end{equation*}
    where, in the last inequality, we use the definition of the set $B$ and equation \eqref{eq:main_proof:Lemma3.3}.
   


    Now we can come back to the equation \eqref{eq:aux:bound_on_P} and finally conclude that
    \begin{equation}\label{eq:aux:bound_on_q}
        q := \inf_{\x \in \mathcal{C}^1_{\varepsilon, \rho}} \Prob_{\x} (\tau_G^\sigma \wedge \gamma \leq T ) \geq \exp{-\frac{2(H + 3\eta)}{\sigma^2}}.
    \end{equation}
    
    This lower bound will help us to calculate the probability of the event $\{\tau_G^\sigma \wedge \gamma > kT\}$ for any $k \in \N$ in the following way. Using the Markov property of the triple $(t, \bmu_t^\sigma, \X^\sigma_t)$, evaluated at the time $kT$, we get the following equation:
    \begin{equation*}
    \begin{aligned}
        \Prob_{\x} (\tau_G^\sigma \wedge \gamma > (k + 1)T) & \leq \sup_{\mathbf{y} \in \mathcal{C}^1_{\varepsilon, \rho}} \Prob_{\mathbf{y}} (\tau_G^\sigma \wedge \gamma > T) \, \Prob_\x (\tau_G^\sigma \wedge \gamma > kT) \\
        & \leq (1 - q)\, \Prob_\x (\tau_G^\sigma \wedge \gamma > kT).
     \end{aligned}    
    \end{equation*}
    Therefore, by induction in $k$, we get:
    \begin{equation}\label{eq:aux:(1-q)^k}
        \sup_{\x \in \mathcal{C}^1_{\varepsilon, \rho}} \Prob_{\x} (\tau_G^\sigma \wedge \gamma > kT) \leq (1 - q)^k.
    \end{equation}
    Note that $\tau_G^\sigma \wedge \gamma \leq \Big( T +  \sum_{k=1}^\infty  T \1_{\{\tau_G^\sigma \wedge \gamma > kT) \}}  \Big)$. Therefore, by \eqref{eq:aux:(1-q)^k} and \eqref{eq:aux:bound_on_q}, we can express the following expectation as:
    \begin{equation*}
        \sup_{\x \in  \mathcal{C}^1_{\varepsilon, \rho}}\E_{\x} \big( \tau_G^\sigma \wedge \gamma \big) \leq T \sum_{k = 0}^\infty (1 - q)^k = \frac{T}{q} \leq T \exp{\frac{2(H + 3\eta)}{\sigma^2}}.
    \end{equation*}
    By Markov's inequality, we get:
    \begin{equation*}
    \begin{aligned}
        \sup_{\x \in  \mathcal{C}^1_{\varepsilon, \rho}}\Prob_\x \left(\tau_G^\sigma \wedge \gamma \geq \exp{\frac{2(H + \delta)}{\sigma^2}}\right) 
        & \leq \exp{-\frac{2(H + \delta)}{\sigma^2}} \sup_{\x \in  \mathcal{C}^1_{\varepsilon, \rho}}\E_{\x} \big( \tau_G^\sigma \wedge \gamma \big) \\
        &\leq Te^{-\delta/\sigma^2} \xrightarrow[\sigma \to 0]{} 0.
    \end{aligned}
    \end{equation*}

    We remark that by Lemmas \ref{lm:conv_to_delta_a}, \ref{lm:gamma_gr_tau} and the derivations above, we have
    \begin{equation*}
    \begin{aligned}
        &\Prob_{x_0} \left(\tau_G^\sigma \geq \exp{\frac{2(H + \delta)}{\sigma^2}} \right) \\
        &\leq \max \Big\{ \Prob_{x_0} (\tau_G^\sigma \leq T_{\text{st}}^\rho);\; \Prob_{x_0} (X^\sigma_{T_{\text{st}}^\rho} \notin B_\rho(a)); \; \Prob_{x_0} (\mu^\sigma_{T_{\text{st}}^\rho} \notin \mathbb{B}_\rho(a)) \Big\} \\
        & \quad + \sup_{\x \in \mathcal{C}^2_\rho} \Prob_\x \left(\gamma \leq \tau_G^\sigma \wedge \exp{\frac{2(H + 1)}{\sigma^2}} \right)  \\
        & \quad + \sup_{\x \in \mathcal{C}^2_{\rho}} \Prob_\x \left(\tau_G^\sigma \geq \exp{\frac{2(H + \delta)}{\sigma^2}}, \gamma > \tau_G^\sigma \wedge \exp{\frac{2(H + 1)}{\sigma^2}} \right)\\
        & \xrightarrow[\sigma \to 0]{} 0.
    \end{aligned}
    \end{equation*}
This concludes the proof.

\subsubsection{Exit-location}\label{s:exit_location}

Before proving the exit-location result, we provide some observations about the geometry of the domain $G$ with respect to the effective potential. Note that its level sets $L_{\alpha} := \{x \in \R^d: U_a(x) = \alpha\}$ are smooth hypersurfaces, since both $V$ and $F$ belong to $C^2(\R^d; \R)$. Here, for simplicity, we denote by $L_{\alpha}$ only the parts of the level sets that intersect with $\overline{G}$. By Assumption \ref{A-2}.\ref{A-2.Effect_pot_conv}, for $\alpha = H$ there is in fact only one ``part'' of the hypersurface intersecting the set $\overline{G}$. Namely, it can be shown that $L_{H}$ is the boundary of the set $L^{-}_{H}:=\{x \in \R^d: U_a(x) \leq H\}$ that itself belongs to $\overline{G}$ and is a bounded \textit{connected} set (otherwise the Assumption \ref{A-2}.\ref{A-2.Effect_pot_conv} is violated).

Note that there exists a constant $C_H > 0$ such that for any point on the surface $x \in L_H$: $\langle n(x); - \nabla U_a(x) \rangle \geq C_H$, where $n(x)$ is the inner unit normal vector to the surface $L_H$ at the respective point. Indeed, if it is not the case, then either by the continuity argument there exists $x^* \in L_H$ such that $n(x^*) = 0$, which violets Assumption \ref{A-2}.\ref{A-2.Effect_pot_conv}, or all the vectors $-\nabla U_a(x)$ pin outside of the set $L^{-}_H$, which also contradicts the assumption above. It means that if we plug the interior $\Int(L^{-}_H)$ and the set $L_H$ in Assumptions \ref{A-2} instead of $G$ and $\partial G$, then all the conditions will be still satisfied. Then, by the proof  of Kramers' type law above, for $\tau^{\sigma}_{\Int(L^{-}_H)}:= \inf\{t \geq 0: \X^{ \sigma}_t \notin \Int(L^{-}_H)\}$, we have for any $\delta > 0$:
\begin{equation*}
    \sup_{\x \in \mathcal{C}^2_{\rho}} \Prob_{\x} \left( e^{\frac{2 (H - \delta)}{\sigma^2}} < \tau^{\sigma}_{\Int(L^{-}_H)} < e^{\frac{2 (H + \delta)}{\sigma^2}}\right),
\end{equation*}
where $\mathcal{C}^2_{\rho}$ is defined in \eqref{eq:def:set_C_4}. 

 \begin{figure}[b]
    \centering
    \begin{tikzpicture}
        \draw (0, 0)node[inner sep=0]{\includegraphics[width=0.465\columnwidth]{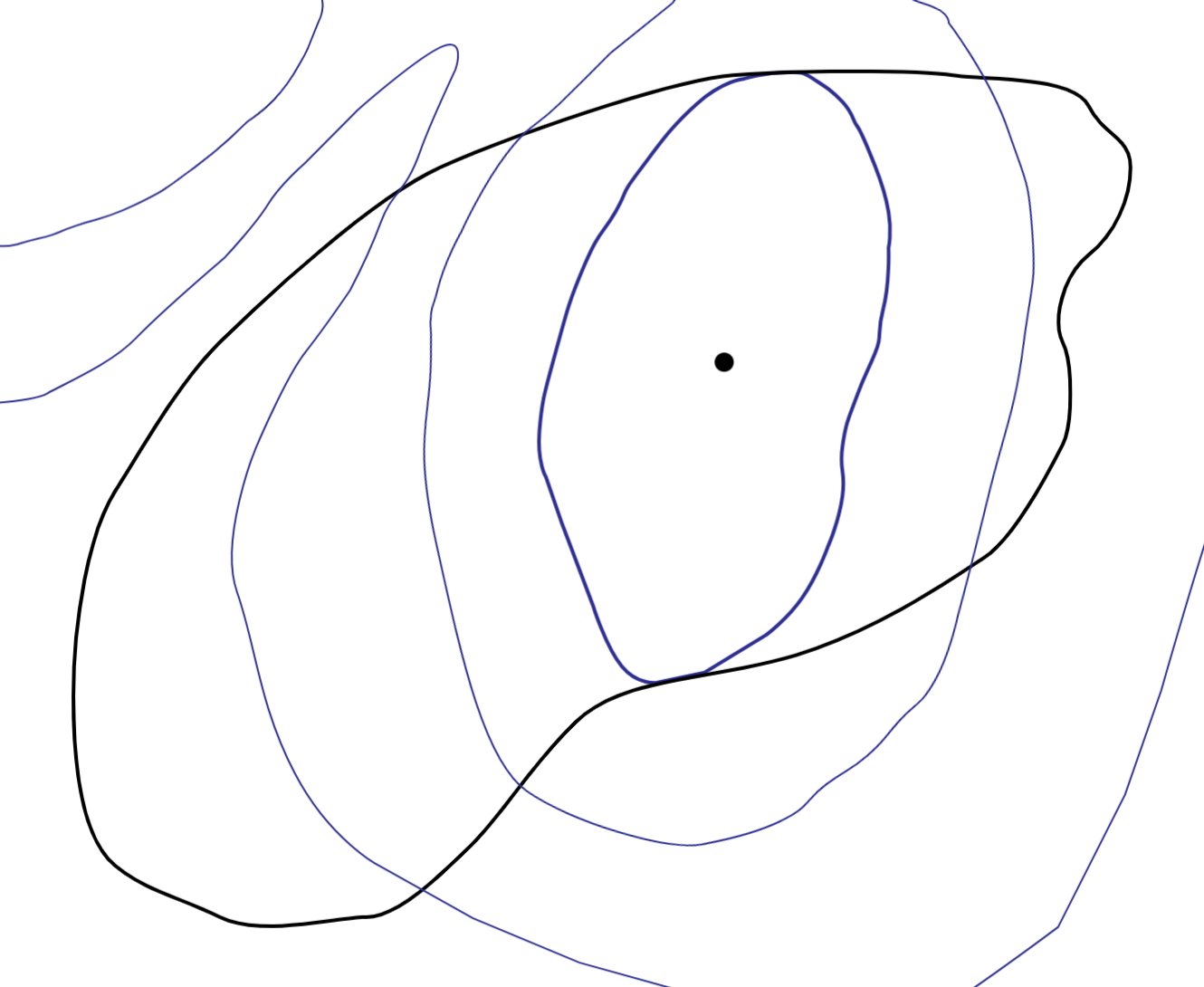}};
        \draw (0.6, 0.9)  node {$a$};
        \draw (-0.4, -0.4)  node {$\color{blue} L_{H}$};
        \draw (-1.1, -1.9)  node {$G$};
    \end{tikzpicture}
    \caption{Domain $G$ with level sets of $U_a = V + F(\cdot - a) - V(a)$. $L_{H} = \{x \in \R^d: U_a(x) = H\}$ is the smallest level set that touches the boundary $\partial G$}
    \label{fig:exit-location}
\end{figure}   

Since $U_a$ is continuous, there exists a small enlargement of $L_H$ such that the property $\langle n(x); - \nabla U_a(x) \rangle \geq \text{Const} > 0$ is still satisfied. Namely, there exists $\bar{\eta} > 0$ small enough such that for any $\eta < \bar{\eta}$ there exists a constant $C_{H + \eta} > 0$ such that $\langle n(x); - \nabla U_a(x) \rangle \geq C_{H + \eta}$. By analogy with the previous statement, it means that Assumptions \ref{A-2} are satisfied for $\eta > 0$ small enough and thus, for any $\delta > 0$,
\begin{equation*}
    \sup_{\x \in \mathcal{C}^2_{\rho}} \Prob_{\x} \left( e^{\frac{2 (H + \eta - \delta)}{\sigma^2}} < \tau^{\sigma}_{\Int(L^{-}_{H + \eta})} < e^{\frac{2 (H + \eta +\delta)}{\sigma^2}}\right).
\end{equation*}

Now we are ready to prove the exit-location result. Take the set $N \subset \partial G$ described in the theorem. By definition, $\inf_{z \in N} \{U_a(z) - U_a(a)\} > H$. That in particular means that we can choose $\eta < \bar{\eta}$ such that $H + \eta$ is smaller then $\inf_{z \in N} \{U_a(z) - U_a(a)\}$. It means that $N \cap \Cl(L_{H + \eta}^{-})$ is an empty set. 

The fact that $X_{\tau_G^\sigma}^\sigma \in N$ means that, after converging first towards $a$, we left the domain $L^{-}_{H + \eta}$ before the time $\tau_G^\sigma$, which is the exit-time from $G$. The following inequalities show that it is unlikely. Consider:
\begin{equation*}
\begin{aligned}
        \Prob_{x_0} (X^\sigma_{\tau_G^\sigma} \in N) & \leq \max \left\{ \Prob_{x_0} (\tau_G^\sigma \leq T_{\text{st}}^\rho); \Prob_{x_0} (X^\sigma_{T_{\text{st}}^\rho} \in B_\rho(a)); \Prob_{x_0} (\mu^\sigma_{T_{\text{st}}^\rho} \in \mathbb{B}_\rho(\delta_a)) \right\} \\
        & \quad + \sup_{\x \in \mathcal{C}^2_{\rho}}\Prob_{\x} (\tau^\sigma_{L^{-}_{H + \eta}} \leq \tau_G^\sigma).
\end{aligned}
\end{equation*}

Taking $\delta = \eta / 3$, and using Lemma~\ref{lm:conv_to_delta_a}, we get: 
\begin{equation*}
\begin{aligned}
    \Prob_{x_0} (X^\sigma_{\tau_G^\sigma} \in N)    
    & \leq o_\sigma + \sup_{\x \in \mathcal{C}^2_{\rho}}\Prob_{\x} \left(\tau_{G}^\sigma \geq e^{\frac{2(H + \delta)}{\sigma^2}} \right) + \sup_{\x \in \mathcal{C}^2_{\rho}}\Prob_{\x} \left(\tau^\sigma_{L^{-}_{H + \eta}} \leq \tau_G^\sigma \leq e^{\frac{2(H + \delta)}{\sigma^2}}\right).
\end{aligned}
\end{equation*}

That finally gives us the following upper bound for $\Prob_{x_0} (X^\sigma_{\tau_G^\sigma} \in N)$:

\begin{equation*}
    o_\sigma + \sup_{\x \in \mathcal{C}^2_{\rho}}\Prob_{\x} \left(\tau_{G}^\sigma \geq e^{\frac{2(H + \delta)}{\sigma^2}} \right) + \sup_{\x \in \mathcal{C}^2_{\rho}}\Prob_{\x} \left(\tau^\sigma_{L^{-}_{H + \eta}} \leq e^{\frac{2(H + \eta - \delta)}{\sigma^2}}\right),
\end{equation*}

that tends to zero by Lemma \ref{lm:gamma_gr_tau} and Kramers' type law for $G$ and $L_{H + \eta}^{-}$ proved above.

\subsection{Proofs of auxiliary lemmas}\label{s:aux_proofs}
\subsubsection{Proof of Lemma \ref{lm:conv_to_delta_a}: Initial descent to the point of attraction \texorpdfstring{$a$}{a}}
Let us fix $l:= \inf_{t>0, z \in \partial G} |X^0_t - z|$,  the distance between the deterministic path starting at $x_0$ and the boundary of the domain $G$. Let us also define as $T_{\text{st}}^{\rho}$ the first time when $\mu^0_{t} \in \mathbb{B}_{\rho/2}(\delta_a)$ as well as $X_t^0 \in B_{\rho/2}(a)$, where $\mu^0_{t}:= \frac{1}{t} \int_0^t \delta_{X^0_s}\dd{s}$ is the empirical measure of the deterministic process. This time obviously exists and is finite, since, by the definition of $a$, $X_t^0 \xrightarrow[t \to \infty]{} a$. 

Let us choose $l_\rho := l \wedge (\rho/4)$. Define $\Phi := \{\phi \in C([0, T_{\text{st}}^\rho]; \R^d): \|X^0 - \phi\|_{\infty} \geq l_{\rho}\}$. Then, since $\{X^\sigma_{T_{\text{st}}^\rho} \notin B_{\rho}(a)\} \subseteq \{X^\sigma \in \Phi\}$, by Theorem \ref{th:LDP_gen}, we have:
    \begin{equation*}
            \limsup_{\sigma \to 0} \frac{\sigma^2}{2} \log \Prob_{x_0} (X^\sigma_{T_{\text{st}}^\rho} \notin B_{\rho}(a)) \leq - \inf_{\phi \in \Phi} I_{T_{\text{st}}^\rho}^{x_0} (\phi).
    \end{equation*}

Furthermore, note that for any $\phi \notin \Phi$ we have $\mathbb{W}_2^2(\frac{1}{T_{\text{st}}^\rho}\int_0^{T_{\text{st}}^\rho} \delta_{\phi(s)}\dd{s}; \delta_a) \leq 2\mathbb{W}_2^2(\frac{1}{T_{\text{st}}^\rho}\int_0^{T_{\text{st}}^\rho} \delta_{\phi(s)}\dd{s}; \mu^0_{T_{\text{st}}^\rho}) + 2\mathbb{W}_2^2(\mu^0_{T_{\text{st}}^\rho}; \delta_a)  = \frac{1}{T_{\text{st}}^\rho}\int_0^{T_{\text{st}}^\rho} |\phi(s) - a|^2\dd{s} \leq \rho^2$. Therefore, by Theorem~\ref{th:LDP_gen}, the following inequality also holds:
    \begin{equation*}
            \limsup_{\sigma \to 0} \frac{\sigma^2}{2} \log \Prob_{x_0} (\mu^\sigma_{T_{\text{st}}^\rho} \notin \mathbb{B}_{\rho}(\delta_a) )
            \leq - \inf_{\phi \in \Phi} I_{T_{\text{st}}^\rho}^{x_0} (\phi).
    \end{equation*}

    Moreover, by the choice of $l_\rho$, we also have:
    \begin{equation*}
            \limsup_{\sigma \to 0} \frac{\sigma^2}{2} \log \Prob_{x_0} (\tau_G^\sigma \leq T_{\text{st}}^\rho)
            \leq - \inf_{\phi \in \Phi} I_{T_{\text{st}}^\rho}^{x_0} (\phi).        
    \end{equation*}

    Since $I_{T_{\text{st}}^\rho}^{x_0}$ has only one minimum which is given by $X^0$, then $I_{T_{\text{st}}^\rho}^{x_0}(X^0) = 0$, and as the distance between $X^0$ and the set $\Phi$ is strictly positive, we conclude that $\inf_{\phi \in \Phi} I_{T_{\text{st}}^\rho}^{x_0} (\phi) > 0$. That completes the proof.

\subsubsection{Proof of Lemma \ref{lm:X0_conv}: Descent of the deterministic process towards \texorpdfstring{$a$}{a}}
    Let us recall that $\Delta_x > 0$ is the positive constant introduced in Assumption~\ref{A-2}.\ref{A-2.Strong_attraction_a}. We separate the dynamics of the system into the following three parts (see Figure~\ref{fig:X0_conv_pic}) and present the proof in several steps. \textit{Step 1} shows that in time $\hat{T}_{1} > 0$, for any starting position $\LP{\x} \in G$, the process $\X^{0}$ converges inside $B_{\Delta_x}(a)$, whereas its occupation measure $\bmu^{0}$ does not move far from $\delta_a$ and stays inside $\mathbb{B}_{(1 + 3\varepsilon)\rho}(\delta_a)$. Rigorously, we prove that there exists a finite time $\hat{T}_{1} > 0$ such that for any $\x$ and for any $0 \leq t \leq \hat{T}_{1}$ we have
    \begin{equation*}
        \X_{t}^{0} \in G, \; \bmu_t^{0} \in \mathbb{B}_{(1 + 2.5\varepsilon)\rho}(\delta_a), \text{ and } \X_{\hat{T}_{1}}^{0} \in B_{\Delta_x}(a).
    \end{equation*}

    In \textit{Step 2}, we show that inside the set $B_{\Delta_x}(a)$, the  attraction force of the effective potential $U_a = V + F*\delta_a$ becomes so strong that, in some time $\hat{T}_2 > 0$, it drags $\X^{0}$ inside even smaller ball $B_{(1 - \vartheta)\rho}(a)$, whereas its occupation measure, as before, does not move far from $\delta_a$ and still stays inside $\mathbb{B}_{(1 + 3\varepsilon)\rho}(\delta_a)$. Rigorously, we show that there exists a finite time $\hat{T}_{2} > 0$ such that for any $\x \in \mathcal{C}^1_{\varepsilon, \rho}$ and for any $0 \leq t \leq \hat{T}_{2}$ we have 
    \begin{equation*}
        \X_{\hat{T}_{1} + t}^{0} \in G, \; \bmu_{\hat{T}_{1} + t}^{0} \in \mathbb{B}_{(1 + 3\varepsilon)\rho}(\delta_a), \text{ and } \X_{\hat{T}_{1} + \hat{T}_{2}}^{0} \in B_{(1 - \vartheta)\rho}(a).
    \end{equation*}
    
    Finally, in \textit{Step 3}, we show that, after hitting $B_{(1 - \vartheta)\rho}(a)$, $\X^{0}$ stays inside this ball for a long time $\hat{T}_3$ that is enough to attract the occupation measure $\bmu^{0}$ back inside $\mathbb{B}_{(1 + 2\varepsilon)\rho}(\delta_a)$. In particular, we show that there exists a finite time $\hat{T}_3 > 0$ such that for any $\x \in \mathcal{C}^1_{\varepsilon, \rho}$ and for any $0 \leq t \leq \hat{T}_{3}$ we have
    \begin{equation*}
        \X_{\hat{T}_{1} + \hat{T}_{2} + t}^{0} \in B_{(1 - \vartheta)\rho}(a), \; \bmu_{\hat{T}_{1} + \hat{T}_{2} + t}^{0} \in \mathbb{B}_{(1 + 3\varepsilon)\rho}(\delta_a), \text{ and } \bmu_{\hat{T}_{1} + \hat{T}_{2} + t}^{0} \in \mathbb{B}_{(1 + 2\varepsilon)\rho}(\delta_a).
    \end{equation*}

    Thus, the time that we are looking for in this lemma is represented through the following sum $T_1^\rho := \hat{T}_1 + \hat{T}_2 + \hat{T}_3$ that is independent of the initial condition $\mathcal{C}^1_{\varepsilon, \rho}$. Now we prove each step separately.

    \begin{figure}[b]
    \centering
    \begin{tikzpicture}
        \draw (0, 0) node[inner sep=0]{\includegraphics[width=0.465\columnwidth]{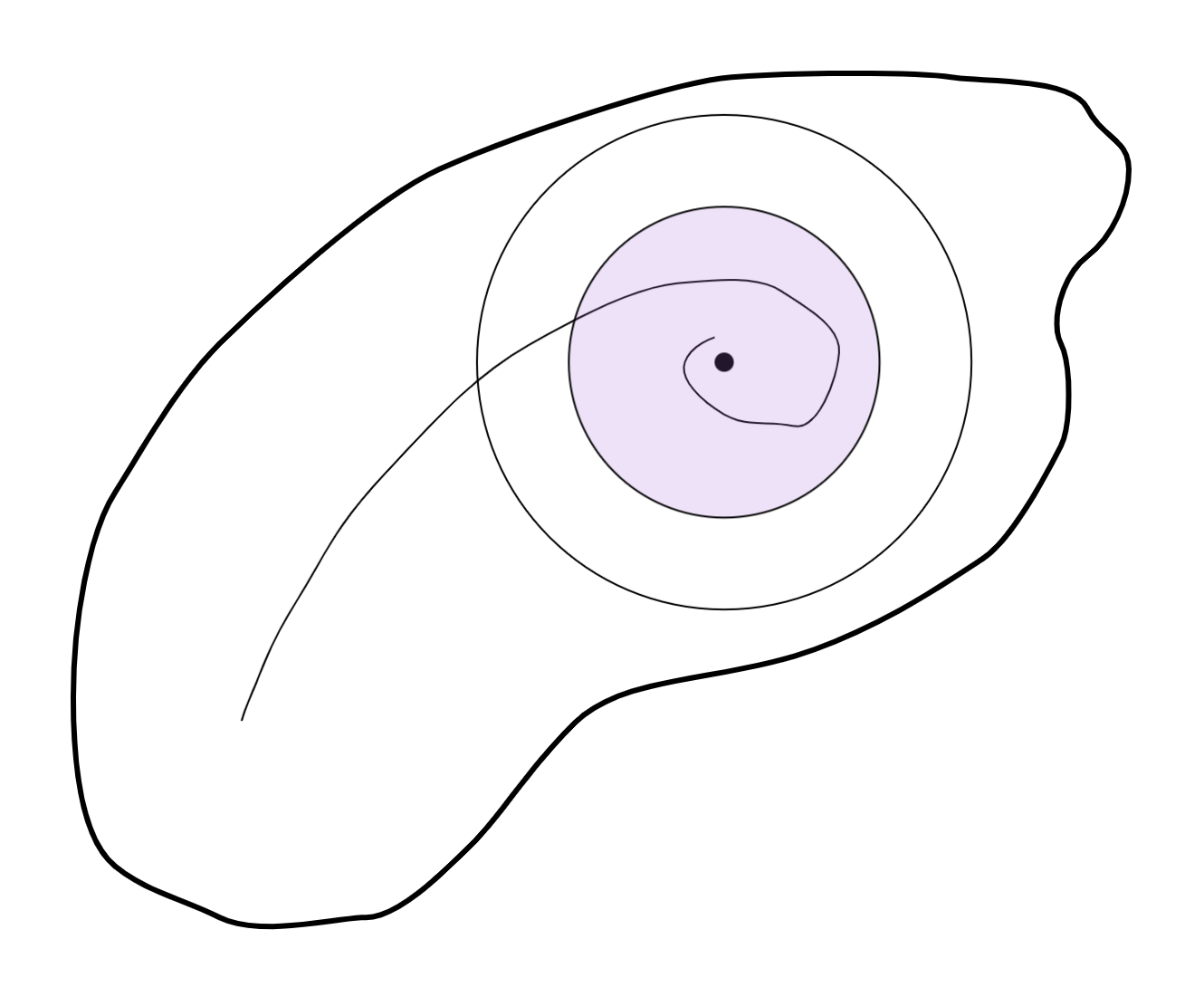}};
        \draw (-0.6, -1.4) node {$\color{blue} G$};
        \fill (-1.82,-1.15)  circle[radius=1pt];
        \draw (-1.55,-1.3)  node {$\LP{\x}$};

        \draw (-1.7, 0)  node {$\hat{T}_1$};
        \fill (-0.62, 0.58)  circle[radius=1pt];
        \draw (-0.3, 1.05)  node {$\hat{T}_2$};
        \fill (-0.13, 0.87)  circle[radius=1pt];

        \draw (1.3, 1.5)  node {$\color{blue} B_{\Delta_x}(a)$};
        \draw (0.7, 0.2)  node {$\color{blue} B_{(1 - \vartheta)\rho}(a)$};
        \draw (0.6, 0.9)  node {$\color{blue} a$};
    \end{tikzpicture}
    \caption{Dynamics of $\X^{0}$}
    \label{fig:X0_conv_pic}
    \end{figure} 

    \textit{Step 1. Convergence inside} $B_{\Delta_x}(a)$. By Assumption \ref{A-2}.\ref{A-2.Effect_pot_conv}, for any $x \in \overline{G}$, the flow $\phi^x$, generated by the effective potential $U_a$, converges towards $a$. Since, by Assumption \ref{A-1}.\ref{A-1.Lip}, both $\nabla V$ and $\nabla F$ are Lipschitz continuous and, by Assumption \ref{A-2}.\ref{A-2.G_open}, $\overline{G}$ is a compact set, there exists a uniform upper bound $\hat{T}_1$ for the time in which $\phi^x$ converges inside $B_{\Delta_x/2}(a)$.
    
    Let us use this fact to establish the time of convergence of $\X^{0}$ inside. Define $\eta_t := |\X^{0}_t - \phi^{\LP{\x}}_t|$. Applying Assumption \ref{A-1}.\ref{A-1.Lip}, we get    
    \begin{equation*}
    \begin{aligned}
        \eta_t &\leq \text{Lip}_{\nabla V}\int_0^t \eta_s \dd{s} + \text{Lip}_{\nabla F} \int_0^t \int_{\R^d}|(\X^{0}_s - z) - (\phi^{\LP{\x}}_s - a)|  \bmu^{0}_s(\dd{z})\dd{s} \\
        & \leq \text{Lip}_{\nabla V}\int_0^t \eta_s \dd{s} + \text{Lip}_{\nabla F} \left( \int_0^t \eta_s\dd{s} + \int_0^t \mathbb{W}_2(\bmu_s^{0}; \delta_a)\dd{s}\right).
    \end{aligned}
    \end{equation*}
    Then, by Grönwall's inequality:
    \begin{equation}\label{eq:aux:eta_t}
        \eta_t \leq \int_0^t \mathbb{W}_2(\bmu_s^{0}; \delta_a)\dd{s} \cdot \exp{(\text{Lip}_{\nabla V} + \text{Lip}_{\nabla F})t}.
    \end{equation}

    Since $\x \in \mathcal{C}^1_{\varepsilon, \rho}$, $\mathbb{W}_2(\bmu_0^{0}; \delta_a) \leq (1 + 2\varepsilon)\rho$. We can show, by \eqref{eq:aux:eta_t}, that if we take $\rhobar$ to be small enough, then $\mathbb{W}_2(\bmu_t^{0}; \delta_a) \leq (1 + 2.5\varepsilon)\rho$ at least for $t \leq \hat{T}_1$. Indeed, if we introduce a generic time $T^\prime := \inf\{t \geq 0: \mathbb{W}_2(\bmu^{0}_t; \delta_a) > (1 + 2.5\varepsilon)\rho\}$, we can express, using \eqref{eq:aux:eta_t}, for $t \leq T^\prime \wedge \hat{T}_1$:
    \begin{equation}\label{eq:aux:eta_t_2}
        \eta_t \leq (1 + 2.5\varepsilon)\rho \hat{T}_1 e^{(\text{Lip}_{\nabla V} + \text{Lip}_{\nabla F}) \hat{T}_1}.
    \end{equation}
    Since $\phi_t$ belongs to $\overline{G}$ for any $t \geq 0$, by inequality above, there exists a constant $C > 0$ such that $|\X_t^{0} - a| \leq C$ for any $t \leq T^{\prime}\wedge\hat{T}_1$. At the same time, by the definition of $\bmu^{0}$,
    \begin{equation*}
        \mathbb{W}_2(\bmu^{0}_t; \delta_a) \leq \frac{\TP{\x}}{\TP{\x} + t} \mathbb{W}_2(\OM{\x}; \delta_a) + \frac{t}{\TP{\x} + t} \mathbb{W}_2\left(\frac{1}{t}\int_0^t\delta_{\X_s^{0}}\dd{s}; \delta_a \right).
    \end{equation*}
    Therefore, for any $t \leq T^{\prime}\wedge\hat{T}_1$, 
    \begin{equation*}
        \mathbb{W}_2(\bmu^{0}_t; \delta_a) \leq (1 + 2\varepsilon)\rho + \frac{\hat{T}_1}{T_{\text{st}}^\rho + \hat{T}_1} C.
    \end{equation*}
    Note that, by Lemma~\ref{lm:conv_to_delta_a}, without loss of generality, we can increase $T_{\text{st}}^\rho$ if necessary such that $\frac{\hat{T}_1}{T_{\text{st}}^\rho + \hat{T}_1} C \leq 0.5 \epsilonbar \rhobar$. Thus, for any $t \leq T^{\prime}\wedge\hat{T}_1$:
    \begin{equation*}
        \mathbb{W}_2(\bmu^{0}_t; \delta_a) \leq (1 + 2.5\varepsilon)\rho.
    \end{equation*}
    That includes the time $t = T^{\prime}\wedge\hat{T}_1$. At the same time, by the definition of $T^\prime$, $\mathbb{W}_2(\bmu^{0}_{T^\prime}; \delta_a) > (1 + 2.5\varepsilon)\rho$ that means that $T^{\prime}\wedge\hat{T}_1 = \hat{T}_1$. 

    Now we can conclude that equation \eqref{eq:aux:eta_t_2} holds for any $t \leq \hat{T}_1$. Note that, without loss of generality, $\rhobar, \epsilonbar > 0$ can be assumed to be small enough such that $(1 + 2.5\epsilonbar)\rhobar < \Delta_x/2$, which gives $\eta_t < \Delta_x/2$ for any $t \leq \hat{T}_1$. That proves the uniform (in initial point $\LP{\x} \in G$) convergence inside $B_{\Delta_x}(a)$.

    \textit{Step 2. Convergence inside} $B_{(1 - \vartheta)\rho}(a)$. By Assumption \ref{A-2}.\ref{A-2.Strong_attraction_a}, in the set $B_{\Delta_x}(a)$ attraction forces towards the point $a$ prevail over the interaction forces. Firstly, without loss of generality, $\rhobar$ can be taken to be small enough such that $(1 + 3\epsilonbar)\rhobar < \Delta_\mu$, where $\Delta_\mu$ is defined in \ref{A-2}.\ref{A-2.Strong_attraction_a}. As was shown before, $\mathbb{W}_2(\bmu^{0}_{\hat{T}_1}; \delta_a) \leq (1 + 2.5\varepsilon)\rho$ and thus $\bmu^{0}_{\hat{T}_1} \in \mathbb{B}_{\Delta_\mu}(\delta_a)$. We can show that moreover $\bmu^{0}_{\hat{T}_1 + t} \in \mathbb{B}_{(1 + 3\varepsilon)\rho}(\delta_a)$ long enough such that in $\X^{0}$ converges inside $B_{(1 - \vartheta)\rho}(a)$.
    
    Let $T^{\prime \prime} := \inf\{t\geq 0: \mathbb{W}_2(\bmu^{0}_{\hat{T}_1 + t}; \delta_a) > (1 + 3\varepsilon)\rho, \text{ or } |\X_{\hat{T}_1 + t}^{0} - a| > \Delta_x\}$ and consider $\xi_t := |\X_{\hat{T}_1 + t}^{0} - a|^2$. By Assumption \ref{A-2}.\ref{A-2.Strong_attraction_a}, for any $t \leq T^{\prime \prime}$, its derivative is bounded by:
    \begin{equation*}
        \dot{\xi}_t \leq - 2 K_1 \xi_t + K_2\mathbb{W}_2(\bmu^{0}_{\hat{T}_1 + t}; \delta_a)\sqrt{\xi_t}. 
    \end{equation*}
    That guarantees that if $\xi_0 > \left(\frac{K_2}{K_1}(1 + 3\varepsilon)\rho \right)^2$ then we have the exponential decrease under the threshold
    \begin{equation*}
        \left(\frac{K_2}{K_1}(1 + 4\varepsilon)\rho \right)^2,
    \end{equation*}
    which in particular means that $T^{\prime \prime} = +\infty$.  
    
    Therefore, if, without loss of generality, we chose $\vartheta<1-\sqrt{\frac{K_2}{K_1}}$ and $\varepsilon<\frac{1}{4}\left(\sqrt{\frac{K_1}{K_2}}-1\right)$, then there exists a finite time $\hat{T}_2 := \inf \left\{t \geq 0: \xi_t\leq\left(\frac{K_2}{K_1}(1 + 4\varepsilon)\rho \right)^2\right\}$. To prove that $T^{\prime\prime}$ is infinite, we proceed by contradiction. Indeed, for any $t \leq T^{\prime \prime} \wedge \hat{T}_2$:
    \begin{equation*}
    \begin{aligned}
        \mathbb{W}_2(\bmu^{0}_{\hat{T}_1 + t}; \delta_a) &\leq  \mathbb{W}_2(\bmu^{0}_{\hat{T}_1}; \delta_a) + \frac{t}{\TP{\x} + \hat{T}_1 + t} \mathbb{W}_2\left(\frac{1}{t}\int_0^t\delta_{\X_{\hat{T}_1 + s}^{ 0}}\dd{s}; \delta_a \right) \\
        & \leq (1 + 2.5\varepsilon)\rho + \frac{\hat{T}_2}{T_{\text{st}}^\rho + \hat{T}_1 + \hat{T}_2} \Delta_x.
    \end{aligned}    
    \end{equation*}

    For any $\rho > 0$ we can choose, without loss of generality, $T_{\text{st}}^\rho$ to be big enough such that $\frac{\hat{T}_2}{T_{\text{st}}^\rho + \hat{T}_1 + \hat{T}_2} \Delta_x \leq 0.5\epsilonbar \rhobar$. That gives, for any $t \leq T^{\prime \prime} \wedge \hat{T}_2$:
    \begin{equation*}
        \mathbb{W}_2(\bmu^{0}_{\hat{T}_1 + t}; \delta_a) \leq (1 + 3\varepsilon)\rho.
    \end{equation*}
    Thus, as before, $\hat{T}_2 < T^{\prime \prime}$ or else we get a contradiction between the definition of $T^{\prime \prime}$ and the inequality above along with the fact that $\xi$ decreases for any $t \leq T^{\prime \prime} \wedge \hat{T}_2$.

    \textit{Step 3. Return of the occupation measure back inside} $\mathbb{B}_{(1 + 2\varepsilon)\rho}(\delta_a)$. The last time period can be found easily. Note that since $\X^{0}$ belongs to $B_{(1 - \vartheta)\rho}(a)$ at time $\hat{T}_1 + \hat{T}_2$, as well as $\bmu^{0}$ belongs to $\mathbb{B}_{(1 + 3\varepsilon)\rho}(\delta_a)$, then by Assumption \ref{A-2}.\ref{A-2.Strong_attraction_a}, $\X^{0}_t$ will not leave $B_{(1 - \vartheta)\rho}(a)$ for any $t \geq \hat{T}_1 + \hat{T}_2$. Using this fact and estimations that we had on $\bmu^{0}$ for time intervals $[0, \hat{T}_1]$ and $[\hat{T}_1, \hat{T}_2]$, we can provide the following bound:
    \begin{equation*}
    \begin{aligned}
        \mathbb{W}_2(\bmu^{0}_{\hat{T}_1 + \hat{T}_2 + t}; \delta_a) & \leq \frac{\TP{\x}}{\TP{\x} + \hat{T}_1 + \hat{T}_2 + t} (1 + 2\varepsilon) \rho + \frac{ \hat{T}_1 }{\TP{\x} + \hat{T}_1 + \hat{T}_2 + t} (1 + 2.5\varepsilon)\rho \\
        &\quad + \frac{\hat{T}_2 }{\TP{\x} + \hat{T}_1 + \hat{T}_2 + t} (1 + 3\varepsilon)\rho + \frac{t}{\TP{\x} + \hat{T}_1 + \hat{T}_2 + t} (1 - \vartheta) \rho.
    \end{aligned}
    \end{equation*}
    
    After regrouping its terms, the expression above is equal to:
    
    \begin{equation*}
    \begin{aligned}
        \rho \; + \; & \frac{2\TP{\x}\varepsilon + 2.5 \hat{T}_1 \varepsilon  + 3\hat{T}_2\varepsilon - t\vartheta }{\TP{\x} + \hat{T}_1 + \hat{T}_2 + t} \rho \\
        &\leq \rho + \frac{2(\TP{\x} + \hat{T}_1  + \hat{T}_2 + t) + \hat{T}_2 + 0.5 \hat{T}_1 - t\left(2 + \frac{\vartheta}{2\varepsilon} \right) }{\TP{\x} + \hat{T}_1 + \hat{T}_2 + t}\varepsilon\rho.
    \end{aligned}
    \end{equation*}
Thus, we have the following bound:    
    \begin{equation*}
    \begin{aligned}
        \mathbb{W}_2(\bmu^{0}_{\hat{T}_1 + \hat{T}_2 + t}; \delta_a) \leq (1 + 2\varepsilon) \rho + \frac{\hat{T}_2 + 0.5\hat{T}_1  - t\left(2 + \frac{\vartheta}{2\varepsilon} \right)}{T_{\text{st}}^\rho + \hat{T}_1 + \hat{T}_2 + t} \varepsilon \rho.
    \end{aligned}
    \end{equation*}
    
    We just need to choose $\hat{T}_3$ big enough such that $\Big(\hat{T}_2 + 0.5 \hat{T}_1 - \hat{T}_3 \left(2 + \frac{\vartheta}{2\varepsilon} \right) \Big)< 0$. We complete the proof by choosing $T_1^\rho = \hat{T}_1 + \hat{T}_2 + \hat{T}_3$.

\subsubsection{Proof of Lemma \ref{lm:tau_0_<=T1}: Attraction of the stochastic process towards~\texorpdfstring{$a$}{a}}

    According to Lemma \ref{lm:X0_conv}, there exists a finite time $T_1^\rho$ such that, for any $\varepsilon, \vartheta$ and for any $\x \in \mathcal{C}^1_{\varepsilon, \rho}$, the deterministic process $\X^{0}_{T_1^\rho} \in B_{(1 - \vartheta)\rho}(a)$, where $\mathcal{C}^1_{\varepsilon, \rho}$ is defined in \eqref{eq:def:set_C_1}.
    
    Let us define $$\Psi^{x_0} := \{\psi \in C([0, T_{1}^\rho]; \R^d), \psi(0) = x_0, \psi(s) \in \Cl(G \setminus B_\rho(a)) \; \forall s \leq T_1^\rho\}.$$ The following inclusion of the events holds: $\{\tau_0 \geq T_1^\rho\} \subset \{\X^{\x, \sigma} \in \Psi^{\LP{\x}}\}$ (the definition of $\tau_0$ was presented on page~\pageref{eq:def_of_tau_theta_0}). Let $\Psi = \bigcup_{x_0 \in G} \Psi^{x_0}$. Note that $\Psi$ is a closed set and the following enlargement of $\mathcal{C}^1_{\varepsilon, \rho}$: $\overline{\mathcal{C}}^1_{\varepsilon, \rho}:= \{\x \in \mathfrak{X}: T_{\text{st}}^\rho \leq \TP{\x} \leq \infty, \LP{\x} \in \overline{G}, \text{ and}, \OM{\x} \in \mathbb{B}_{(1 + 2\varepsilon)\rho}(\delta_a)\}$ is a compact set by Lemma \ref{lm:K_is_compact}. By Lemma~\ref{lm:LDP_compact_init_cond} and the inclusion of the events, we get
    \begin{equation*}
    \begin{aligned}
        \limsup_{\sigma \to 0} \frac{\sigma^2}{2} &\log \sup_{\x \in \mathcal{C}^1_{\varepsilon, \rho}} \Prob_{\x}(\tau_0 > T_1^\rho) \\
        &\leq \limsup_{\sigma \to 0} \frac{\sigma^2}{2} \log \sup_{\x \in \overline{\mathcal{C}}^1_{\varepsilon, \rho}} \Prob_{\x}(\X^{\sigma} \in \Psi ) \leq - \inf_{\x \in \overline{\mathcal{C}}^1_{\varepsilon, \rho}} \inf_{\phi \in \Psi} I_{T_1^\rho}^{\x}(\phi).        
    \end{aligned}
    \end{equation*}
    
    What is left to prove is that 
    \begin{equation}\label{eq:aux:I_T>0}
        \inf_{\x \in \overline{\mathcal{C}}^1_{\varepsilon, \rho}} \inf_{\phi \in \Psi} I_{T_1^\rho}^{\x}(\phi) > 0.
    \end{equation}
    We can not apply Corollary \ref{cor:inf_I_over_compact} directly since $\Psi$ is not necessarily compact. Let $\Psi^{\x}_1 := \{\psi \in \Psi: I_{T_1^\rho}^\x(\psi) \leq 1\}$. Of course, if all $\Psi^{\x}_1$ are empty, then expression in \eqref{eq:aux:I_T>0} is indeed strictly greater than 0. If it is not the case, then the sets $\{\Psi^\x_1\}_{\x \in \overline{\mathcal{C}}^1_{\varepsilon, \rho}}$ satisfy conditions of Lemma \ref{lm:A_x_compact}, and, therefore, union $\Psi_1 := \bigcup_{\x \in \overline{\mathcal{C}}^1_{\varepsilon, \rho}} \Psi^\x_1$ is precompact, and $\Psi_2 = \overline{\Psi_1}$ is a compact set. Thus, by Corollary \ref{cor:inf_I_over_compact}, there exists some $\x^* \in \overline{\mathcal{C}}^1_{\varepsilon, \rho}$, $\phi^* \in \Psi_2$ such that 
    \begin{equation*}
        \inf_{\x \in \overline{\mathcal{C}}^1_{\varepsilon, \rho}} \inf_{\phi \in \Psi_2} I_{T_1^\rho}^{\x}(\phi) = I_{T_1^\rho}^{\x^*}(\phi^*).
    \end{equation*}
    By Lemma \ref{lm:X0_conv}, for any $\x \in \overline{\mathcal{C}}^1_{\varepsilon, \rho}$, the corresponding deterministic trajectories converge inside $B_{(1 - \vartheta)\rho}(a)$ in time $T_1^\rho$ for some $\vartheta$. Thus, for any $\psi \in \Psi$: $\|\psi - \X^{0}\|_{\infty} > \vartheta$ for any $\x \in \overline{\mathcal{C}}^1_{\varepsilon, \rho}$. The same bound obviously holds for $\Psi_2$ since it is a closure of some subset of $\Psi$. Since the deterministic trajectories are the unique functions in $C([0, T_1^\rho]; \R^d)$ for which $I^\x_{T_1^\rho} (\X^{0}) = 0$, we conclude that
    \begin{equation*}
        \inf_{\x \in \overline{\mathcal{C}}^1_{\varepsilon, \rho}} \inf_{\phi \in \Psi} I_{T_1^\rho}^{\x}(\phi) > I_{T_1^\rho}^{\x^*}(\phi^*) > 0.
    \end{equation*}
    And this concludes the proof of the lemma.

\subsubsection{Proof of Lemma \ref{lm:rest_inside_annulus_prob}: Behavior in the annulus between \texorpdfstring{$B_{\rho}(a)$}{Brho(a)} and \texorpdfstring{$\partial G$}{boundary of G}}

   The idea of the proof is to show that, since $T_1^\rho$ represents the time in which the noiseless process converges inside $B_{\rho}(a)$, after each time interval of length $T_1^\rho$ it should be more and more unlikely that diffusion $\X^\sigma$ did not follow the deterministic path even once.

    We introduce the following set of functions whose path stay inside the annulus $\Cl(G\setminus B_\rho(a))$:
    \begin{equation*}
        \Psi_t := \Big\{\psi \in C([0, t]; \R^d): \psi_s \in \Cl(G\setminus B_\rho(a))  \text{ and } \frac{1}{s} \int_0^s \delta_{\psi_{u}} \dd{u} \in \mathbb{B}_{(1 + 2\varepsilon)\rho}(\delta_a)\; \forall s \leq t \Big\}.
    \end{equation*}
    Obviously, for any $\x \in \mathcal{C}^1_{\varepsilon, \rho}$, the following inclusion of events takes place $\{t < \tau_0 < \gamma \} \subset \{\X^{\sigma} \in \Psi_t\}$. By Lemma \ref{lm:LDP_compact_init_cond} for any $t<\infty$:
    \begin{equation*}
        \lim_{\sigma \to 0} \frac{\sigma^2}{2} \log \sup_{\x \in \mathcal{C}^1_{\varepsilon, \rho}} \Prob_{\x}(\X^\sigma \in \Psi_t) \leq - \inf_{\psi \in \Psi_{t}} I_{t}(\psi),
    \end{equation*}
    where $I_{t}(\psi) := \inf_{\x \in \mathcal{C}^1_{\varepsilon, \rho}} I_t^{\x}(\psi)$. Therefore, it is enough to show that $\lim_{t \to \infty} \inf_{\psi \in \Psi_t} I_{t} (\psi) = \infty$. Let us assume that it is not true and there exists $M < \infty$ such that for any $n$ there exists some function $\psi^n \in \Psi_{n T_1^\rho}$ such that $I_{nT_1^\rho}(\psi^n) \leq M$, where $T_1^\rho$ is defined by Lemma \ref{lm:tau_0_<=T1}. Then, we can separate the path of $\psi^{n}$ into $n$ parts and establish the following lower bound:
    
    \begin{equation}\label{eq:aux:bound_by_M}
        M \geq I_{nT_1^\rho}(\psi^n) \geq \sum_{k=0}^{n - 1} I^{\x_k^n}_{T_1^\rho}(\psi^{n, k}) \geq n \min_{k \leq n} I_{T_1^\rho}^{\x_k^n} (\psi^{n, k}),
    \end{equation}
    where $\x^n_k \in \mathfrak{X}$ are such that:
    \begin{equation*}
    \begin{aligned}
        \TP{\x_k^n} &= \TP{\x_0} + k T_1^\rho, \\
        \LP{\x_k^n} &= \psi^{n, k - 1}(k T_1^\rho), \\
        \OM{\x_k^n} &= \frac{\TP{\x_0}}{\TP{\x_0} + k T_1^\rho} \OM{\x_0} + \frac{1}{\TP{\x_0} + k T_1^\rho}\int_0^{k T_1^\rho} \delta_{\psi_s^n} \dd{s}; 
    \end{aligned}
    \end{equation*}
    and $\psi^{n, k} \in C([0, T_1^\rho]; \R^d)$ are the corresponding peaces of $\psi^n$ of length $T_1^\rho$, i.e., $\psi^{n, k}(s) = \psi^n(kT_1^\rho + s)$ for $s \in [0, T_1^\rho]$. 
    
    For equation \eqref{eq:aux:bound_by_M} to hold for any $n$ it is required that there is a sequence of functions $\{\psi^{n, k_n}\}_{n = 1}^\infty$ such that $I^{\x^n_{k_n}}_{T_1^\rho} (\psi^{n, k_n}) \to 0$. Thus, after some $n_0$ all of these functions belong to the set 
    \begin{equation*}
        \Phi = \bigcup_{\x \in \mathcal{C}^1_{\varepsilon, \rho}} \Phi^{\x},
    \end{equation*}
    where $\Phi^{\x} := \{f \in C([0, T_1^\rho]): I^{\x}(f) \leq 1\}$. It also means that $\inf_{\phi \in \Phi} I_{T_1^\rho} (\phi) = 0$. We use the same logic as in the proof of Lemma~\ref{lm:tau_0_<=T1} and Lemma \ref{lm:theta_0_>_T2}. If all the sets $\{\Phi^\x\}_{\x \in \mathcal{C}^1_{\varepsilon, \rho}}$ are empty, then we get the contradiction with the fact that infimum of rate functions over this set should be equal to 0. In the other case, we can apply Lemma~\ref{lm:A_x_compact} to conclude that $\Phi$ is precompact set, which makes $\Phi_1 := \overline{\Phi}$ to be a compact set. By Corollary \ref{cor:inf_I_over_compact}, there exist $x^* \in \mathcal{C}^1_{\varepsilon, \rho}$ and $\phi^* \in \Phi_1$ such that
    \begin{equation*}
        \inf_{\phi \in \Phi_1} \inf_{\x \in \mathcal{C}^1_{\varepsilon, \rho}} I^\x_{T_1^\rho}(\phi) = I^{\x^*}_{T_1^\rho}(\phi^*).
    \end{equation*}
    By Lemma \ref{lm:conv_to_delta_a}, for any $\x \in \mathcal{C}^1_{\varepsilon, \rho}$, the deterministic processes $\X^{0}$ converge inside $B_{(1 - \vartheta)\rho}(a)$ for some $\vartheta > 0$ in time $T_1^\rho$. Thus, for any $\phi \in \Psi_{T_1^\rho}$ (and as a consequence in $\Phi_1$): $\|\phi - \X^{0}\|_{\infty} \geq \vartheta$. That means that $I^{\x^*}_{T_1^\rho}(\phi^*)$ is strictly positive, which contradicts existence of $M < \infty$ in \eqref{eq:aux:bound_by_M} and, thus, proves the lemma.

\subsubsection{Proof of Lemma \ref{lm:gamma_gr_tau}: Stabilization of the occupation measure}

The proof of this result requires an additional Lemma \ref{lm:theta_0_>_T2}. This lemma claims that for any choice of constant $T_2 > 0$ we can make $\sigma$ small enough such that with high probability the stochastic process $\X^{\sigma}$ spends inside a small neighbourhood of the point of attraction $a$ at least time $T_2$.

We remind that $\theta_{0} := \inf\{t:\; \X_t^{\sigma} \in S_{(1 + \varepsilon)\rho}(a)\}$ is the time that the process starting in some point inside $B_\rho(a)$ spends inside the ball $B_{(1 + \varepsilon)\rho}(a)$. Let 
\begin{equation*}
    \mathcal{C}^4_{\varepsilon, \rho} := \{ \x \in \mathfrak{X}: T^\rho_{\text{st}} \leq \TP{\x} \leq \infty;  \OM{\x} \in \mathbb{B}_{(1 + 2\varepsilon)\rho}(\delta_a); \LP{\x} \in B_{\rho}(a)\} 
\end{equation*}

Then the following lemma holds.

\begin{lemma}\label{lm:theta_0_>_T2}
    There exist $\epsilonbar, \rhobar > 0$ such that or any $0 < \rho < \rhobar$, for any $0 < \varepsilon < \epsilonbar$ and for any constant $T_2 > 0$ the following limit holds
    \begin{equation}
         \lim_{\sigma \to 0}\sup_{\x \in \mathcal{C}^4_{\varepsilon, \rho}}\Prob_{\x} (\theta_0 < T_2 ) = 0.
    \end{equation}
    
\end{lemma}

\begin{proof}
    The proof of the following lemma follow the same logic as the one of Lemma~\ref{lm:tau_0_<=T1}.
    
    By Lemma \ref{lm:X0_conv}, the deterministic process $\X^{0}$ starting with any initial conditions $\x \in \mathfrak{X}$: $T^\rho_{\text{st}} \leq \TP{\x} \leq \infty$, $\LP{\x} \in B_{\rho}(a)$, and $\OM{\x} \in \mathbb{B}_{(1 + 2\varepsilon)\rho}(\delta_a)$ stays inside $B_{\rho}(a)$ for any $t \geq 0$. 
    
    Fix $T_2$ and define $$\Psi^{y}:= \{\psi \in C([0, T_2]; \R^d): \psi(0) = y, |\psi(s) - a| \geq (1 + \varepsilon)\rho, \text{ for some } s \leq T_2\}.$$ Obviously, by definition of $\Psi^{y}$, the following inclusion holds: $\{\theta_0 \geq T_2\} \subset \{\X^{\sigma} \in \Psi^{\LP{\x}}\}$. Let $\Psi := \Cl(\bigcup_{y \in B_{\rho}(a)} \Psi^{y})$. Note that, by Lemma \ref{lm:K_is_compact}, the set $\mathcal{C}^4_{\varepsilon, \rho}$ is compact. Then, by Lemma \ref{lm:LDP_compact_init_cond}, we get
	\begin{equation*}
			\limsup_{\sigma \to 0} \frac{\sigma^2}{2} \log \sup_{\x \in \mathcal{C}^4_{\varepsilon, \rho}} \Prob_{\x} (\theta_0 > T_2) \leq - \inf_{\x \in \mathcal{C}^4_{\varepsilon, \rho}} \inf_{\phi \in \Psi} I_{T_2}^\x(\phi).
	\end{equation*}
	
	We have to show now that
	\begin{equation*}
		\inf_{\x \in \mathcal{C}^4_{\varepsilon, \rho}} \inf_{\phi \in \Psi} I_{T_2}^\x (\phi) > 0.
	\end{equation*}
	
	In order to apply Corollary \ref{cor:inf_I_over_compact}, consider $\Psi_1^\x := \{\psi \in \Psi: I_{T_2}^\x(\psi) \leq 1\}$. If all the sets are empty, then the inequality above is trivially satisfied. If not, the family of sets $\{\Psi_1^\x\}_{\x \in \mathcal{C}^4_{\varepsilon, \rho}}$ satisfies conditions of Lemma \ref{lm:A_x_compact}. Thus, the union $\bigcup_{\x \in \mathcal{C}^4_{\varepsilon, \rho}} \Psi_1^\x$ is precompact and $\Psi_2 := \overline{\Psi_1}$ is a compact set. Thus, by Corollary \ref{cor:inf_I_over_compact}, there exist some $\x^* \in \mathcal{C}^4_{\varepsilon, \rho}$ and $\phi^* \in \Psi_2$ such that 
	\begin{equation*}
		\inf_{\x \in \mathcal{C}^4_{\varepsilon, \rho}} \inf_{\phi \in \Psi} I_{T_2}^\x (\phi) = I_{T_2}^{\x^*} (\phi^*).
	\end{equation*}
    As was pointed out before, $\X^{0}$ starting with initial conditions $\x \in \mathcal{C}^4_{\varepsilon, \rho}$ stays inside $B_{\rho}(a)$ for any $0 \leq t \leq T_2$. By definition of $\Psi$, any function $\X^{0}$ has a positive distance of at least $\varepsilon \rho$ to $\Psi$. The same holds for $\Psi_2$. Since $\X^{0}$ are unique minimizers of $I_{T_2}^\x$ for corresponding initial conditions $\x$ and $I_{T_2}^\x(\X^{0}) = 0$,
	\begin{equation*}
		I_{T_2}^{\x^*} (\phi^*) > 0.
	\end{equation*}
	This completes the proof.
\end{proof}

Now we are ready to prove Lemma \ref{lm:gamma_gr_tau} itself. 

\begin{proof}[Proof of Lemma \ref{lm:gamma_gr_tau}]
    \textit{Dynamics of $\bmu_t^\sigma$}. We separate the path of $\bmu_t^\sigma$ into two parts: when it belongs to the ball of radius $\rho$ and when big excursions occur. Let 
    \begin{equation}\label{eq:def:set_C_4}
        \x \in \mathcal{C}^2_\rho := \{\x \in \mathfrak{X}: T_{\text{st}}^\rho \leq \TP{\x} \leq \infty, \OM{\x} \in \mathbb{B}_{\rho}(\delta_a), \text{ and } \LP{\x} \in B_{\rho}(a)\}.    
    \end{equation}
    Let $\bmu_0 = \OM{\x}$ and $t_0 = \TP{\x}$. Consider the following equations for $t \in [\theta_{m}, \tau_{m + 1}]$ (see page \pageref{eq:def_of_tau_theta_0} for definition).
    \begin{equation*}
            \mathbb{W}_2(\bmu_t^\sigma; \delta_a) \leq \frac{t_0}{t_0 + t} \mathbb{W}_2(\bmu_0; \delta_a) + \frac{t}{t_0 + t} \mathbb{W}_2 \left(\frac{1}{t} \int_0^t\delta_{X_s} \dd{s}; \delta_a \right).
    \end{equation*}
By definition of $\tau_k$ and $\theta_k$, the expression above is less or equal to: 
    \begin{equation*}
        \begin{aligned}
            & \frac{t_0}{t_0 + t} \mathbb{W}_2(\bmu_0; \delta_a) +  \sum_{k=1}^{m - 1} \frac{\tau_{k + 1} - \theta_k}{t_0 + t} \mathbb{W}_2 \left(\frac{1}{\tau_{k + 1} - \theta_k} \int_{\theta_k}^{\tau_{k + 1}}\delta_{X_s} \dd{s}; \delta_a \right) \\
            &\quad + \sum_{k = 1}^{m} \frac{\theta_{k} - \tau_k}{t_0 + t} \mathbb{W}_2 \left(\frac{1}{\theta_{k} - \tau_k} \int_{\tau_k}^{\theta_{k}} \delta_{X_s} \dd{s}\!; \delta_a \right) + \frac{t - \theta_{m}}{t_0 + t} \mathbb{W}_2\left( \frac{1}{t - \theta_m} \int_{\theta_{m}}^{t} \delta_{X_s} \dd{s}\!; \delta_a \right).
        \end{aligned}
    \end{equation*}
    That gives us:
    
    \begin{equation*}
        \mathbb{W}_2(\bmu_t^\sigma; \delta_a) \leq \frac{t_0}{t_0 + t} \rho + \left(\sum_{k=1}^{m - 1}\frac{\tau_{k + 1} - \theta_k}{t_0 + t} + \frac{t - \theta_{m}}{t_0 + t} \right)\!\!R + \left(\sum_{k = 1}^{m} \frac{\theta_{k} - \tau_k}{t_0 + t} \right)\!(1 + \varepsilon)\rho,
    \end{equation*}
    
    where $R := \sup_{z \in \partial G} |z - a|$ is the maximal distance between the point $a$ and the frontier of $G$. Let us define $\mathcal{T}_{\text{out}}(m) := \sum_{k=1}^m(\tau_{k+1} - \theta_k)$ that is the total amount of time spent significantly outside of the ball $B_{\rho}(a)$ after $m$ full exits, as well as $\mathcal{T}_{\text{in}}(m) := \sum_{k=1}^m(\theta_{k} - \tau_k)$. Taking into account this notation, for $t \in [\theta_{m}, \tau_{m + 1}]$, we express
    \begin{equation}\label{eq:aux:mu_t_upper_bound}
    \begin{aligned}
         \mathbb{W}_2(\bmu_t^\sigma; \delta_a) &\leq \frac{t_0}{t_0 + t} \rho + \frac{\mathcal{T}_{\text{out}}(m)}{t_0 + t}R + \frac{t - \mathcal{T}_{\text{out}}(m - 1)}{t_0 + t}(1 + \varepsilon)\rho \\
         & \leq (1 + \varepsilon)\rho  + \frac{\mathcal{T}_{\text{out}}(m)}{t_0 + t}R \leq (1 + \varepsilon)\rho  + \frac{\mathcal{T}_{\text{out}}(m)}{\mathcal{T}_{\text{in}}(m)}R.
    \end{aligned}
    \end{equation}
    Here we emphasize that $\bmu_t^\sigma$, $\tau_k$, $\theta_k$ and, as a consequence, $\mathcal{T}_{\text{out}}(m)$ depend also on elementary event $\omega$. Now, in order to prove that $\gamma > \tau_G^\sigma \wedge \exp{\frac{2(H + 1)}{\sigma^2}}$ with high probability, it suffices to show that $\mathcal{T}_{\text{out}}(m)$ constitutes such a small part of $t$ that it will not be able to move the occupation measure $\bmu_t^\sigma$ significantly away from $\delta_a$. 
    
    \textit{Control of} $\mathcal{T}_{\text{in}}$. By Lemma \ref{lm:rest_inside_annulus_prob} if we choose $\sigma$ small enough, there exists $T_1^\rho > 0$ big enough such that
    \begin{equation}\label{eq:aux:tau_0>T_1}
        \sup_{\x \in \mathcal{C}^2_\rho} \Prob_\x (\tau_0 > T_1^\rho, \gamma > \tau_0) < \exp{-\frac{8 (H + 1)}{\sigma^2}}.
    \end{equation}
    
    Define $T_2^\rho > 0$ as a number that is big enough such that:
    \begin{equation}\label{eq:aux:T_1/T_2R}
        \frac{T_1^\rho}{T_2^\rho} R < \varepsilon \rho. 
    \end{equation}
    
    We recall that Lemma~\ref{lm:theta_0_>_T2} establishes the following asymptotic behaviour for small $\sigma$:
    \begin{equation}\label{eq:aux:theta_0>T_2}
        \sup_{\x \in \mathcal{C}^3_{\varepsilon, \rho}}\Prob_{\x} (\theta_0 < T_2^\rho) = o_\sigma,
    \end{equation}
    where $o_\sigma$ is an infinitesimal w.r.t. $\sigma$. Thus, we can get the following lower bound for time spent inside the ball $B_{(1 + \varepsilon) \rho}(a)$. For any $\x \in \mathcal{C}^2_\rho$,
    \begin{equation}\label{eq:aux:T_in<m/2T2}
        \begin{aligned}
            \Prob_{\x} &(\mathcal{T}_{\text{in}}(m) < \frac{m}{2}T_2^\rho, \gamma \geq \tau_{m + 1}) \\
            &\leq \Prob_{\x} ( \#\{i \leq m: \theta_{i} - \tau_i < T_2^\rho(\sigma)\} > \frac{m}{2}, \gamma \geq \tau_{m + 1}) \\
            & \leq \sum_{k = \lceil \frac{m}{2} \rceil }^m \sum_{(i_1, \dots, i_k)}\!\! \Prob_{\x}\! \left( \bigcap_{j} \{\theta_{i_j} - \tau_{i_j} < T_2^\rho(\sigma)\}, \gamma \geq \tau_{m + 1} \right),
        \end{aligned}
    \end{equation}
    where the summation is taken with respect to all tuples of the form $(i_1,~ \dots, i_k)$ for $i_1 < \dots < i_k$. Note that we can roughly estimate the number of such tuples to be less than $2^m$. We also emphasize that each respective probability can be expressed as
    \begin{equation*}
    \begin{aligned}
        &\Prob_{\x}\! \left(\! \bigcap_{j = 1}^k \{\theta_{i_j} - \tau_{i_j} < T_2^\rho\}, \gamma \geq \tau_{m + 1} \right) \\
        &= \prod_{j = 1}^k \Prob_{\x} \left( \theta_{i_j} - \tau_{i_j} < T_2^\rho, \gamma \geq \tau_{m + 1} \Big| \bigcap_{z \leq j} \{\theta_{i_m} - \tau_{i_z} < T_2^\rho\} \! \right)\leq \left(\sup_{\x \in \mathcal{C}^3_{\varepsilon, \rho}}\!\!\! \Prob_{\x} \big(\theta_0 < T_2^\rho\big) \right)^k \!\!\!.
    \end{aligned}
    \end{equation*}
    
    By equation \eqref{eq:aux:theta_0>T_2} we conclude the following bound for probability \eqref{eq:aux:T_in<m/2T2}. For any $m \in \N $ and $\x \in \mathcal{C}^2_\rho$,
    \begin{equation}\label{eq:aux:Tin_final_bound}
    \begin{aligned}
        \Prob_{\x} &(\mathcal{T}_{\text{in}}(m) < \frac{m}{2}T_2^\rho, \gamma \geq \tau_{m + 1}) \\
        &\leq \sum_{k = \lceil \frac{m}{2} \rceil }^m 2^m \left(\sup_{\x \in \mathcal{C}^3_{\varepsilon, \rho}} \Prob_{\x} \big(\theta_0 < T_2^\rho\big) \right)^k \leq 2^m\frac{o_{\sigma}^{\lceil\frac{m}{2}\rceil}(o_\sigma^{\lceil\frac{m}{2}\rceil}- 1)}{o_\sigma - 1} \leq \frac{o_{\sigma}^{\lceil\frac{m}{2}\rceil}}{1 - o_\sigma}.  
        \end{aligned}
    \end{equation}
    
    \textit{Control of the number of excursion}. Inequality \eqref{eq:aux:Tin_final_bound} also provides us with the following upper bound for $\tau_m$. Let $m^* := \lceil\frac{2}{T_2^\rho} \rceil\exp{\frac{2(H + 1)}{\sigma^2}}$ and consider
    \begin{equation}\label{eq:aux:m_final_bound}
        \begin{aligned}
            \Prob_\x \left(\tau_{m^*} < \exp{\frac{2(H + 1)}{\sigma^2}} \right) &\leq \Prob_\x \left(\mathcal{T}_{\text{in}}(m^*) < \exp{\frac{2(H + 1)}{\sigma^2}} \right)\\
            & \leq \frac{o_{\sigma}^{\lceil\frac{m^*}{2}\rceil}}{1 - o_\sigma}.
        \end{aligned}
    \end{equation}
    Note that $m^*$ tends to infinity with $\sigma \to 0$. It means that the probability in \eqref{eq:aux:m_final_bound} tends to zero, which provides us with an asymptotic upper bound for $\tau_{m}$ knowing that $m$ is large enough. Inequality \eqref{eq:aux:m_final_bound} in particular means that the probability that there were more than $m^*$ (that grows exponentially fast with $\sigma$) excursions before time $\exp{\frac{2(H + 1)}{\sigma^2}}$ is very small.
    
    \textit{Control of $\mathcal{T}_{\text{out}}$}. For time spent significantly outside of $B_{\rho}(a)$ --  $\mathcal{T}_{\text{out}}(m)$ -- we provide the following simple bound. For any $\x \in \mathcal{C}^1_{\varepsilon, \rho}$,
    \begin{equation} \label{eq:aux:T_out_final_bound}
    \begin{aligned}
        \Prob_{\x}\big(\mathcal{T}_{\text{out}}(m) > m T_1^\rho, \gamma \geq \tau_m \big) &\leq \Prob_{\x} \big(\{\exists i \leq m : \tau_i - \theta_i > T_1^\rho \}, \gamma \geq \tau_m \big) \\
        & \leq m\sup_{\x \in \mathcal{C}^1_{\varepsilon, \rho}} \Prob_{\x} \big( \tau_0 > T_1^\rho \big) \\
        &\leq m \exp{-\frac{8(H + 1)}{\sigma^2}},
    \end{aligned}
    \end{equation}
    where we obtain the last inequality by \eqref{eq:aux:tau_0>T_1}.
    
    \textit{Control of $\gamma$}. Note that, by definition of $\gamma$, $\mathbb{W}_2(\bmu_\gamma^\sigma; \delta_a) \geq (1 + 2\varepsilon)\rho$ a.s. Consider the following inequalities. For any $\x \in \mathcal{C}^2_\rho$, where $\mathcal{C}^2_\rho$ is defined in \eqref{eq:def:set_C_4}, we have:
    \begin{equation*}
        \begin{aligned}
            \Prob_{\x} &\left(\gamma \leq \tau_{G}^\sigma \wedge \exp{\frac{2(H + 1)}{\sigma^2}} \right) \\
            &= \Prob_{\x} \left(\mathbb{W}_2(\bmu_\gamma^\sigma; \delta_a) \geq (1 + 2\varepsilon)\rho ,\gamma \leq \tau_{G}^\sigma \wedge \exp{\frac{2(H + 1)}{\sigma^2}} \right) \\
            &\leq \sum_{m = 1}^{\infty} \Prob_\x \left(\mathbb{W}_2(\bmu_\gamma^\sigma; \delta_a) \geq (1 + 2 \varepsilon)\rho, \theta_m \leq \gamma \leq \tau_{m + 1} \wedge \exp{\frac{2(H + 1)}{\sigma^2}}\right) \\
            &\leq \sum_{m = 1}^{m^*} \Prob_\x \left(\mathbb{W}_2(\bmu_\gamma^\sigma; \delta_a) \geq (1 + 2\varepsilon)\rho, \theta_m \leq \gamma \leq \tau_{m + 1} \wedge \exp{\frac{2(H + 1)}{\sigma^2}} \right) \\
            & \quad\quad + \Prob_\x \left(\tau_{m^*} \leq \exp{\frac{2(H + 1)}{\sigma^2}} \right) =: A + B.
        \end{aligned}
    \end{equation*}
    Let us first deal with the term $B$. By \eqref{eq:aux:m_final_bound}, we have
    \begin{equation*}
        \begin{aligned}
            B \leq  \frac{o_{\sigma}^{\lceil\frac{m}{2}\rceil}}{o_\sigma - 1} \xrightarrow[\sigma \to 0]{} 0. 
        \end{aligned}
    \end{equation*}
    Consider now $A$. Separate each probability inside the sum the following way
    \begin{equation*}
    \begin{aligned}
       &\Prob_\x \left(\mathbb{W}_2(\bmu_\gamma^\sigma; \delta_a) \geq (1 + 2\varepsilon)\rho, \theta_m \leq \gamma \leq \tau_{m + 1} \wedge \exp{\frac{2(H + 1)}{\sigma^2}} \right) \\
        & \leq \sum_{k = m}^{m^*} \Prob_\x \Big(\mathbb{W}_2(\bmu_\gamma^\sigma; \delta_a) \geq (1 + 2 \varepsilon)\rho, \mathcal{T}_{\text{out}}(k) < k  T_1^\rho, \mathcal{T}_{\text{in}}(k) > \frac{k}{2}T_2^\rho, \gamma \in [\theta_k, \tau_{k+1}]\Big) \\
        & \quad + \sum_{k = m}^{m^*} \Prob_{\x} (\mathcal{T}_{\text{out}}(k) \geq k T_1^\rho, \gamma \geq \tau_k) + \sum_{k = m}^{m^*} \Prob_{\x}(\mathcal{T}_{\text{in}}(k) \leq \frac{k}{2}T_2^\rho, \gamma \geq \tau_k ).
    \end{aligned}
    \end{equation*}
    
    Note that all the probabilities of the following form $\Prob_\x \big(\mathbb{W}_2(\bmu_\gamma^\sigma; \delta_a) \geq (1 + 2\varepsilon)\rho$, $\mathcal{T}_{\text{out}}(k) < k T_1^\rho$, $\mathcal{T}_{\text{in}}(k) > \frac{k}{2}T_2^\rho$, $\gamma \in [\theta_k, \tau_{k+1}]\big)$ are equal to $0$ by \eqref{eq:aux:mu_t_upper_bound}. For the latter two sums we can use \eqref{eq:aux:Tin_final_bound} and \eqref{eq:aux:T_out_final_bound}. Finally, we can conclude that the twice summation of infinitesimals above gives us
    \begin{equation*}
    \begin{aligned}
       A \leq \frac{o_\sigma}{(1 - o_\sigma)^3} + (1 + m^*)^3 \exp{-\frac{8 (H + 1)}{\sigma^2}}.
    \end{aligned}
    \end{equation*}
    
    We remind that, by definition, $m^* = \lceil\frac{2}{T_2^\rho} \rceil\exp{\frac{2(H + 1)}{\sigma^2}}$. It means that $A$ is also bounded by some function that tends to 0. Combining the bounds above for $A$ and $B$ as well as Lemma \ref{lm:conv_to_delta_a} (decrease $\sigma$ if necessary such that $\exp{\frac{2(H + 1)}{\sigma^2}} > T_{\text{st}}^\rho$), we finally get
    \begin{equation*}
        \begin{aligned}
        \sup_{\x \in \mathcal{C}^2_\rho} \Prob_{\x}\left(\gamma \leq \tau_{G}^\sigma \wedge \exp{\frac{2(H + 1)}{\sigma^2}} \right) \xrightarrow[\sigma \to 0]{} 0.
        \end{aligned}
    \end{equation*}

\end{proof}

\subsubsection{Proof of Lemma \ref{lm:tau_0=tau_G}: Exit before nearing \texorpdfstring{$a$}{a}}

    First, we show that it is possible to establish a lower bound for $$\inf_{z \in \partial G}Q^\rho(x, z) := \frac{1}{2}\inf_{z \in \partial G}\inf_{t > 0}\inf_{\phi} I_{t}^{\x}(\phi),$$ where $\x \in \mathcal{C}^1_{\varepsilon, \rho}$ and  $\mathcal{C}^1_{\varepsilon, \rho} = \{ \x \in \mathfrak{X}: T_{\text{st}}^\rho \leq \TP{\x} \leq \infty; \OM{\x} \in \mathbb{B}_{(1 + 2\varepsilon)\rho}(\delta_a); \LP{\x} \in G \}$, and infimum is taken over functions $\phi \in C([0, t]; \R^d)$ such that $\phi_t = z$ and $\mu^\phi_{s} := \frac{1}{s} \int_0^s \delta_{\phi_{z}} \dd{z} \in \mathbb{B}_{(1 + 2\varepsilon)\rho}(\delta_a)$ for any $s \leq t$. Moreover, we show that this lower bound approaches $H$ with $\rho \to 0$. Indeed, for any $f \in C([0, T]; \R^d)$, such that $\mu^f_{t} \in \mathbb{B}_{(1 + 2\varepsilon)\rho}(\delta_a)$ for any $t \leq T$, and $f_0 \in G$, $f_T \in \partial G$:
    \begin{equation*}
        \begin{aligned}
            \frac{1}{4} &\int_0^T |\dot{f}_t + \nabla V(f_t) + \nabla F*\delta_a(f_t)|^2\dd{t}\\
            &= \frac{1}{4} \int_0^T |\dot{f}_t + \nabla V(f_t) + \nabla F * \mu^f_{t}(f_t)|^2 \dd{t} + \frac{1}{4} \int_0^T |\nabla F*\mu^f_{t}(f_t) - \nabla F*\delta_a(f_t)|^2\dd{t} \\
            &\quad+ \frac{1}{2} \int_0^T \big\langle \dot{f}_t + \nabla V(f_t) + \nabla F*\delta_a(f_t); \nabla F*\mu^f_{t}(f_t) - \nabla F*\delta_a(f_t) \big\rangle \dd{t} \\
            &\leq \frac{1}{4} \int_0^T |\dot{f}_t + \nabla V(f_t) + \nabla F * \mu^f_{t}(f_t)|^2 \dd{t} + \frac{\text{Lip}_{\nabla F}}{4}T(1 + 2\varepsilon)^2\rho^2 \\
            &\quad+ \frac{1}{2} \sqrt{\int_0^T |\dot{f}_t + \nabla V(f_t) + \nabla F * \mu^f_{t} (f_t)|^2\dd{t}} \cdot \text{Lip}_{\nabla F}(1 + 2\varepsilon)\rho.
        \end{aligned}
    \end{equation*}
    
    We can first take the infimum of both sides of the inequality above over all $f$ such that $\mu^f_{t} \in \mathbb{B}_{(1 + 2\varepsilon)\rho}(\delta_a)$ and $f_0 \in G$, $f_T \in \partial G$, then over all $T > 0$. Note that $H$ will be less than the infimum in the left-hand side. Thus, we conclude that
    \begin{equation*}
    \begin{aligned}
        H &\leq \inf_{z \in \partial G} Q^\rho (a, z) + \text{Lip}_{\nabla F}(1 + 2\varepsilon)\rho\sqrt{\inf_{z \in \partial G} Q^\rho (a, z)}\\
        & \leq Q^\rho(a, y) + \inf_{z \in \partial G} Q^\rho (y, z) + \text{Lip}_{\nabla F}(1 + 2\varepsilon)\rho\sqrt{\inf_{z \in \partial G} Q^\rho (a, z)},
    \end{aligned}
    \end{equation*}
    or, equivalently,
    \begin{equation}
        \inf_{z \in \partial G} Q^\rho (y, z) \geq H - C(\rho),
    \end{equation}
    where $C(\rho) > 0$ is some function of $\rho$, that tends to $0$, whenever $\rho \to 0$. 
    
    By Lemma \ref{lm:rest_inside_annulus_prob}, there exists time $T > 0$ big enough such that, $$\limsup_{\sigma \to 0} \frac{\sigma^2}{2} \log \sup \Prob_{\x}(\tau_0 > T, \gamma > \tau_0) < -H,$$ where supremum is taken over $\x \in \mathcal{C}^3_{\varepsilon, \rho} \subset \mathcal{C}^1_{\varepsilon, \rho}$, where $\mathcal{C}^1_{\varepsilon, \rho}$ and $\mathcal{C}^3_{\varepsilon, \rho}$ are defined in \eqref{eq:def:set_C_1} and \eqref{eq:def:set_C_2} respectively. Consider the following set:
    \begin{equation*}
        \Phi := \{\phi \in C([0, T]): \exists t \in [0, T]: \phi_t \in \partial G, \text{ and } \mu^\phi_{s} \in \mathbb{B}_{(1 + 2\varepsilon)\rho}(\delta_a) \; \forall s \in [0, T]\}.
    \end{equation*}
    By the definition of $Q^\rho(x, z)$ and proved above facts,
    \begin{equation*}
        \inf_{\substack{T_{\text{st}}^\rho \leq \TP{\x} \leq \infty \\ \LP{\x} \in S_{(1 + \varepsilon)\rho}(a)\\ \OM{\x} \in \mathbb{B}_{(1 + 2\varepsilon)\rho}(\delta_a)\\ }} \inf_{\phi \in \Phi} I_{T}^{\x}(\phi) \geq \inf_{ \substack{y \in S_{(1 + \varepsilon)\rho}(a) \\ z \in \partial G}} Q^\rho(y, z) \geq H - C(\rho).
    \end{equation*}
    By Lemma \ref{lm:LDP_compact_init_cond}, 
    \begin{equation*}
    \begin{aligned}
        \limsup_{\sigma \to 0} \frac{\sigma^2}{2} \log \sup_{ \substack{T_{\text{st}}^\rho \leq \TP{\x} \leq \infty \\ \LP{\x} \in S_{(1 + \varepsilon)\rho}(a)\\ \OM{\x} \in \mathbb{B}_{(1 + 2\varepsilon)\rho}(\delta_a)}} \sup_{\phi \in \Phi}\Prob_{\x}(\X^\sigma \in \Phi) &\leq - \inf_{\substack{T_{\text{st}}^\rho \leq \TP{\x} \leq \infty \\ \LP{\x} \in S_{(1 + \varepsilon)\rho}(a)\\ \OM{\x} \in \mathbb{B}_{(1 + 2\varepsilon)\rho}(\delta_a)}} \inf_{\phi \in \Phi} I_{T}^{\x}(\phi) \\
        & \leq - H + C(\rho).        
    \end{aligned}
    \end{equation*}
    And, since $\Prob_{\x}(\tau_0 = \tau_G^\sigma, \tau_0 < \gamma) \leq \Prob_{\x}(\tau_0 < T, \tau_0 < \gamma) + \Prob_{\x}(\X^\sigma \in \Phi)$, we get:
    \begin{equation*}
         \limsup_{\sigma \to 0} \frac{\sigma^2}{2} \log \sup_{\substack{T_{\text{st}}^\rho \leq \TP{\x} \leq \infty \\ \LP{\x} \in S_{(1 + \varepsilon)\rho}(a)\\ \OM{\x} \in \mathbb{B}_{(1 + 2\varepsilon)\rho}(\delta_a)}} \sup_{\phi \in \Phi} \Prob_{\x}(\tau_0 = \tau_G^\sigma, \tau_0 < \gamma) \leq -H + C(\rho),
    \end{equation*}
    which proves the lemma by taking $\rho \to 0$.

\subsubsection{Proof of Lemma \ref{lm:T(eps,c)}: Control of dynamics for small time intervals}

    Let us fix some $0 < \rho, \varepsilon < 1$ and let us investigate the dynamics of $|\X^\sigma - \LP{\x}|$. We plug in the expression describing $\X^\sigma_t$ from \eqref{eq:SID_main_sys_gen}, use Lipschitzness of $\nabla V$ and $\nabla F$ (Assumption \ref{A-1}.\ref{A-1.Lip}) as well as $\nabla V(a) = 0$, $\nabla F(0) = 0$, and express for any $t$:
    
    \begin{equation*}
        \begin{aligned}
            |\X_t^\sigma - \LP{\x}| &\leq \sigma|W_t| + \text{Lip}_{\nabla V} \int_0^t |\X_s^\sigma - \LP{\x}|\dd{s} + \text{Lip}_{\nabla V} |\LP{\x} - a|t \\
            &\quad + \int_0^t\frac{\TP{\x}}{\TP{\x} + s}\text{Lip}_{\nabla F} \Big( |\X_s^\sigma - \LP{\x}| + \int|z - \LP{\x}|\OM{\x}(\dd{z}) \Big)\dd{s} \\
            &\quad + \int_0^t\frac{1}{\TP{\x} + s} \text{Lip}_{\nabla F} \int_0^s \Big(|\X_s^\sigma - \LP{\x}| + |\X_u^\sigma - \LP{\x}| \Big)\dd{u}\dd{s}. 
        \end{aligned}
    \end{equation*}
    Use the Jensen's inequality, integrate the second part of the last integral over $s$, and get: 
    \begin{equation*}
        \begin{aligned}        
            |\X_t^\sigma - \LP{\x}| &\leq \sigma|W_t| + \text{Lip}_{\nabla V} \int_0^t |\X_s^\sigma - \LP{\x}|\dd{s} + \text{Lip}_{\nabla V} |\LP{\x} - a|t \\
            &\quad + \text{Lip}_{\nabla F} \int_0^t |\X^\sigma_s - \LP{\x}|\dd{s} + \text{Lip}_{\nabla F}t \Big( \mathbb{W}_2(\OM{\x}; \delta_a) + |\LP{\x} - a| \Big)\\
            &\quad + \text{Lip}_{\nabla F} \Big(\int_0^t|\X_s^\sigma - \LP{\x}|\dd{s} + \int_0^t \log(\frac{\TP{\x} + t}{\TP{\x} + s}) |\X_s^\sigma - \LP{\x}|\dd{s}  \Big).
        \end{aligned}
    \end{equation*}
    Finally, introducing $R:= \sup_{z \in \partial G} |z - a|$ and using the fact that $\x$ is assumed, in the lemma, to belong to $\mathcal{C}^1_{\varepsilon, \rho}$, we get the bound: 
    \begin{equation*}
        \begin{aligned}             
            |\X_t^\sigma - \LP{\x}| &\leq \sigma |W_t| + \Big((\text{Lip}_{\nabla V} + \text{Lip}_{\nabla F})R  + \text{Lip}_{\nabla F}(1 + 2\varepsilon)\rho \Big)t \\
            &\quad + \left(\text{Lip}_{\nabla V} + 2\text{Lip}_{\nabla F} + \log(\frac{\TP{\x} + t}{\TP{\x}}) \right) \int_0^t |\X_s^\sigma - \LP{\x}|\dd{s}.
        \end{aligned}
    \end{equation*}
    We apply Gr\"{o}nwall's inequality and the fact that, by the chose of $\rho$ and $\varepsilon$, we have $(1 + 2\varepsilon)\rho < 3$, and get:
    \begin{equation}\label{eq:|X_t-x|_1}
    \begin{aligned}
        |\X_t^\sigma - &\LP{\x}| \leq \Big( \sigma |W_t| + \Big((\text{Lip}_{\nabla V} + \text{Lip}_{\nabla F})R  + 3\text{Lip}_{\nabla F} \Big)t \Big) \\
        &\quad\times \exp{(\text{Lip}_{\nabla V} + 2\text{Lip}_{\nabla F})t + \TP{\x}\frac{\TP{\x} + t}{\TP{\x}}\log(\frac{\TP{\x} + t}{\TP{\x}}) - t}.
    \end{aligned}
    \end{equation}
    Thus, it follows that, for the absolute value of the Brownian motion itself,
    \begin{equation*}
        \sigma |W_t| \geq - C_1 t + C_2(t)|\X_t^\sigma - \LP{\x}|,
    \end{equation*}
    where, for simplicity of further derivations, we introduced a positive constant $C_1 := (\text{Lip}_{\nabla V} + \text{Lip}_{\nabla F})R  + 3\text{Lip}_{\nabla F}$ and a function $C_2(t)$ that is one over the exponent that appears in the equation \eqref{eq:|X_t-x|_1}. For our purposes, it is not the form of $C_2$ itself that is important, but rather its following properties: $C_2(t) \geq 0$ for any $t \geq 0$ and $C_2(t) \xrightarrow[t \to 0]{} 1$. 
    
    The infimum of $C_2$ over any time interval $[0, T]$, for $T < 1$, is either equal to $C_2(T)$, if $\text{Lip}_{\nabla V} + 2\text{Lip}_{\nabla F} \geq 1$, or can be bounded by $C_2(T)\exp\{\text{Lip}_{\nabla V} + 2\text{Lip}_{\nabla F}  - 1\}$ otherwise. This observation can be expressed in the following form: 
    \begin{equation*}
        \inf_{t \in [0, T]} C_2(t) \geq \min \{C_2(T);C_2(T)\exp{\text{Lip}_{\nabla V} + 2\text{Lip}_{\nabla F} - 1}\}.
    \end{equation*}
    Moreover, it is easy to check that $\lim_{T \to 0} C_2^2(T)/T = \infty$. 
    
    Taking into account these remarks, we can now use the Schilder theorem \cite[Lemma~5.2.1]{DZ10} that provides the LDP for the path of the Brownian motion. Hence, for some fixed $\zeta$ and $0 < T < 1$: 
    \begin{equation*}
    \begin{aligned}
        \Prob_{\x}(\sup_{t \in [0, T]} |\X^{\sigma}_t - \LP{\x}| \geq \zeta) &\leq \Prob_\x (\sigma \sup_{t \in [0, T]}|W_t| \geq - C_1 t + C_2(t)\zeta) \\
        &\leq 4d \exp{-\frac{(-C_1 t + \inf_{t \in [0, T]}C_2(t)\zeta)^2}{4dT} \cdot \frac{2}{\sigma^2}},
    \end{aligned}
    \end{equation*}
    where $d$ is the dimension of the space. 
    
    Note that for any $\zeta > 0$ the following limit holds: 
    \begin{equation*}
        \frac{\big(-C_1 T + \inf_{t \in [0, T]}C_2(t)\zeta \big)^2}{T} \xrightarrow[T \to 0]{} \infty. 
    \end{equation*}
     Thus, for any given in advance $c$ and $\zeta$, we can choose $T = T(\zeta, c)$ to be small enough, such that the rate $-\big(-C_1 T(\zeta, c) + \inf_{t \in [0, T(\zeta, c)]}C_2(t)\zeta\big)^2/\big(4dT(\zeta, c)\big)$ will be less or equal to $-c$, which completes the proof.

\subsubsection{Proof of Lemma \ref{lm:unif_lower_bound}: Uniform lower bound for the probability of exit from \texorpdfstring{$G$}{G}}

In order to prove the lemma, we need to find for some $\delta > 0$ small enough, for some $T_0 > 0$ and for any $\x \in \mathcal{C}^3_{\varepsilon, \rho}$ a function $\psi^\x \in C([0, T_0]; \R^d)$ such that $\psi^\x(0) = \LP{\x}$, $\inf_{z \in G}|\psi^\x(s) - z| > \delta$ for some $s \leq T_0$ and $I^{\x}_{T_0}(\psi^\x) < H + \eta$. Moreover, this function should posses an empirical measure that is close to $\delta_a$ at any point of time, i.e. $\mathbb{W}_2\left( \frac{1}{t} \int_0^t \delta_{\psi^\x_s}\dd{s}; \delta_a \right) \leq \rho$. Given such a function, by a simple inclusion of events, we can get the following bound for probability of leaving domain $G$ before time $T_0$:
\begin{equation*}
    \Prob_{\x} (\tau_G^\sigma \leq T_0, \gamma > \tau_G^\sigma) \geq \Prob_{\x} (\X^{\sigma} \in \Psi),
\end{equation*}
where $\Psi:= \bigcup_{\x \in \mathcal{C}^3_{\varepsilon, \rho}} \Psi^\x$ and $\Psi^\x := \{\phi \in C([0, T_0]; \R^d): \|\phi - \psi^\x\|_{\infty} < \delta \}$. Of course, we should take $\delta$ small enough such that any $\phi \in \Psi$ has an occupation measure that satisfies $\mathbb{W}_2\left( \frac{1}{t} \int_0^t \delta_{\phi_s}\dd{s}; \delta_a\right)$. Note that $\Psi$ is an open subset of $C([0, T_0]; \R^d)$, as a union of open sets. Therefore, we can use Theorem \ref{th:LDP_gen} and get
\begin{equation*}
    \limsup_{\sigma \to 0} \frac{\sigma^2}{2} \log \Prob_{\x} (\tau_G^\sigma \leq T_0, \gamma > \tau_G^\sigma) \geq -\inf_{\phi \in \Phi} I^{\x}_{T_0}(\phi) \geq - I^{\x}_{T_0}(\psi^\x) > - (H + \eta).
\end{equation*}

The same lower bound obviously holds for infimum $\inf_{\x \in \mathcal{C}^3_{\varepsilon, \rho}} \Prob_{\x} (\tau_G^\sigma \leq T_0, \gamma > \tau_G^\sigma)$, which is what the lemma claims. Thus, we only need to find a $\psi^\x$ with the properties given above.

\textit{Construction of $\psi^\x$}. The function $\psi^\x$ will be represented as a consecutive gluing of four functions (see Figure \ref{fig:psi_x_constr}): $\psi_0^\x \in C([0, 1]; \R^d)$, $\psi_a \in C([0, T_a]; \R^d)$, $\psi_1 \in C([0, T_1]; \R^d)$, and $\psi_2 \in C([0, 1]; \R^d)$, for some $T_a, T_1$ that will be defined below and will determine $T_0$ as $T_0 := T_a + T_1 + 2$.

 \begin{figure}[b]
    \centering
    \begin{tikzpicture}
        \draw (0, 0) node[inner sep=0]{\includegraphics[width=0.465\columnwidth]{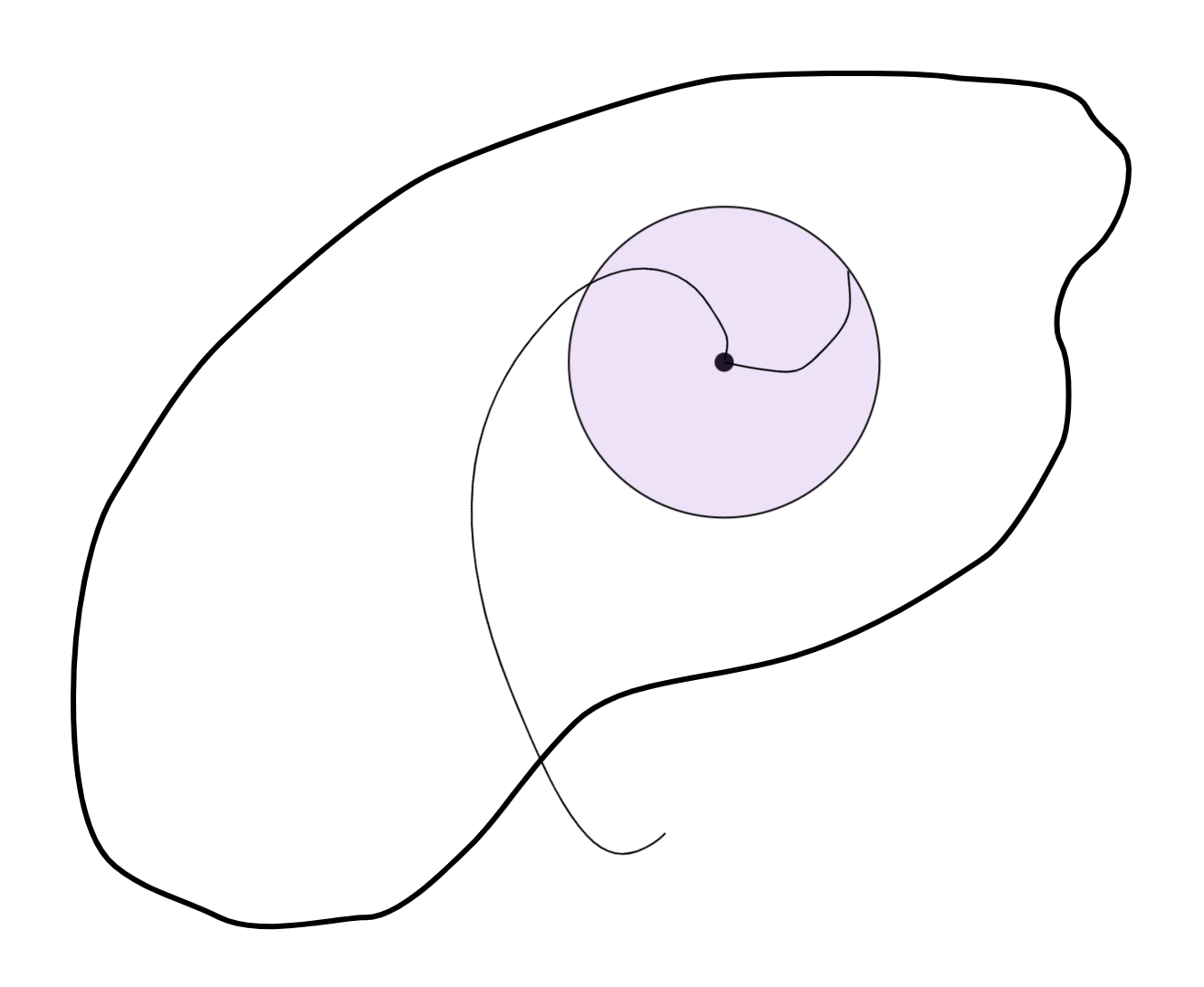}};
        \draw (-1, -1.4) node {$\color{blue} G$};
        \draw (0.4, 0.5)  node {$\color{blue} a$};

        \fill (1.23, 1.16)  circle[radius=1pt];
        \draw (1.4, 1.4) node {$\LP{\x}$};
        \draw (1, 0.45) node {$\psi_0^\x$};

        \draw (-0.9, 0) node {$\psi_1$};
        \fill (-0.31, -1.35)  circle[radius=1pt];
        \draw (0.2, -1.4) node {$\psi_2$};
        
        \draw (-0.1, -2) node {$x_{\text{out}}$};
        \fill (0, -1.78)  circle[radius=1pt];
    \end{tikzpicture}
    \caption{Construction of $\psi^\x$}
    \label{fig:psi_x_constr}
    \end{figure}

1. The first function $\psi_0^\x$ will bring us from point $\LP{\x}$ to point $a$ in at most time $1$ with small rate function. Define
\begin{equation*}
    \psi^\x_0(s) = \begin{cases}
        \LP{\x} + \frac{a - \LP{\x}}{|a - \LP{\x}|} s, \text{ for } 0 \leq s \leq |a - \LP{\x}|, \\
        a, \text{ for } |a - \LP{\x}| \leq s \leq 1.
    \end{cases} 
\end{equation*}
We can establish the following bound for its rate function. Let $\x_0 = (\infty, \delta_a, \LP{\x})$ and $\mu^{\psi_0^\x}_s := \frac{\TP{\x}}{\TP{\x} + s} \OM{\x} + \frac{1}{\TP{\x} + s}\int_0^s \delta_{\psi_0^\x(s)}\dd{s}$. We want to separate the interaction . Then:
\begin{equation*}
    \begin{aligned}
        I_1^\x (\psi_0^\x) &\leq 2I_1^{\x_0} (\psi_0^\x) + \frac{1}{2} \int_0^1\Big|\nabla F*\mu^{\psi_0^\x}_s(\psi_0^\x(s))  - \nabla F(\psi_0^\x(s) - a) \Big|^2 \dd{s}.
    \end{aligned}
\end{equation*}

Note that, since $\psi^\x_0$ never leaves $B_\rho(a)$ and since $\OM{\x} \in \mathbb{B}_{(1 + 2\varepsilon)\rho}(\delta_a)$, then $\mathbb{W}_2(\mu^{\psi_0^\x}_s; \delta_a) \leq (1 + 2\varepsilon)\rho$ for any $0 \leq s \leq 1$. We use Lipschitz continuity of $\nabla F$ (Assumption \ref{A-1}.\ref{A-1.Lip}) and get
\begin{equation*}
    \begin{aligned}
     I_1^\x (\psi_0^\x) &\leq 2I_1^{\x_0} (\psi_0^\x) + \frac{\text{Lip}_{\nabla F}^2}{2} \max_{0 \leq s \leq 1}\mathbb{W}_2^2(\mu^{\psi_0^\x}_s; \mu^{\psi_0^\x}_{1}) \\
     & \leq 2I_1^{\x_0} (\psi_0^\x) + \frac{\text{Lip}_{\nabla F}^2}{2} (1 + 2\varepsilon)^2\rho^2.
    \end{aligned}
\end{equation*}
As for $I^{\x_0}_1(\psi_0^\x)$,
\begin{equation}\label{eq:aux:4I_x0_1}
    \begin{aligned}
        4I_1^{\x_0} &= \int_0^{|a - \LP{\x}|} \left| \frac{a - \LP{\x}}{|a - \LP{\x}|} + \nabla V(\psi_0^\x(s)) \right|^2 \dd{s} \\
        & \leq 2|a - \LP{\x}| + 2\text{Lip}_{\nabla V} \int_0^{|a - \LP{\x}|}|\psi_0^\x(s)|^2\dd{s} \\
        & \leq 2 \rho + 4 \text{Lip}_{\nabla V}(|a|^2 + \rho^2) \rho.
    \end{aligned}
\end{equation}
Therefore, we can choose $\rhobar(\eta) > 0$ small enough such that
\begin{equation*}
    I_1^\x(\psi_0^\x) \leq \frac{\eta}{5}.
\end{equation*}

2. The second function $\psi_a$ is defined to be constant and equal to $a$ for some time $T_a$, which will be increased later if necessary:
\begin{equation*}
    \psi_a \equiv a.
\end{equation*}
We need this segment in order to put enough ``mass'' in $\delta_a$ and balance the later path. Of course, its occupation measure never exceeds $(1 + 2\varepsilon)\rho$. Moreover, if we denote $\mathbf{a} = (\infty, \delta_a, a)$, $\x_a := (\TP{x} + 1, \mu^{\psi_0^\x}_{1}, a)$, and $\mu^{\psi_a}_s := \frac{\TP{\x} + 1}{\TP{\x} + 1 + s} \mu^{\psi^\x_0}_1 + \frac{1}{\TP{\x} + 1 + s} \int_0^s \delta_{\psi_a(s)} \dd{s}$, we can, as before, establish the following bound for the rate function:
\begin{equation*}
    \begin{aligned}
     I_{T_a}^{\x_a} (\psi_a) &\leq 2I_{T_a}^{\mathbf{a}} (\psi_a) + \frac{\text{Lip}_{\nabla F}^2}{2} \max_{0 \leq s \leq T_a}\mathbb{W}_2^2(\mu^{\psi_a}_s; \delta_a) \\
     & \leq \frac{\text{Lip}_{\nabla F}^2}{2} (1 + 2\varepsilon)^2\rho^2 \leq \frac{\eta}{5},
    \end{aligned}
\end{equation*}
for $\rho$ small enough. Note that this bound is independent of the choice of $T_a$.

3. In order to construct the third function $\psi_1$, we first remind that $H = \inf_{z \in \partial G} \{V(z) + F(z - a) - V(a)\}$ is the height of the effective potential $V + F(\cdot - a)$ within domain $G$. By the classical result, $H$ can be also expressed as an infimum
\begin{equation*}
    \inf_{t > 0} \inf_{\phi \in \Phi^a_t}I^{\mathbf{a}}_t (\phi) = H,
\end{equation*}
where $\mathbf{a} := (\infty, \delta_a, a)$ and $\Phi_t^a := \{\phi \in C([0, t]; \R^d): \phi(0) = a, \phi(t) \in \partial G\}$. Hence, there exists $T_1 > 0$ and $\psi_1 \in \Phi^a_{T_1}$ (which among other things implies $\psi_1(T) \in \partial G$) such that $I^{\mathbf{a}}_{T_1} (\psi_1) \leq H + \eta/10$. Define
\begin{equation*}
    \mu^{\psi_1}_s := \frac{\TP{\x} + 1}{\TP{\x} + 1 + T_a + s} \mu^{\psi_0^\x}_1 + \frac{T_a}{\TP{\x} + 1 + T_a + s} \delta_a + \frac{s}{\TP{\x} + 1 + T_a + s} \cdot \frac{1}{s} \int_0^s \delta_{\psi_1(s)} \dd{s}.    
\end{equation*}

So, if we define $R = \sup_{z \in G} |z - a|$, we can express:
\begin{equation*}
    \begin{aligned}
        \mathbb{W}_2(\mu^{\psi_1}_s ; \delta_a) \leq (1 + 2\varepsilon) \rho + \frac{-(1 + 2\varepsilon)\rho T_a + T_1 R}{\TP{\x} + 1 + T_a + s}.
    \end{aligned}
\end{equation*}

And we can increase without loss of generality $T_a$ to be big enough such that $-(1 + 2\varepsilon)\rho T_a + T_1 R < 0$, which guarantees that
\begin{equation*}
    \mathbb{W}_2(\mu^{\psi_1}_s ; \delta_a) \leq (1 + 2\varepsilon)\rho
\end{equation*}
for any $s \leq T_1$.

We can also achieve the following upper bound for the rate function of $\psi_1$ with initial condition $\x_1 = (\TP{\x} + 1 + T_a, \mu_1^{\psi_a}, a)$. For any $c \in (0, 1)$,
\begin{equation*}
\begin{aligned}
    I^{\x_1}_{T_1}(\psi_1) &\leq (1 + c) I^{\mathbf{a}}_{T_1}(\psi_1) + \frac{1 + c}{4c} \int_0^{T_1}\left| \nabla F * \mu^{\psi_1}_s (\psi_1(s)) - \nabla F (\psi_1(s) - a) \right|^2 \dd{s} \\    
    & \leq (1 + c)I^{\mathbf{a}}_{T_1}(\psi_1) + \frac{\text{Lip}_{\nabla F} T_1 (1 + c)}{4 c} \mathbb{W}_2^2(\mu^{\psi_1}_s; \delta_a) \\
    & \leq (1 + c)I^{\mathbf{a}}_{T_1}(\psi_1) + \frac{\text{Lip}_{\nabla F} T_1 (1 + c)}{2 c} (1 + 2\varepsilon)^2 \rho^2 \xrightarrow[\rho \to 0]{} (1 + c)I^{\mathbf{a}}_{T_1}(\psi_1). 
\end{aligned}
\end{equation*}
Since the limit above holds for any $c \in (0, 1)$, we can, without loss of generality, make $\rho$ small enough such that $I^{\x_1}_{T_1}(\psi_1) \leq I^{\mathbf{a}}_{T_1}(\psi_1) + \eta/10$. That leads to
\begin{equation*}
    I^{\x_1}_{T_1}(\psi_1) \leq H + \frac{\eta}{5},
\end{equation*}
for any $\x \in \mathcal{C}^3_{\varepsilon, \rho}$ (we remind that $\x_1$ depends on $\x$).

4. For the last part, we choose a point $x_{\text{out}} \notin G$ such that the line segment starting from point $\psi_1(T_1)$ and finishing at $x_{\text{out}}$ is included in $\R^d \setminus G$ (this point exists since $\partial G$ is smooth by Assumption \ref{A-2}.\ref{A-2.G_open}). Define 
\begin{equation*}
    \psi_2(s) := 
        \psi_1(T_1) + \frac{\psi_1(T_1) - x_{\text{out}}}{|\psi_1(T_1) - x_{\text{out}}|} s,
\end{equation*}
for $0 \leq s \leq \delta := |\psi_1(T_1) - x_{\text{out}}|$. After that, let this function follow a path that does not contribute to the rate function for the rest of the time $\delta < s \leq 1$. Precisely, if we define the cumulative empirical measure of $\psi_2$ as  
\begin{equation*}
    \mu^{\psi_2}_s := \frac{\TP{\x} + 1 + T_a + T_1}{\TP{\x} + 1 + T_a + T_1 + s} \mu^{\psi_1}_{T_1} + \frac{1}{\TP{\x} + 1 + T_a + T_1 + s} \int_0^s \delta_{\psi_2(s)} \dd{s},
\end{equation*}
we can also construct a deterministic process $\X^{\x_3, 0}_s$, defined by \eqref{eq:SID_main_sys_gen} with $\sigma=0$, using the initial conditions $\x_3 := (T_3, \mu_3, x_3)$, where $T_3 = \LP{\x} + 1 + T_a + T_1 + \delta$, $\mu_3 = \mu^{\psi_2}_{\delta}$ and $x_3 = x_{\text{out}}$. Lemma \ref{lm:X0_conv} provides the upper bound for the occupation measure of $\X^{\x_3, 0}$: $\mathbb{W}_2(\bmu^{\x_3, 0}_s; \delta_a) \leq (1 + 2\varepsilon)\rho$, for some $\varepsilon > 0$ that is defined in the lemma. Hence, similarly to previous computations, we can control the cumulative empirical measure $\mu^{\psi_2}_s$. We can increase $T_a$ if necessary such that
\begin{equation*}
    \mathbb{W}_2(\mu^{\psi_2}_s; \delta_a) \leq (1 + 2\varepsilon)\rho
\end{equation*}
for any $s \leq T_1$

Define $\x_2 := (\LP{\x} + 1 + T_a + T_1, \mu^{\psi_1}_{T_1}, \psi_1(T_1))$. We remind that $\psi_2$ was constructed in a way that the second part of it does not contribute to the rate function $I^{\x_2}_1(\psi_2)$. For the first part the computations are similar to those of \eqref{eq:aux:4I_x0_1}. As a result, without loss of generality, we can choose $\delta > 0$ to be small enough ($x_{\text{out}}$ closer to $\partial G$) and get:
\begin{equation*}
    I^{\x_2}_{1}(\psi_2) \leq \frac{\eta}{5}.
\end{equation*}

5. Finally, by letting 
\begin{equation}
    \psi^\x = \begin{cases}
        \psi_0^\x(s) , \text{ for } 0 \leq s \leq 1, \\
        \psi_a(s - 1), \text{ for } 1 < s \leq 1 + T_a, \\
        \psi_1(s - 1 - T_a), \text{ for } 1 + T_a < s \leq 1 + T_a + T_1, \\
        \psi_2(s - 1 - T_a - T_1), \text{ for } 1 + T_a + T_1 < s \leq T_a + T_1 +2;
    \end{cases}
\end{equation}
we can observe that $I_{T_0}^{\x} (\psi^\x) = I_{1}^{\x} (\psi^\x_0) + I_{T_a}^{\x_a} (\psi_a) + I_{T_1}^{\x_1} (\psi_1)  + I_{1}^{\x_2} (\psi_2) \leq H + \frac{4 \eta}{5} < H + \eta$ and $\sup_{s \in [0, T_0]}\inf_{z \in G}|\psi^\x(s) - z| > \delta$. Moreover, by construction of all the function, we have ensured that $\frac{1}{t}\int_0^t \delta_{\psi^\x(s)} \dd{s} \in \mathbb{B}_{(1 + 2\varepsilon) \rho}(\delta_a)$ for any $t \leq T_0$, which is what is needed.

\section{Generalization}\label{s:Generalization}
In this section we present the possible generalisation of Assumptions \ref{A-1} and \ref{A-2} as well as the exit-time result for this more general process. In this paper we did not talk in details about the problem of existence and uniqueness of the self-interacting diffusion. While in the Lipschitz case (Assumption \ref{A-1}) these questions are standard, locally Lipschitzness of the potentials add some complication. In \cite{ADMKT22}, the authors provide a proof for existence and uniqueness of self-interacting diffusion under the following assumptions that we take as a baseline:
\begin{custom_assumption}{A-1$^\prime$}\label{A-1_prime}
    \hfill
    \begin{enumerate}
        \item \label{A-1_prime.Cont} (regularity) Potentials $V$ and $F$ belong to the space $C^2(\R^d; \R)$.
        \item \label{A-1_prime.growth} (growth) There exist a constant $C_{\text{gr}}$ and an order $n > 0$ such that for any $x, y \in \R^d$, $|\nabla V(x) - \nabla V(y)| \leq C_{\text{gr}}|x - y|(1 + |x|^{2n} + |y|^{2n})$ as well as $|\nabla F(x) - \nabla F(y)|\leq C_{\text{gr}}|x - y|(1 + |x|^{2n} + |y|^{2n})$.
        \item \label{A-1_prime.conv_inf} (confinement at infinity) $\underset{\vert x\vert\rightarrow \infty}{\lim} V(x) =+ \infty$,  $\underset{\vert x\vert\rightarrow \infty}{\lim} \frac{\vert\nabla V(x)\vert^2}{V(x)} = +\infty$ and there exists $\alpha > 0$ such that $\Delta V(x)\leq \alpha V(x)$.
    \end{enumerate}
\end{custom_assumption}

It is worth noting that Assumption \ref{A-1_prime}.\ref{A-1_prime.growth} is a combination of locally Lipschitz and polynomial growth conditions. Assumption \ref{A-1_prime}.\ref{A-1_prime.conv_inf} has been introduced in \cite{ADMKT22} to control the Lyapunov functional of the diffusion process to ensure that there is no explosion within a finite time. It is important to mention that, if required, this assumption can be replaced by another one that ensures the process's global existence and uniqueness without any impact on the exit-time result. 

We can also relax Assumptions \ref{A-2}. Indeed, the domain $G$ does not necessarily have to be bounded, but it is necessary to add another assumption on the level sets instead. Consider:

\begin{custom_assumption}{A-2$^\prime$}\label{A-2_prime}
    \hfill
    \begin{enumerate}
        \item \label{A-2_prime.G_open} (domain $G$) $G \subset \R^d$ is an open connected set such that $\partial G = \partial \overline{G}$. The boundary $\partial G$ is a smooth $(d-1)$-dimensional hypersurface.
        \item \label{A-2_prime.bounded_level_set} (bounded sublevel set) Define $L^{-}_{H}:=\{x \in U_a(x) \leq H\}$ the sublevel set of the height $H:= \inf_{x \in \partial G} \{U_a(x) - U_a(a)\}$. Then $L^{-}_{H} \cap G$ is a bounded connected set.
        \item \label{A-2_prime.Effect_pot_conv} (stability of $\overline{G}$ under the effective potential) Let $\phi$ be defined as $\phi_t^x = x - \int_0^t \nabla U_a(\phi_s^x) \dd{s}$. For any $x \in \overline{G}$, $\{\phi^x_t\}_{t > 0} \subset G$ and $\phi^x_t \xrightarrow[t\to \infty]{} a$. Moreover, $\nabla U_a(a) = 0$.
        \item \label{A-2_prime.Strong_attraction_a} (strong attraction around $a$) There exist $\Delta_\mu, \Delta_x > 0$ small enough and constants $0 < K_2 < K_1<\infty$, such that for any $\mu$ satisfying $\mathbb{W}_2(\mu;\delta_a) \leq \Delta_\mu$ and for any $x \in B_{\Delta_x}(a)$, we have $$\langle \nabla V(x) + \nabla F*\mu(x); x - a \rangle \geq K_1|x - a|^2-K_2|x-a|\mathbb{W}_2\left(\mu;\delta_a\right).$$
    \end{enumerate}
\end{custom_assumption}

Under these assumptions, we can obtain the Kramers' type law and the exit-location result. Consider
\begin{corollary}
    Let Assumptions \ref{A-1_prime} and \ref{A-2_prime} be fulfilled. Let the process $X^\sigma$ be the unique strong solution of the system \eqref{eq:SID_main_sys}. Let $\tau_G^\sigma := \inf \{t: X_t^\sigma \notin G\}$ denote the first time when $X^\sigma$ exits the domain $G$. Let $H:= \inf_{x \in \partial G} \{U_a(x) - U_a(a)\}$ be the height of the effective potential. Then, the following two results hold:
    \begin{enumerate}
        \item Kramers' type law: for any $\delta > 0$
    \begin{equation}
        \lim_{\sigma \to 0}\Prob_{x_0}\left( e^{\frac{2 (H - \delta)}{\sigma^2}} < \tau_G^\sigma < e^{\frac{2 (H + \delta)}{\sigma^2}}\right) = 1;
    \end{equation}
    \item Exit-location: for any closed set $N \subset \partial G$ such that $\inf_{z \in N} \{U_a(x) - U_a(a)\} > H$ the following limit holds:
    \begin{equation}
        \lim_{\sigma \to 0} \Prob_{x_0} (X^\sigma_{\tau_G^\sigma} \in N) = 0.
    \end{equation}

    \end{enumerate}
\end{corollary}

We do not prove this corollary rigorously but rather give an idea of the proof. The main insight comes from the proof of the exit-location result---Section~\ref{s:exit_location}. Namely, after the stabilization around $\delta_a$, the exit from the domain $G$ happens around the level set that touches its boundary $L_H$ (see page \pageref{s:exit_location}). Despite $G$ is not bounded, we assumed $L^{-}_{H} \cap G$ to be bounded. That means that we can introduce a diffusion
\begin{equation*}
    \begin{cases}
    \dd{Y^\sigma_t} \!\!\!\!&= - \nabla \overline{V}(Y^\sigma_t)\dd{t}  - \frac{1}{t} \int_0^t\nabla \overline{F}(Y^\sigma_t - Y^\sigma_s) \dd{s} \dd{t}  + \sigma \dd{W_t}, \\
    Y^\sigma_0 &= x_0 \in \R^d \; \text{a.s.}
    \end{cases}
\end{equation*}
driven by the same Brownian motion as \eqref{eq:SID_main_sys}, where $\overline{V}$ and $\overline{F}$ are modifications of $V$ and $F$ respectively defined in the following way. Let $B_{\text{mod}} \subset \R^d$ be a big enough closed ball that contains the sets $\{X^0_t\}_{t \geq 0}$ and $L_H^-$. First, we let $\overline{V}(x) = V(x)$ and $\overline{F}(x) = F(x)$ for any $x \in B_{\text{mod}}$, Then, we extend $\overline{V}$ and $\overline{F}$ on other points of $\R^d$ such that in the end the potentials $\overline{V}$ and $\overline{F}$ are Lipschitz continuous and belong to $C^2(\R^d; \R)$. 

It means that $\overline{V}$, $\overline{F}$ and $B_{\text{mod}} \cap G$ satisfy Assumptions \ref{A-1} and \ref{A-2} and we can establish the Kramers' type law for $\overline{\tau}^\sigma := \inf\{t \geq 0: Y^\sigma_t \notin B_{\text{mod}} \cap G \}$ and exit-location result for $Y^\sigma$. Note also that for any $t \geq 0$ we have:
\begin{equation*}
    X_{t \wedge \overline{\tau}^\sigma}^\sigma = Y_{t \wedge \overline{\tau}^\sigma}^\sigma \text{ a.s.}
\end{equation*}

The last observation that one needs to do in order to prove the corollary is the following. By the exit-location result, the probability that $Y^\sigma$ exits the domain $B_{\text{mod}} \cap G$ around $\partial G \cap L_H$ tends to one as $\sigma$ tends to zero. It means that the probability that $\tau_G^\sigma = \overline{\tau}^\sigma$ should also tend to one as $\sigma$ tends to zero, since leaving the domain $B_{\text{mod}} \cap G$ on the boundary $\partial G$ also means leaving $G$ itself.
\bigskip

{\bf Acknowledgments: }\emph{This work has been supported by the French ANR grant METANOLIN (ANR-19-CE40-0009).}

\end{document}